\documentclass[11pt,oneside,reqno]{amsart}
 \usepackage{amsmath}
\usepackage{amssymb}
\usepackage{color}
\usepackage{cite}
\usepackage{enumerate}
\usepackage{fancyhdr}
\usepackage{graphicx}
\usepackage{bm}
\usepackage{soul,xcolor}

\newcommand{\supp}{{\mathrm{supp}}}
\usepackage[margin=3.5cm]{geometry}

\usepackage{subfigure}
\usepackage{algorithm}
\usepackage{algorithmic}
\usepackage{float}
\usepackage{lipsum}


\numberwithin{equation}{section}

\def\span{\textrm{ span }}

\def\R{ {\mathbb R} }

\def\Z{ {\mathbb Z} }

\def\supp{\textrm{supp }}

\newtheorem{theorem}{Theorem}[section]
\newtheorem{lemma}[theorem]{Lemma}
\newtheorem{proposition}[theorem]{Proposition}

\newtheorem{corollary}[theorem]{Corollary}

\newtheorem{remark}[theorem]{Remark}

\newtheorem{example}[theorem]{Example}
\newtheorem{question}[theorem]{Question}

\newtheorem{assumption}[theorem]{Assumption}

\def\G{ {\mathcal G} }
\def\R{ {\mathbb R} }
\def\Z{ {\mathbb Z} }

\def\U{ {\bf U} }

\let\oldv\v
\def\v{{\bf v}}

\def\x{{\bf x}}

\def\span{{\rm{span}}}
\def\diag{{\rm{diag}}}
\numberwithin{equation}{section}

\title{Shift-invariant spaces on finite undirected graphs}
\author{ Yang Chen, Seok-Young Chung and  Qiyu Sun}
\address{Chen: Key Laboratory of Computing and Stochastic Mathematics (Ministry of Education), School of Mathematics and Statistics,
 Hunan Normal University, Changsha, Hunan 410081, P. R. China}
 \email{ychenmath@hunnu.edu.cn}

 \address{Chung: Department of Mathematics, Michigan State University, East Lansing, MI 48824, USA}    %
\email{sychung@msu.edu}

\address{Sun: School of Data, Mathematical, and Statistical Sciences, University of Central Florida, Orlando, FL 32816, USA}
\email{qiyu.sun@ucf.edu}

\begin{document}
	\date{}
	\pagestyle{plain}
	
	\date{}
\begin{abstract}
Shift-invariant spaces (SISs) on the real line provide a natural framework for representing,  analyzing and processing signals with inherent shift-invariant structure. 
In this paper, we extend this framework to the finite undirected graph setting by introducing the concept of graph shift-invariant spaces (GSISs). We examine several properties of GSISs, including their characterization via range functions and fiber functions in the Fourier domain, their connections to shift-invariant filters and polynomial filters, the frame and Riesz basis structures of finitely generated GSISs, and their intricate relationships with bandlimited spaces, finitely generated GSISs, and  graph reproducing kernel Hilbert spaces  with shift-invariant reproducing kernels (SIGRKHSs). Our analysis reveals several distinctions between SISs on the line and GSISs, such as 
the shift-invariance of the frame operator, the existence of shift-invariant dual frames, the emergence of fractional shift-invariance, and the interrelationships among GSISs, finitely generated GSISs, SIGRKHSs and bandlimited spaces.

In this paper, we also introduce a spectral decomposition of the identity associated with graph shifts and  propose 
a novel definition of the graph Fourier transform (GFT) of spectral type,  together with explicit formulations for the GFTs on complete graphs and circulant graphs. In addition, we establish a clear connection between polynomial filters and shift-invariant filters, and we derive  a graph uncertainty principle 
 governing the essential supports of a
nonzero graph signal and its GFT.

\end{abstract}

\maketitle
%
%

\section{Introduction}

A Shift-Invariant Space (SIS) $U$ on the real line is a linear space of real functions being invariant under integer shifts, i.e., $f(\cdot-k)\in U$ for all $f\in U$ and $k\in {\mathbb Z}$. 
Shift-invariant spaces provide a natural framework for capturing features preserved under integer shifts and for analyzing signals with inherent shift-invariant structure. This property makes them particularly powerful across a broad range of mathematical and engineering disciplines, including wavelet analysis,   Gabor analysis, approximation theory,  frame theory, and sampling theory. In these fields,  shift invariance is crucial 
 for  constructing well-structured bases and frames, and enabling efficient signal representation and fast data processing \cite{akram2001b, akram2001,  akram2005, bownik2000, bownik2003, chui2006, Daubechies1992, deBoor1994, Grochenig2001, Mallat1998, ron1997, ron1995,  Unser2000}.

Graph Signal Processing (GSP) extends classical signal processing concepts to represent, process, analyze,
and visualize 
signals and  
datasets defined on networks or irregular domains
\cite{dong2020, Isufi2024, Ortega2022, ortega2018, sandryhaila2013,  sandryhaila2014, shuman13, Stankovic2019, yan2024}.
Similar to the unit shift on the real line and the one-sample delay in classical signal processing, the concept of graph shifts has been introduced in GSP \cite{emirov2022, Ortega2022, sandryhaila2013, Stankovic2019}. These graph shifts  are usually designed to capture the structural characteristics of the underlying graph, with common choices including the adjacency matrix, the Laplacian matrix, and their variants. 
Serving as fundamental building blocks, graph shifts provide the foundation from which most GSP tools and techniques are derived, including  the Graph Fourier Transform (GFT) and  shift-invariant filtering procedure.
In analogy with the unit directional shifts in high-dimensional Euclidean spaces and unit-delay in multi-dimensional signal processing, we select real-valued, symmetric and commutative graph shifts 
${\bf S}_1, \ldots, {\bf S}_d$  
as the basic structural  blocks of this paper; see Assumption \ref{graphshift.assumption}.

Similar to the shift-invariant space on the real line,  a Graph Shift-Invariant Space (GSIS) $U$ on  a finite undirected graph ${\mathcal G}$ of order $N\ge 2$ is defined as
the linear space of graph signals that remains invariant under the graph shifts, i.e., 
\begin{equation}\label{sis.def}
{\bf S}_l{\bf x}\in U \ \  {\rm for \ all}  \ 1\le l\le d \
\ {\rm  and }\  \ {\bf x}\in U.
\end{equation} 
 Numerical simulations in \cite{Chung2024} suggest that GSIS may offer greater flexibility and practical interpretability for modeling graph signals and datasets that exhibit both shift-invariance structure and strong localization patterns. Such scenarios include transportation delays and diffusion phenomena in networked systems, where GSIS demonstrates advantages over widely used approaches such as the bandlimited model and the Graph Reproducing Kernel Hilbert space (GRKHS) model with shift-invariant reproducing kernels.

The concept of GSIS was introduced in \cite{Chung2024}, where its bandlimiting, reproducing kernel, and sampling properties were investigated.  Let ${\pmb \Lambda}=\{{\pmb \lambda}(n)\ | \ 1\le n\le N\}\subset {\mathbb R}^d$ be the joint spectrum of the graph shifts
${\bf S}_1, \ldots, {\bf S}_d$; see  \eqref{jointspectrum.def}.
Under the  distinct 
joint spectrum  assumption for the graph shifts, 
i.e., 
\begin{equation}\label{distincteigenvalue.assumption}
    {\pmb \lambda}(n)\ne {\pmb \lambda(n')} \  {\rm for \ all}\ 1\le n\ne n'\le N,
\end{equation}
it is shown in \cite{Chung2024} that
the terminologies of bandlimitedness, shift-invariance, and principal shift-invariance for a linear space of graph signals are essentially equivalent in the undirected finite graph setting.  We remark that these spaces characterize graph signals from different perspectives: bandlimitedness describes them in the Fourier domain, shift-invariance emphasizes their invariance in the spatial domain, and principal shift-invariance highlights their spatial–frequency localized representation.

A fundamental limitation of the existing GSIS theory is its heavy reliance on the distinct joint spectrum assumption \eqref{distincteigenvalue.assumption}. However, structure‑preserving graph shifts, such as the Laplacians on circulant and complete graphs, often exhibit eigenvalues with high multiplicity and therefore fail to satisfy this distinct joint spectrum condition; see \cite{Chung1997, Merris1994, Simic2015} and also Example \ref{uncertainty.example} and  Appendix \ref{circulantgraph.section} of this paper. This, in turn, introduces ambiguity into Fourier-domain signal representations, thereby compromising the performance of fundamental GSP tasks \cite{ Ortega2022, ortega2018,  Stankovic2019}.
 These challenges motivate us to study the GFT and GSISs without imposing the distinct joint spectrum constraint \eqref{distincteigenvalue.assumption} on the graph shifts ${\bf S}_1, \ldots, {\bf S}_d$. In this paper, we introduce a novel definition of the GFT and investigate bandlimiting, frame, and reproducing kernel properties of GSISs. 
 We  observe that removing the distinct joint spectrum constraint \eqref{distincteigenvalue.assumption} leads to several significant differences in the resulting GSIS theory.

\smallskip

The main contributions of this paper are summarized as follows:

\begin{itemize}
\item[{(i)}] We introduce a spectral decomposition  of the identity 
associated with the graph shifts ${\bf S}_1, \ldots, {\bf S}_d$; see \eqref{projection.eq0}, \eqref{projection.eq1}, \eqref{projection.eq2} and \eqref{projection.eq3}. 
Based on this spectral decomposition, we define a new GFT {\bf independently} of the eigendecomposition of the graph shifts, under which regular signals exhibit energy mainly concentrated in the low-frequency components of the Fourier domain; see \eqref{projection.eq4}, \eqref {projection.eq5},  \eqref{GFT.def} and Theorem \ref{bandapproximation.thm}.
 In addition, we characterize the sets of polynomial filters and shift-invariant filters, and show that the commutative semi-group of polynomial filters constitutes the {\bf center} of the noncommutative semi-group of shift-invariant filters; see Theorems \ref{polynomialfilter.thm}, \ref{polynomialinvariance.thm}, \ref{shiftinvariantfilter.thm}, and \ref{center.thm}.

\item[{(ii)}] We introduce the graph range function and the graph dimension function for a GSIS, and characterize a GSIS in the Fourier domain via its graph range function; see \eqref{rangefunction.def-1}, \eqref{rangefunction.def0}, \eqref{dimension.def} and Theorem \ref{rangecharacterization.thm}. This characterization parallels to the classical range function approach for a SIS on the real line \cite{bownik2000, bownik2003, deBoor1994}. Consequently, every GSIS can be described both as the range space of some shift-invariant filter and as the kernel space of another shift-invariant filter; see Theorem \ref{filterrange.thm}.
We remark that a GSIS is always pseudo-inverse shift-invariant, and it is also {\bf fractional} shift-invariant provided that all graph shifts are positive semidefinite; see Remarks \ref{pseudoinversesis.rem} and \ref{fractional.rem}. In contrast, in the  real line setting, inverse shift-invariance always holds whereas fractional shift-invariance typically fails \cite{akram2010, akram2011, cabrelli2016, fuhr2014, hardin2018}.

\item[{(iii)}] We show that a linear space of graph signals is bandlimited if and only if it is super–shift-invariant; see Theorem \ref{bandlimited.thm}. Consequently, every bandlimited space can be realized as the range space of some polynomial filter; see Corollary \ref{bandlimitedpolynomialfilter.cor} and cf. Theorem \ref{filterrange.thm}. This result recovers the conclusion in \cite{Chung2024} that a linear space of graph signals is bandlimited if and only if it is shift-invariant, provided that the distinct joint spectrum condition \eqref{distincteigenvalue.assumption} is satisfied, since in that scenario every shift-invariant filter reduces to a polynomial filter \cite{emirov2022, sandryhaila2013}. We remark that while every bandlimited space is a GSIS, the converse does {\bf not} necessarily hold without imposing the constraint \eqref{distincteigenvalue.assumption} on the joint spectrum of the graph shifts.

\item[{(iv)}] If 
the graph shifts ${\bf S}_1, \ldots, {\bf S}_d$ are invertible (as is the case for the unit shift operator on the line), then one may verify that
\begin{equation}\label{maximalsis.eq0}
{\bf S}_lU=U, \quad 1\le l\le d,
\end{equation}
hold for every GSIS $U$.
In Theorem \ref{sisequality.thm}, we 
use the dimension function of a GSIS $U$ to determine whether  a GSIS $U$ satisfies
\eqref{maximalsis.eq0} without requiring the invertibility of the graph shifts. Furthermore, 
in Theorem \ref{maximalsis.thm},    we show that ${\bf S}_1\ldots {\bf S}_dU$  is  the {\bf maximal} shift-invariant subspace  $W$ of a GSIS $U$ such that the space $W$ and its shifted versions ${\bf S}_l W, 1\le l\le d$, are the same, i.e., 
 \begin{equation}\label{maximalsis.eq1}
{\bf S}_l W=W,\quad  1\le l\le d.
\end{equation}

\item [{(v)}] We introduce the fiber function for a finitely generated GSIS and use it to characterize a finitely-generated GSIS; see \eqref{fiber.def} and Theorem \ref{fiber.section}, and cf. the analogous fiber characterizations of SISs in the  real line setting \cite{bownik2000, bownik2003, deBoor1994}. We show that every GSIS {\bf is} finitely generated, and that the length of a finitely generated GSIS coincides with the maximal bound of its dimension function; see Theorem \ref{FGSIS.thm} and Corollaries \ref{length.cor} and \ref{wholespacesis.cor}. Consequently, we recover the conclusion in \cite{Chung2024} that every GSIS is generated by a single graph signal,  where the distinct joint spectrum condition \eqref{distincteigenvalue.assumption} is imposed. 

\item[{(vi)}] For a  finitely-generated GSIS ${\mathcal S}(\Phi)$, we  provide a frame bound estimate for the generating system $\big\{{\mathbf S}_1^{\alpha_1}\cdots {\bf S}_d^{\alpha_d} \phi\ |\ \alpha_1, \ldots, \alpha_d\in {\mathbb Z}_+,\ \phi\in \Phi\big\}$;
 see Theorem \ref{frame.thm}. Unlike in the real line setting where the frame bound is governed solely by the Grammian of the fiber associated with the generator, the frame bound in Theorem \ref{frame.thm} is expressed as the product of {\bf two} quantities: one determined by the Grammian of the graph fiber, and the other 
 depending on  the distribution of distinct eigenvalues of the graph shifts.
 We further provide a characterization of the shift-invariant dual frame for a finitely generated GSIS, together with a necessary and sufficient condition for the existence of such a dual frame; see Theorems \ref{dualframe.thm} and \ref{dualframeexistence.thm}.   Remarkably,
 unlike in the  real line setting where a shift-invariant dual frame {\bf always} exists \cite{bownik2000, bownik2003,  christensen2004,  deBoor1994, hammod2011}, in the graph setting this is {\bf not} necessarily the case; see Corollary \ref{dualframe.PGSIS.cor} and   Example \ref{uncertainty.example.part2} for a detailed analysis of the existence of shift-invariant dual frames
 on a complete graph.
 A possible reason could be that, unlike in the real line setting, the frame operator $S$ in \eqref{frameoperator.def}
 does {\bf not} commute with the graph shifts ${\bf S}_1, \ldots, {\bf S}_d$ in general, see Theorem \ref{frameoperator.thm},  Corollary \ref{frameoperator.cor} and Proposition \ref{dualframeCompletegraph.pr}.
 In Section 
 \ref{FGSIS.section}, we also explore the Riesz property for the finitely-generated  GSIS; see Theorem \ref{riesz.thm}.

 \item[{(vii)}] For a linear space of graph signals equipped with the standard inner product, we show that it is shift-invariant {\bf if and only if} it is  a GRKHS whose reproducing kernel is a shift-invariant filter, and that it is bandlimited if and only if it is a GRKHS whose reproducing kernel is a polynomial filter; see Theorems \ref{rkhs.thm1} and \ref{rkhs.bandlimited.thm}. Consequently, we recover the conclusion in \cite{Chung2024} that every GSIS is a GRKHS with a shift-invariant reproducing kernel, provided that the distinct joint spectrum assumption \eqref{distincteigenvalue.assumption} holds. In addition, we provide a GFT-based characterization of the inner product of a GRKHS with a shift-invariant reproducing kernel, as well as a characterization of the isometric property of the shift-invariant linear operator between two GRKHSs with shift-invariant reproducing kernels; see Theorems \ref{decomposition.thm}, \ref{decomposition.polynomial.thm}, and \ref{rkhsisometric.thm}.
 
\end{itemize}

\smallskip

The paper is organized as follows.
In Section \ref{gft.subsection}, we propose a spectral decomposition of the identity associated with the graph shifts ${\bf S}_1, \ldots, {\bf S}_d$, and introduce a new definition of the GFT of spectral  type, which is independent of the choice of orthonormal basis in the eigen-decomposition of the graph shifts. In Section \ref{polynomialF.subsection}, we examine polynomial filters of the graph shifts and investigate their multiplier property in the Fourier domain, along with the spectral decomposition and invariance induced by the eigendecomposition of the graph shifts. In Section \ref{shiftinvariantF.subsection}, we present the spectral characterization of shift-invariant filters and identify the center of the noncommutative semigroup of shift-invariant filters.
In Section \ref{rangefunction.section}, we introduce the range function to characterize GSISs in the Fourier domain. In Section \ref{dimension.subsection}, we examine the dimension function properties of the sum of two GSISs, as well as the orthogonal complement of a GSIS. In Section \ref{filterrange.section}, we establish a connection between GSISs and shift-invariant filters and provide a spectral decomposition
for shift-invariant linear operators between two GSISs. In Section \ref{bandlimited.section}, we define the bandlimited space and demonstrate that it is super-shift-invariant, and therefore qualifies as a GSIS.  In Section \ref{maximalsis.section}, we examine whether a GSIS is the same as its shifted versions, and we further identify the maximal shift-invariant subspace of a GSIS that coincides with its shifted versions. In Section \ref{fiber.section}, we introduce the fiber function associated with  generators and use it to characterize finitely generated GSISs, and to show that every GSIS is finitely generated,  with the number of generators equal to the maximal bound on the dimension function of the GSIS.
In Section \ref{frame.section}, we consider the frame property of the shift-invariant system associated with a finitely generated GSIS, provide  characterizations for the shift-invariance of the frame operator and for the shift-invariant dual frame, and find a  necessary and sufficient condition for the existence of such a shift-invariant dual frame.
In Section \ref{riesz.section}, we provide a Fourier-domain characterization of when a finite system,
consisting of finitely many shifts of certain generators,  forms a Riesz basis for a finitely generated GSIS. In Section \ref{rkhs.section}, we consider GRKHSs with shift-invariant reproducing kernels, show that every GSIS is essentially a GRKHS of this type, and  provide a characterization for the inner product  of a GRKHS with shift-invariant kernels
in the Fourier domain. In Appendix \ref{circulantgraph.section}, we provide an explicit formulation of the  GFT on a  circulant graph. 
In Appendix \ref{UP.subsection}, we introduce a quantity associated with a  pair of nonempty sets in the spatial–Fourier domain, use it to derive a sufficient condition for the pair to form a strong annihilating pair, and thereby establish uncertainty principles governing the essential supports of a nonzero graph signal and its GFT.

\section{Spectral decomposition of the identity, graph Fourier transforms, polynomial filters and shift-invariant filters}
\label{preliminaries.section}

In this paper, we assume that the underlying graph ${\mathcal G}=(V, E, {\bf W})$
is 
    undirected, simple and finite, where we denote
    its set of vertices by $V$, its set of edges by $E\subset V\times V$, and its weight on edges by ${\bf W}=[w(i,j)]_{i,j\in V}$, and its order  by $N\ge 2$.

 A {\em graph shift}\ ${\bf S}$ on the graph ${\mathcal G}$,
represented by a matrix $\mathbf{S} =  [s(i,j)]_{i,j \in V}$, is defined such that 
$s(i,j) = 0$ unless $i=j$  or $(i, j)\in E$.
 The concept of a graph shift in GSP  plays a role analogous to the one-sample delay in classical signal processing and to the unit shift on the real line in the SIS theory.
 Graph shifts are intentionally designed to exhibit specific properties and carry a physical interpretation. They serve as the foundational building blocks from which many tools and techniques in GSP are derived to perform essential tasks such as denoising, compression, and feature extraction
\cite{emirov2022, gavili2017, Ortega2022, sandryhaila2013}. 
 Illustrative examples are
 the weighted adjacency matrix ${\mathbf W}$, 
   the Laplacian matrix ${\mathbf L}= \mathbf{D} -{ \mathbf W}$,
     the symmetrically normalized Laplacian ${\mathbf L}^{\rm sym} := \mathbf{D}^{-1/2}\mathbf{L}\mathbf{D}^{-1/2}$
  and their variants, 
   where ${\bf D}={\rm diag}[d(i)]_{i\in V}$ is the degree matrix with diagonal entries $d(i)=\sum_{j\in V} w(i,j), i\in V$.     Similar to the unit directional shifts in the Euclidean
space  ${\mathbb R}^d$, in this paper we impose the following  assumption regarding the choice of graph shifts
 ${\bf S}_1, \ldots, {\bf S}_d$ on the graph ${\mathcal G}$:

   \begin{assumption}\label{graphshift.assumption}
Graph shifts
${\bf S}_1, \ldots, {\bf S}_d$ are  real-valued, symmetric and commutative, i.e., 
\begin{equation} {\bf S}_l\in {\mathbb R}^{V\times V},\ \  {\bf S}_l^T={\bf S}_l\ \ {\rm and}\ \
{\bf S}_l {\bf S}_{l'}={\bf S}_{l'}{\bf S}_l \ 
\ {\rm for \ all}  \ 1\le l, l'\le d.
\end{equation}
\end{assumption}

\smallskip

The Graph Fourier Transform (GFT) is a fundamental tool in GSP 
to decompose graph signals into distinct frequency components and  to effectively represent them through diverse modes of variation \cite{Chen2023, Cheng2023, Chung1997, Chung2023, Isufi2024, Ortega2022, Stankovic2019,   Ricaud2019, shuman13}.
 A conventional definition of the GFT relies on the orthonormal basis obtained from the eigendecomposition \eqref{eigen.dec} of the graph shifts ${\bf S}_1, \ldots, {\bf S}_d$; 
 see \eqref{GFT.distinct}. However, the orthonormal basis in \eqref{eigen.dec} is not unique, which raises a natural question: can we define a GFT that is independent of the choice of orthonormal basis in the eigendecomposition?  In Section \ref{gft.subsection}, we introduce a spectral decomposition of the identity associated with the graph shifts  ${\bf S}_1, \ldots, {\bf S}_d$ to address this issue; see \eqref{projection.eq0}-\eqref{projection.eq3}. The new definition of the GFT, presented in \eqref{GFT.def}, is of spectral type and remains independent of the specific orthonormal basis chosen in \eqref{eigen.dec}, cf. the graphon Fourier transform \cite{Ghandehari22}. Under this formulation of GFT, the Parseval identity holds and regular signals exhibit energy concentration primarily in the low-frequency components of the Fourier domain; see \eqref{parsevelidentity} and Theorem \ref{bandapproximation.thm}.

\smallskip

Polynomial filters of  graph shifts have been extensively used in GSP \cite{coutino2019, emirov2022, hammod2011, Isufi2024, isufi2017,  Ortega2022,   sandryhaila2013,  sandryhaila2014,  segarra2017,  shuman13, shuman2018,   waheed2018}.  A polynomial filtering procedure operates as a shift-and-sum mechanism applied to the input signal in the spatial domain. It can be implemented in a distributed manner, where each vertex of the graph is equipped with a data processing unit with limited storage and computational capabilities, along with a communication module for data exchange exclusively with neighboring agents. In Section \ref{polynomialF.subsection}, we show that a polynomial filter is a multiplier in the Fourier domain and the orthogonal projections ${\bf P}_m, 1\le m\le M$, used in the definition of the GFT are polynomial filters;
see \eqref{polynomialGFT} and Theorem \ref{projection.polynomialfilter.thm}.
We remark that the polynomial filter property of the orthogonal projections associated with the graph shifts plays a crucial role in our analysis of the GSIS theory. 
In Section \ref{polynomialF.subsection}, we also characterize polynomial filters via its spectral decomposition and invariance  induced by the eigendecomposition \eqref{eigen.dec}; 
see Theorems \ref{polynomialfilter.thm} and \ref{polynomialinvariance.thm}.

\smallskip 

Shift-invariant filters have been used in graph reproducing kernel space and stationary graph signal processing \cite{chen2025, girault2015, jian2022, kotzag2019, marques2017,  ortega2018, perraudin2017,rmero2017, sandryhaila2013,   segarra2018, seto2014,  shuman13, smola2003, ward2020, zheng2025}. 
In Theorem \ref{shiftinvariantfilter.thm},  we provide a characterization of a shift-invariant filter in the spectral domain, cf. Theorem \ref{polynomialfilter.thm} for polynomial filters.
One may easily verify that the set  ${\mathcal T}$  of shift-invariant filters is a unital {\bf noncommutative} semi-group, and the set  ${\mathcal P}$ of polynomial filters is a commutative semi-subgroup. In  Theorem \ref{center.thm}, we show that  ${\mathcal P}$ is the {\bf center} of the   noncommutative semi-group ${\mathcal T}$. Thereby  we recover the result in \cite[Theorem A.3]{emirov2022} that every shift-invariant filter is a polynomial filter,  under the additional distinct 
joint spectrum  assumption \eqref{distincteigenvalue.assumption}; see Corollary \ref{shiftinvariantfilter.cor}.

\subsection{Graph Fourier transform and spectral decomposition of the identity associated with graph shifts}
\label{gft.subsection}
Let ${\bf S}_1, \ldots, {\bf S}_d$ be the graph shifts satisfying Assumption \ref{graphshift.assumption}. 
Then  we can find an orthogonal matrix ${\bf U}=[{\bf u}_1,\ldots, {\bf u}_N]$ so that they
can be  diagonalized simultaneously by the  matrix ${\bf U}$,
\begin{equation}\label{eigen.dec}
{\bf S}_l=\U {\pmb\Lambda}_l\U^T=\sum_{n=1}^N \lambda_l(n) {\bf u}_n {\bf u}_n^T,\quad  1\le l\le d,
 \end{equation}
where  for every $1\le l\le d$, the diagonal matrix ${\pmb\Lambda}_l=\diag [\lambda_l(1),\ldots, \lambda_l(N)]$ has  eigenvalues of the graph shift ${\bf S}_l$ as its diagonal entries.
Following \cite{emirov2022}, we define the {\em joint spectrum} of the commutative graph shifts
${\bf S}_1, \ldots, {\bf S}_d$ by
\begin{equation} \label{jointspectrum.def} {\pmb \Lambda}:=\big\{{\pmb \lambda}(n):=[\lambda_1(n), \ldots, \lambda_d(n)]^T \ \mid\ 1\le n\le N\big\} \subset {\mathbb R}^d.\end{equation}

 We call a function ${\bf x}:V\mapsto \R$ from the vertex set $V$ to the field of real numbers $\R$
as a {\em graph signal}  on the graph $\G$, represented by a vector ${\bf x}=[x(v)]_{v\in V}\in \R^V$. With the help of the orthonormal basis ${\bf u}_1, \ldots, {\bf u}_N$ in the eigendecomposition
 \eqref{eigen.dec} of the graph shifts ${\bf S}_1, \ldots, {\bf S}_d$, 
a conventional definition of  GFT of a graph signal ${\bf x}$ is defined by
\begin{equation}\label{GFT.distinct}
\widehat\x=\U^T\x=[ {\bf u}_1^T {\bf x},\  \ldots,\  {\bf u}_N^T {\bf x}]^T,
\end{equation}
where  ${\pmb \lambda}(1), \ldots, {\pmb \lambda}(N)$ in the joint spectrum ${\pmb \Lambda}$ are interpreted as the frequency components of the GFT and columns
 ${\bf u}_1, \ldots, {\bf u}_N$ in the orthogonal matrix  ${\bf U}$ as the corresponding variation modes \cite{Chung1997,  Chung2023, Isufi2024,  Ortega2022,   Ricaud2019, Stankovic2019}. 
We notice that the orthonormal basis ${\bf u}_1, \ldots, {\bf u}_N$ 
used to define GFT is {\bf not} unique. For instance, one may easily verify that
 the eigendecomposition  in  \eqref {eigen.dec} still hold with ${\bf u}_n$ replaced by $\tilde {\bf u}_n=\epsilon_n {\bf u}_n$ where $\epsilon_n\in \{-1, 1\}, 1\le n\le N$,  and hence $\tilde {\bf u}_1, \ldots, \tilde  {\bf u}_N$ can also be used to define GFT. It is worth noting that the algebraic structure of orthonormal bases arising from the eigendecomposition \eqref{eigen.dec} of the graph shifts ${\bf S}_1, \ldots, {\bf S}_d$ becomes more intricate when the distinct joint spectrum condition \eqref{distincteigenvalue.assumption} fails to hold \cite{Deri2017}.
 This naturally raises the following question:

 \begin{question} \label{Question1}
Can we define a graph Fourier transform   independent of the selection of the orthonormal basis  $\{{\bf u}_1, \ldots, {\bf u}_N\}$ in 
\eqref{eigen.dec}? 
\end{question}

Let ${\pmb \gamma}(m)=[\gamma_1(m), \ldots, \gamma_d(m)]^T\in {\mathbb R}^d, 1\le m\le M$, be distinct elements in the joint spectrum $\pmb \Lambda$ of the graph shifts ${\bf S}_1, \ldots, {\bf S}_d$, and without loss of generality, they are ordered within their magnitudes in  nondecreasing order,
\begin{equation}\label{gamma.order}
\|{\pmb \gamma}(1)\|\le \|{\pmb \gamma}(2)\|\le  \cdots \le \|{\pmb \gamma}(M)\|,
\end{equation}
where 
$\|{\pmb \gamma}(m)\|=\big(\sum_{l=1}^d |\gamma_l(m)|^2\big)^{1/2}, 1\le m\le M$. 
To answer Question \ref{Question1}, we define orthogonal projections  
\begin{equation}\label{projection.eq0}
{\bf P}_m=\sum_{\pmb\lambda(n)=\pmb\gamma(m)} {\bf u}_n {\bf u}_n^T,\quad 1\le m\le M.\end{equation}
One may verify that
${\bf P}_m, 1\le m\le M$,  form a spectral decomposition  of the identity:
\begin{equation}\label{projection.eq1}
    {\bf P}_m^2={\bf P}_m={\bf P}_m^T \ \ {\rm and}  \  \
    {\bf P}_m {\bf P}_{m'}={\bf 0} \ \ {\rm for  \ all}\ \ 1\le m\ne m'\le M,
\end{equation}
and  
\begin{equation}\label{projection.eq2}
\sum_{m=1}^M {\bf P}_m ={\bf I},
\end{equation}
where ${\bf I}$ is the identity matrix of  appropriate size.
Moreover, ${\bf P}_m, 1\le m\le M$, are invariant  and
commutative with the graph shifts  ${\bf S}_l, 1\le l\le d$,
by \eqref{eigen.dec} and \eqref{projection.eq0}, i.e., 
\begin{equation}\label{projection.eq3}
    {\bf P}_m {\bf S}_l=  {\bf S}_l {\bf P}_m= \gamma_l(m) {\bf P}_m
    \ \ {\rm for \ all}\ \ 1\le m\le M\ \ {\rm and} \ \ 1\le l\le d.
\end{equation}
Due to the above invariant and commutative property, we refer to
 ${\bf P}_m, 1\le m\le M$, as the {\em spectral decomposition of the identity} associated with the graph shifts 
 ${\bf S}_1, \ldots, {\bf S}_d$.

Denote the range space of the orthogonal projection ${\bf P}_m$ by
\begin{equation}\label{projection.eq4} 
W_m={\bf P}_m {\mathbb R}^V,\quad  1\le m\le M.\end{equation}
For every $1\le m\le M$, it follows from
\eqref{eigen.dec} and \eqref{projection.eq0} that
the range space $W_m$ is the intersection of 
eigenspaces $N({\bf S}_l-\gamma_l(m) {\bf I})$ of the graph shift ${\bf S}_l$ corresponding to eigenvalues $\gamma_l(m), 1\le l\le d$, i.e., 
\begin{equation}\label{projection.eq5}
W_m= \cap_{l=1}^d N({\bf S}_l-\gamma_l(m) {\bf I}), \quad 1\le m\le M.
\end{equation}
Consequently,   the spaces $W_m$ in  \eqref{projection.eq4}, and thereby  the  orthogonal projections  ${\bf P}_m, 1\le m\le M$,
  are {\bf independent} on the selection of the orthonormal basis ${\bf u}_1, \ldots, {\bf u}_N$
in the eigendecomposition \eqref{eigen.dec}.

Based on the above independence observation on the orthogonal projections ${\bf P}_m, 1\le m\le M$, we introduce a novel definition of the {\it graph Fourier transform} (GFT) of a graph signal $\bf x$ by
\begin{equation}\label{GFT.def}
\widehat {\bf x}:=[\widehat{\bf x}(1), \ldots, \widehat {\bf x}(M)]=[{\bf P}_1 {\bf x}, \ldots, {\bf P}_M\x] \in {\mathbb R}^{V\times M},\end{equation}
 cf. the graphon Fourier transform in \cite{Ghandehari22} via some spectral decomposition of the identity.
For $1\le m\le M$, we call the  $m$-th column $\widehat {\bf x}(m)$ in the GFT  
 of the graph signal ${\bf x}$ 
 as its {\em frequency component} corresponding to the {\em frequency} ${\pmb \gamma}(m)$.
Clearly, the original graph signal ${\bf x}$ can be reconstructed from its Fourier transform $\widehat{\bf x}$ by \eqref{projection.eq2}:   
\begin{equation}
{\bf x}=\widehat {\bf x} {\bf 1}=\sum_{m=1}^M \widehat {\bf x}(m)
\ \ {\rm for \ all} \ \ {\bf x}\in {\mathbb R}^V,
\end{equation}
where ${\bf 1}$ is the column vector of appropriate size with all entries taking value one. 

\smallskip

We now provide an explicit formulation of the just-defined  GFT on complete graphs, as discussed below, and of the GFT on circulant graphs, presented in Appendix \ref{circulantgraph.section}.

\begin{example} \label{uncertainty.example}
{\rm On the complete graph $K_N$ of order $N\ge 2$, one may verify that the symmetric Laplacian ${\bf L}^{\rm sym}(K_N)$
has value $1$ on the main diagonal and $-1/(N-1)$ on the off-diagonal, it has eigenvalue zero with multiplicity one and eigenvalue $N/(N-1)$ with multiplicity $N-1$, and the orthogonal projections ${\bf P}_1$ and ${\bf P}_2$ in \eqref{projection.eq0} are given by
\begin{equation}\label{uncertainty.example.eq1} 
{\bf P}_1= N^{-1} {\bf 1} {\bf 1}^T \ \ {\rm and}  \ \ {\bf P}_2={\bf I}-{\bf P}_1\end{equation}
respectively.
Therefore the GFT based on the symmetric Laplacian on the complete graph $K_N$ is given by
\begin{equation}\label{uncertainty.example.eq2}
\widehat {\bf x}(1)= \overline {\bf x}\ \ {\rm and}  \ \ \widehat {\bf x}(2)={\bf x}-\overline {\bf x}\end{equation}
for arbitrary graph signal ${\bf x}$, where $\overline {\bf x}= ({\bf 1}^T {\bf x}/N) {\bf 1}$ is the mean vector of the signal ${\bf x}$.
}\end{example}

Denote the Frobenius norm of a matrix ${\bf A}$ by $\|{\bf A}\|_F$
and the energy of a graph signal ${\bf x}=[x(v)]_{v\in V}$  by $\|{\bf x}\|_2=(\sum_{v\in V} |x(v)|^2)^{1/2}$. 
Then it follows from \eqref{projection.eq0} and \eqref{projection.eq1} that
  Parseval's identity holds  for the GFT defined in \eqref{GFT.def}:
\begin{equation}\label{parsevelidentity}
\|\widehat {\bf x}\|_F=\Big( \sum_{m=1}^M \|\widehat{\bf x}(m)\|_2^2\Big)^{1/2}=
\|{\bf x}\|_2\ \  {\rm for \ all}
\ \  {\bf x}\in {\mathbb R}^V. 
\end{equation}

\smallskip

Graph shifts ${\bf S}_1, \ldots, {\bf S}_d$ are typically selected to have some special features and are widely used to characterize the regularity of a graph signal
${\bf x}$, which is often quantified  in  the small ratios $\|{\bf S}_l{\bf x}\|_2/\|{\bf x}\|_2, 1\le l\le d$. In such a scenario,
as established in the next theorem, regular graph signals exhibit a concentration of energy predominantly in the low-frequency components of the Fourier domain, see Section \ref{bandapproximation.thm.pfsection} for the detailed proof.

\begin{theorem}\label{bandapproximation.thm}
    For a bandwidth $K\in\{1,2,\ldots, M\}$, define the low frequency component of the graph signal $\x$ with bandwith $K$ by
    \begin{equation}
    \label{bandapproximation.thm.eq1}
        \x_K=\sum_{m=1}^{K} \widehat {\bf x}(m).
    \end{equation}
    Then 
    \begin{equation}
    \label{bandapproximation.thm.eq2}
        \|{\bf x}- {\bf x}_K\|_2 
\le \frac{
\big(\sum_{l=1}^d \|{\bf S}_l {\bf x}\|_2^2\big)^{1/2}} {\|{\pmb \gamma}(K+1)\|},
    \end{equation}
    where ${\pmb \gamma}(K+1)$ is the cut-off frequency of the bandlimiting procedure.
\end{theorem}

\subsection{Graph polynomial filters}
\label{polynomialF.subsection}

We say that ${\bf H}$ is a {\em polynomial filter}
of the graph shifts ${\bf S}_1, \ldots, {\bf S}_d$ if
\begin{equation} \label{filter.def}
{\bf H}:=h({\bf S}_1, \ldots, {\bf S}_d)=\sum_{l_1=0}^{L_1} \cdots \sum_{l_d=0}^{L_d} c_{l_1,\ldots,l_d} {\bf S}_1^{l_1}\ldots {\bf S}_d^{l_d}
\end{equation}
for some multivariate polynomial $h(t_1, \ldots, t_d)=\sum_{l_1=0}^{L_1} \cdots \sum_{l_d=0}^{L_d} c_{l_1,\ldots, l_d} t_1^{l_1} \ldots t_d^{l_d}$, and denote the set of all polynomial filters by ${\mathcal P}$
\cite{emirov2022}.  Clearly, 
${\mathcal P}$ is a linear space under  matrix addition and scalar multiplication, and   a unital  commutative semi-group under matrix multiplication.

By  \eqref{eigen.dec} and \eqref{projection.eq3}, we  observe that the graph shift operation in the spatial domain is a multiplier in the Fourier domain,
\begin{equation}\label{multiplier.prop}
\widehat{{\bf S}_l{\bf x}}=\widehat {\bf x} \,{\pmb\Gamma}_l,\ \ 1\le l\le d,
\end{equation}
where ${\pmb \Gamma}_l=\diag(\gamma_l(1),\gamma_l(2),\ldots, \gamma_l(M)), 1\le l\le d$.
Applying \eqref{multiplier.prop} repeatedly, we observe that
the polynomial filtering  procedure  in the spatial domain is a multiplier in the  Fourier domain,
$\widehat{{\bf H}{\bf x}}=
\widehat {\bf x} h({\pmb\Gamma}_1, \ldots, {\pmb \Gamma}_L)$, 
or equivalently
\begin{equation}\label{polynomialGFT}
\widehat {{\bf H} {\bf x}}(m)= h({\pmb \gamma}(m)) \widehat {\bf x}(m),\ \ 1\le m\le M. 
\end{equation}

Let $\delta$ be the conventional Kronecker delta function.
In the following theorem, we show that the orthogonal  projections ${\bf P}_m, 1\le m\le M$, in \eqref{projection.eq0} 
are polynomial filters of the graph shifts ${\bf S}_1, \ldots, {\bf S}_d$; see Section \ref{projection.polynomialfilter.thm.pfsection} for the detailed proof. 

\begin{theorem}\label{projection.polynomialfilter.thm} 
Let ${\bf P}_m, 1\le m\le M$, be  the orthogonal projections  in 
\eqref{projection.eq0}. Then 
\begin{equation} \label{projection.polynomialfilter.thm.eq0}
{\bf P}_m\in {\mathcal P},\quad 1\le m\le M.
\end{equation}
Moreover,  
\begin{equation} \label{projection.polynomialfilter.thm.eq1}
{\bf P}_m= h_m({\bf S}_1, \ldots, {\bf S}_d),\ \ 1\le m\le M,
\end{equation}
where $h_m, 1\le m\le M$, are multivariate polynomials of degree at most $M-1$ satisfying 
$h_m ({\pmb \gamma}(m'))=\delta_{mm'}, 1\le m, m'\le M$.
\end{theorem}

We emphasize that the polynomial filter property of the orthogonal projections ${\bf P}_m$, $1\le m\le M$, established in Theorem \ref{projection.polynomialfilter.thm}, plays an essential role
in our analysis of the GSIS theory.

For a continuous function $f$ on  a neighborhood of the joint spectrum $\pmb \Lambda$ of the graph shifts ${\bf S}_1, \ldots, {\bf S}_d$,  
we have 
$$f({\bf S}_1, \ldots, {\bf S}_d)=\sum_{m=1}^M f({\pmb \gamma}(m)){\bf P}_m.$$
This together with Theorem \ref{projection.polynomialfilter.thm} implies that
\begin{equation} \label{analytic.polynomial} f({\bf S}_1, \ldots, {\bf S}_d)\in {\mathcal P}.
\end{equation}
Hence we have the following result about fractional graph shifts, which does {\bf not} hold in the  real line setting. 

\begin{corollary}\label{fractionalshift.cor}
If all graph shifts ${\bf S}_1, \ldots, {\bf S}_d$ are positive semi-definite, then all fractional shifts
${\mathbf S}_l^t, 1\le l\le d, t\ge 0$, are polynomial filters. 
\end{corollary}

Next,  we provide a characterization for a filter to be a polynomial filter via its spectral decomposition; see Section \ref{polynomialfilter.thm.pfsection} for the detailed proof.

\begin{theorem}\label{polynomialfilter.thm} Let ${\bf H}$ be  a graph filter
and ${\bf P}_m, 1\le m\le M$, be as in 
\eqref{projection.eq0}.  
Then  ${\bf H}\in {\mathcal P}$ if and only if
  there exist real numbers $\mu_m, 1\le m\le M$, such that
\begin{equation}\label{shiftinvariantfilter.thm.eq1}
{\bf P}_m {\bf H}= {\bf H}{\bf P}_m=\mu_m{\bf P}_m,\quad 1\le m\le M. 
\end{equation} 
\end{theorem}

By \eqref{projection.eq1}, \eqref{projection.eq2} and  Theorem \ref{polynomialfilter.thm}, we see that an equivalent requirement for a  filter  ${\bf H}$ to be a polynomial filter is
 that there exist real numbers $\mu_m, 1\le m\le M$, such that
\begin{equation}\label{shiftinvariantfilter.thm.eq3}
{\bf H}=\sum_{m=1}^M\mu_m{\bf P}_m,
\end{equation} 
or equivalently in the Fourier domain,
\begin{equation} \label{shiftinvariantfilter.thm.eq2}
\widehat { {\bf H} {\bf x}}(m)= \mu_m \widehat {\bf x}(m), \quad 1\le m\le M,
\end{equation}
hold for all graph signals ${\bf x}$.

By \eqref{shiftinvariantfilter.thm.eq3}, we have the following  property for the power of a polynomial filter, cf. Corollary \ref{fractionalshift.cor}.

\begin{corollary}\label{inversepolynomial.cor}
Let ${\mathbf H}\in {\mathcal P}$. Then 
\begin{itemize}
\item[{(i)}]  ${\bf H}^t\in {\mathcal P} $ for all $t\ge 0$ if  ${\bf H}$ is positive semi-definite. 

\item[{(ii)}]  ${\bf H}^{-1}\in {\mathcal P}$ if  ${\bf H}$ is invertible. 

\item[{(iii)}] ${\bf H}^{s}\in {\mathcal P}$ for all $s\in {\mathbb R}$ if  ${\bf H}$ is positive definite.

\end{itemize}

 \end{corollary}

\smallskip

We conclude this subsection  by characterizing polynomial filters via an invariance property induced by the eigendecomposition \eqref{eigen.dec};
see Section \ref{polynomialinvariance.thm.pfsection}
for the  detailed proof.

\begin{theorem}\label{polynomialinvariance.thm}  Let ${\bf H}$ be a graph filter. Then ${\bf H}\in {\mathcal P}$ if and only if
 ${\bf U}^T{\bf H} {\bf U}$ is independent on the orthogonal matrix
 ${\bf U}$ in the eigendecomposition \eqref{eigen.dec} of the graph shifts
${\bf S}_1, \ldots, {\bf S}_d$.
\end{theorem}

\subsection{Shift-invariant filters}
\label{shiftinvariantF.subsection}

We say that  ${\bf H}$ is a {\em  shift-invariant} filter if it  commutes with the
graph shifts ${\bf S}_1, \ldots, {\bf S}_d$, i.e.,  
\begin{equation}\label{shiftinvariant.def}
{\bf H} {\bf S}_l={\bf S}_l {\bf H}, \quad 1\le l\le d,
\end{equation}
and denote the set of all shift-invariant filters by ${\mathcal T}$.
One may easily verify that 
${\mathcal T}$ is a linear space under  matrix addition and scalar multiplication, and   a unital {\bf noncommutative} semi-group under matrix multiplication. 
Moreover,  both the  transpose and pseudo-inverse of a shift-invariant filter remain shift-invariant, and thereby
 the inverse of any invertible shift-invariant filter is  shift-invariant too.

Let ${\bf H}$ be a shift-invariant filter. 
Applying \eqref{shiftinvariant.def} repeatedly, we see that
a shift-invariant filter commutes with any polynomial filter.
This together with 
Theorem \ref{projection.polynomialfilter.thm} implies that
a shift-invariant filter ${\bf H}$ commutes with the projections ${\bf P}_m, 1\le m\le M$,  in 
\eqref{projection.eq0}, i.e., 
\begin{equation}\label{shiftinvariantfilter.projection}
{\bf H}{\bf P}_m=  {\bf P}_m {\bf H}, \quad 1\le m\le M. 
\end{equation} 
  This together with  \eqref{projection.eq1} and \eqref{projection.eq2} implies that
  \begin{equation}\label{shiftinvariantfilter.commutative.eq1}
{\bf H}{\bf P}_m=  {\bf P}_m {\bf H}={\bf P}^2_m {\bf H}={\bf P}_m {\bf H}{\bf P}_m, \quad 1\le m\le M, 
\end{equation}
 and \begin{equation}\label{shiftinvariantfilter.commutative.eq2}
{\bf H}=  {\bf H}\sum_{m=1}^{M}{\bf P}_m=\sum_{m=1}^{M}{\bf P}_m {\bf H}{\bf P}_m. 
\end{equation}
    Similar to the spectral decomposition
\eqref{shiftinvariantfilter.thm.eq3}
for a polynomial filter, we have 
the following spectral decomposition for  a shift-invariant filter; see Section 
\ref{shiftinvariantfilter.thm.pfsection} for the detailed proof and cf.
\eqref{shiftinvariantfilter.thm.eq3} and
Theorem \ref{polynomialfilter.thm} for polynomial filters.

\begin{theorem}\label{shiftinvariantfilter.thm}
Let ${\bf P}_m, 1\le m\le M$, be the orthogonal projections  in 
\eqref{projection.eq0}.
Then ${\bf H}\in {\mathcal T}$ if and only if     \begin{equation}\label{shiftinvariantfiltr.thm.eq1}
{\bf H}=\sum_{m=1}^{M}{\bf P}_m{\bf H}_m{\bf P}_m
\end{equation}
for some filters ${\bf H}_m, 1\le m\le M$.  
\end{theorem}

By Theorem \ref{shiftinvariantfilter.thm},  we have 
 \begin{corollary}\label{inverseinvariant.cor}
Let ${\mathbf H}\in {\mathcal T}$. Then 
\begin{itemize}
\item[{(i)}]  ${\bf H}^t\in {\mathcal T} $ for all $t\ge 0$ if  ${\bf H}$ is positive semi-definite. 

\item[{(ii)}]  ${\bf H}^{-1}\in {\mathcal T}$ if  ${\bf H}$ is invertible. 

\item[{(iii)}] ${\bf H}^{s}\in {\mathcal T}$ for all $s\in {\mathbb R}$ if  ${\bf H}$ is positive definite.

\end{itemize}

 \end{corollary}

\smallskip

 Define 
the {\em center} of the unital noncommutative semi-group ${\mathcal T}$  by 
\begin{equation}\label{shiftinvariantcenter.def}
   C(\mathcal T):=\{{\bf G}\in\mathcal T\ |\ {\bf H}{\bf G}={\bf G}{\bf H} {\rm\ for\  all\ } \ {\bf H}\in \mathcal T\}.
\end{equation}  
By \eqref{projection.eq1} and
Theorems \ref{polynomialfilter.thm} and \ref{shiftinvariantfilter.thm}, we see that all polynomial filters are in the  center  $C(\mathcal T)$. In the following theorem, we show that the converse is also true; see Section 
\ref{center.thm.pfsection} for the detailed proof.

\begin{theorem}
\label{center.thm}
Let ${\mathcal P}$ and ${\mathcal T}$ be the unital semi-group of  polynomial filters and shift-invariant filters respectively, and let $C({\mathcal T})$ be as in  
\eqref{shiftinvariantcenter.def}.
Then 
\begin{equation}\label{center.thm.eq1}
C({\mathcal T})=
{\mathcal P}.
\end{equation}
\end{theorem}

Let ${\bf H}$ be a shift-invariant filter.
For the scenario that
 the distinct
joint spectrum  assumption 
\eqref{distincteigenvalue.assumption} holds, 
we have
\begin{equation}
{\bf H}{\bf P_m}={\bf P}_m {\bf H}{\bf P_m}=\mu_m {\bf P}_m
\end{equation}
for some real numbers $\mu_m, 1\le m\le M$,
where the first equality follows from
\eqref{projection.eq1} and \eqref{shiftinvariantfilter.projection},
and the second equality holds as the range  space $W_m$ of the orthogonal projection ${\bf P}_m$
is one-dimensional  for every $1\le m\le M$ by \eqref{distincteigenvalue.assumption}. 
Therefore as a consequence of Theorems \ref{shiftinvariantfilter.thm} and \ref{center.thm},  we recover the conclusion in 
 \cite[Theorem A.3]{emirov2022}.

\begin{corollary}\label{shiftinvariantfilter.cor} 
Let the graph shifts ${\bf S}_1, \cdots, {\bf S}_d$ satisfy 
\eqref{distincteigenvalue.assumption} and Assumption \ref{graphshift.assumption}.  Then 
\begin{equation}\label{shiftinvariantfilter.cor.eq1}
C({\mathcal T})={\mathcal P}={\mathcal T}.
\end{equation}
\end{corollary}

\section{Graph shift-invariant spaces, bandlimited spaces and maximal invariant subspaces}
\label{sisbandlimiting.section}

Let ${\mathcal G}=(V, E, {\bf W})$ be an undirected simple graph of order $N\ge 2$,  ${\bf S}_1, \ldots, {\bf S}_d$ be graph shifts satisfying Assumption \ref{graphshift.assumption}, and
let ${\pmb \gamma}(m),  
1\le m\le M$, be distinct elements in the joint spectrum $\pmb \Lambda$  in \eqref{jointspectrum.def}.
We say  
that $J$ is a  {\em graph range function} if
\begin{equation} 
\label{rangefunction.def-1} J({\pmb \gamma}(m))\ 
{\rm is \ a \ linear \ subspace \ of }\  \R^V  \ {\rm for \ every} \ 1\le m\le M.
\end{equation}
In  Section \ref{rangefunction.section}, we first characterize a 
 GSIS in the Fourier domain via a graph range function; see Theorem \ref{rangecharacterization.thm}.
Moreover, from the argument used in the proof of Theorem \ref{rangecharacterization.thm}, we see that
the graph range function $J$ of a GSIS $U$ satisfies
\begin{equation}\label{rangefunction.def0}
J({\pmb \gamma}(m))\cap W_m=U\cap W_m, \quad 1\le m\le M.
\end{equation}
We remark that the characterization of GSISs in Theorem \ref{rangecharacterization.thm} is analogous to the Fourier-based description of SISs on the real line via range functions \cite{bownik2000, bownik2003, deBoor1994, ron1995}. However, in contrast to the equation \eqref{rangefunction.def0} in the graph setting, there is {\bf no} 
 direct analogue  in the real line setting.

For a GSIS $U$, define its {\em  dimension function} by  
\begin{equation}
    \label{dimension.def}
\dim_U({\pmb \gamma}(m))=\dim U\cap W_m,\  1\le m\le M.\end{equation} 
In Section \ref{dimension.subsection}, we examine various properties of the dimension function of a GSIS, focusing on the summation of two GSISs and the orthogonal complement of a GSIS; see Propositions \ref{dimensionsum.prop} and \ref{complement.prop}.

Given a shift-invariant filter ${\bf H}$, we may easily verify that
its range and kernel spaces are GSISs. In  Theorem
\ref{filterrange.thm}, we show that every GSIS is the  space of some shift-invariant filter and also the kernel space of another shift-invariant filter. In Section \ref{filterrange.section}, we also discuss 
injective, surjective and isometric properties of a
shift-invariant operator between two GSISs; see Proposition \ref{sioperator.prop}.
Here given GSISs $U_1$ and $U_2$, we say that
 a linear operator ${\bf T}: U_1\to U_2$ is
{\em shift-invariant} if it can be represented by a shift-invariant filter, i.e., 
\begin{equation} \label{shiftinvariantoperator.def}
{\bf T} {\bf S}_l={\bf S}_l {\bf T}, \quad 1\le l\le d. 
    \end{equation}

We say that a linear space $U$  of graph signals on the graph ${\mathcal G}$ is
 {\em bandlimited} if
\begin{equation}\label{bandlimited.def}
U\cap W_m=W_m \ {\rm or}\ \{0\}, \quad 1\le m\le M,
\end{equation}
and  {\em super-shift-invariant}
if
\begin{equation}\label{superSIS.def}
{\bf T} {\bf x}\in U \ \text{\rm for all shift-invariant filters} \ {\bf T}\in\mathcal T \ {\rm and\ graph\ signals}\ 
\x\in U.
\end{equation}
The bandlimited spaces, also known as Paley-Wiener spaces, have been widely used in  GSP;
see \cite{chen2015, Chung2024, pesenson2008, pesenson2009, puy2018}
and references therein.
The binary structure in \eqref{bandlimited.def} allows the bandlimited space to be reformulated in the familiar  Fourier domain framework:  
\begin{equation} \label{bandlimit.def2} B_\Omega=\big\{{\bf x}\in {\mathbb R}^V \!  \mid \! {\rm supp} \ \widehat {\bf x}\subset \Omega\big\}\end{equation}
where  $\supp{\widehat{\bf x}}=\{m\in\{1,2,\ldots, M\}\mid\widehat{\bf x}(m)\ne {\bf 0}\}$ and $\Omega$ is a subset of 
$\{1, \ldots, M\}$. In Theorem \ref{bandlimited.thm} of Section \ref{bandlimited.section}, we show that a linear space of graph signals  is super-shift-invariant if and only if it is bandlimited. 
The above equivalence is established in \cite{Chung2024}
when the graph shifts ${\bf S}_1, \ldots, {\bf S}_d$  satisfies
the district joint spectrum condition
\eqref{distincteigenvalue.assumption}. 
In Section \ref{bandlimited.section}, we also present a Beurling-type theorem for a bandlimited space; see Theorem \ref{Beurling.thm}.

Given  a  GSIS  $U$, we know that
 ${\bf S}_1 U, \ldots, {\bf S}_d U$ are  linear subspaces of $U$. In the scenario that the graph shifts ${\bf S}_1, \ldots, {\bf S}_d$ are invertible, the GSIS $U$ coincides with its shifted versions ${\bf S}_1 U, \ldots, {\bf S}_d U$; 
cf.  Corollary \ref{allpassinvariance.cor}.
This naturally raises the question of whether a GSIS and its shifted versions can coincide, i.e., \eqref{maximalsis.eq0} holds,  without imposing the invertibility of the graph shifts  ${\bf S}_1, \ldots, {\bf S}_d$.  
In Theorem \ref{sisequality.thm} of  Section \ref{maximalsis.section}, we provide an answer to the above question using the dimension function  of the GSIS $U$. 
In Section \ref{maximalsis.section},  we also find the maximal shift-invariant subspace  $W$ of $U$ such that
\eqref{maximalsis.eq1} holds;
see Theorem \ref{maximalsis.thm}.

Before proceeding with a more detailed analysis of GSISs, let us  make two  remarks concerning their pseudo-inverse shift-invariance
and fractional shift-invariance.  

\begin{remark}\label{pseudoinversesis.rem}
{\rm
Denote the pseudo inverses  of the graph shifts ${\bf S}_l$ by ${\bf S}_l^\dag, 1\le l\le d$. Since
the pseudo inverses of the graph shifts are polynomial filters by
\eqref{analytic.polynomial}, we  have
\begin{equation}
{\bf S}_1^\dag \x, \ldots, {\bf S}_d^\dag \x\in U \ {\rm for \  all}
\ \x\in U
\end{equation}
provided that $U$ is  a GSIS.
Therefore for the scenario where the graph shifts ${\bf S}_1, \ldots, {\bf S}_d$ are invertible, a GSIS $U$ is also {\em bi-directional shift-invariant}, i.e., 
\begin{equation}\label{bi_sis.def1}
{\bf S}_1^{\alpha_1}\cdots {\bf S}_d^{\alpha_d}{\bf x}\in U \quad  {\rm for \ all} \ \alpha_1, \ldots, \alpha_d\in\Z
\ {\rm  and } \ {\bf x}\in U.
\end{equation} 
In the  real line setting, where the shift operator is invertible, the notion of an  SIS is likewise defined in terms of bi-directional invariance \cite{akram2001, akram2005, bownik2000, bownik2003, chui2006, deBoor1994}.
}
\end{remark}

\begin{remark}\label{fractional.rem}
{\rm We note that, under the assumption that the graph shifts 
${\bf S}_1,\dots,{\bf S}_d$
 are positive semi-definite, their fractional powers 
 ${\mathbf S}_l^t, 1\le l\le d, t\ge 0$, are polynomial filters by Corollary \ref{inversepolynomial.cor}. 
 Consequently, any GSIS is  invariant under fractional shift operations as well. In contrast, this fractional shift-invariance does {\bf not} generally hold in the  real line setting, where such invariance fails in most cases \cite{akram2010, akram2011, cabrelli2016, fuhr2014, hardin2018}.
 }
\end{remark}

\subsection{Graph shift-invariant spaces and graph range function}
\label{rangefunction.section}

In the following theorem, 
we provide a characterization to a  GSIS via its graph range function; see Section \ref{rangecharacterization.thm.pfsection} for the detailed proof.

\begin{theorem}\label{rangecharacterization.thm}
 Define the  GFT $\widehat {\bf x}$ of a graph signal ${\bf x}$ as in 
\eqref{GFT.def}.
Then a linear space $U$ of graph signals is shift-invariant if and only if there exists a  graph range function $J$ such that 
    \begin{equation} \label{rangecharacterization.thm.eq0}
        U=\{{\bf x}\in {\mathbb R}^V\ \mid \ \widehat{{\bf x}}(m)\in J({\pmb\gamma}(m)),\ 1\le m\le M\}.
    \end{equation}
 \end{theorem}

We remark that the range function $J$ in \eqref{rangecharacterization.thm.eq0} used to define the GSIS $U$  is not unique. However, from the argument in the proof of Theorem  \ref{rangecharacterization.thm}, we see that
for two range functions $J_1$ and $J_2$,
the  corresponding GSISs in  \eqref{rangecharacterization.thm.eq0}  are the same if and only if 
\begin{equation*}
J_1({\pmb \gamma}(m))\cap W_m= J_2({\pmb \gamma}(m))\cap W_m=U\cap W_m, \ 1\le m\le M,
\end{equation*}
where the  orthogonal subspaces $W_m, 1\le m\le M$, are given in \eqref {projection.eq4}. For a given GSIS $U$, we may select the range function $J$  in \eqref{rangecharacterization.thm.eq0}
so that 
\begin{equation*} 
J({\pmb \gamma}(m)) \subset W_m,\quad  1\le m\le M.
\end{equation*}
Such a range function is unique
and we call it as the {\em graph range function} associated with the GSIS $U$ and denote it by $J_U$, i.e., 
\begin{equation}\label{Juspace.def}
    J_U({\pmb \gamma}(m))=U\cap W_m, \ \ 1\le m\le M,
\end{equation}
and call its support  
\begin{equation}\label{spectrumsis.def}
\Omega=\big\{m\in \{1, \ldots ,M\} \ |\  J_U({\pmb \gamma}(m))\ne \{0\}\big \},
\end{equation}
as the {\em spectrum} of the GSIS $U$.
For a bandlimited space $B_\Omega$, we obtain from \eqref{bandlimit.def2}
and \eqref{spectrumsis.def} that its spectrum coincides with the supporting set $\Omega$ of its frequencies.

We say that a polynomial filter ${\bf P}=p({\bf S}_1, \ldots, {\bf S}_d)$  is an {\em all-pass filter}
if $p({\pmb \gamma}(m))\ne 0$ for all $1\le m\le M$, or equivalently  if  the matrix to represent the polynomial filter ${\bf P}$ is invertible.  By \eqref{projection.eq3} and Theorem 
\ref{rangecharacterization.thm}, we have the following result on the invariance of a GSIS under all-pass polynomial filtering procedure.

\begin{corollary}\label{allpassinvariance.cor}
    Let ${\bf P}$  be an all-pass polynomial filter
    and $U$ be a GSIS. Then 
 ${\bf P}U=U$.
\end{corollary}

\subsection{Graph shift-invariant space and dimension function} \label{dimension.subsection}
Let $\dim_U$ be the dimension function of the  GSIS $U$ defined in \eqref{dimension.def}.
By 
\eqref{Juspace.def},  we obtain
\begin{equation}  \label{dimension.def2}
    \dim_U({\pmb \gamma}(m))=\dim J_U({\pmb \gamma}(m)) ,\quad  1\le m\le M.
\end{equation}
This together with \eqref{projection.eq1} and \eqref{projection.eq2} implies that
\begin{equation}\label{dimensionU.eq}
\sum_{m=1}^M  \dim_U({\pmb \gamma}(m))=\sum_{m=1}^M  \dim J_U({\pmb \gamma}(m))=\dim U.
\end{equation}

From the definition of a GSIS, we may verify that
the sum and intersection of two GSISs are still shift-invariant. By   \eqref{dimension.def2}, we  have the following results for their dimension functions.

\begin{proposition} \label{dimensionsum.prop} Let 
 $U_1$ and $U_2$ be GSISs.  Then 
 $U_1+U_2$  and $U_1\cap U_2$ are  GSISs. Moreover, 
\begin{equation*}
J_{U_1+U_2}({\pmb \gamma}(m))=
J_{U_1}({\pmb \gamma}(m))+ J_{U_2}({\pmb \gamma}(m)),
\end{equation*}
\begin{equation*}
J_{U_1\cap U_2}({\pmb \gamma}(m))=
J_{U_1}({\pmb \gamma}(m))\cap J_{U_2}({\pmb \gamma}(m)),
\end{equation*}
and
\begin{equation}
\label{dimsum.eq}
\dim_{U_1+U_2}({\pmb \gamma}(m))=\dim_{U_1}({\pmb \gamma}(m))+\dim_{U_2}({\pmb \gamma}(m))-\dim_{U_1\cap U_2}({\pmb \gamma}(m)), \quad 1\le m\le M.\end{equation}
\end{proposition}

 As a consequence of Proposition \ref{dimensionsum.prop}, we see that two GSISs ${ U}_1$ and ${ U}_2$
 has trivial intersection if and only if
the dimension function associated with $U_1+U_2$ is the sum of the dimension functions associated with $U_1$ and $U_2$.

For a GSIS $U$,  one may also verify that its orthogonal complement $U^\perp$ in ${\mathbb R}^V$ has the following property.

\begin{proposition}\label{complement.prop}
    Let $U$ be a GSIS. Then its orthogonal complement $U^\perp$
is  a GSIS as well. Moreover, the  range functions $J_U$ and $J_{U^\perp}$ for the GSIS $U$ and its orthogonal complement $U^\perp$  satisfy
\begin{equation}\label{orthogonalcomplement.eq1}
J_U({\pmb \gamma}(m))\oplus J_{U^\perp}({\pmb \gamma}(m))=W_m, \quad 1\le m\le M,
\end{equation}
and their dimension functions satisfy
$$
\dim_{U}({\pmb \gamma}(m))+\dim_{U^\perp}({\pmb \gamma}(m))=
\dim W_m, \quad 1\le m\le M.$$
\end{proposition}

\smallskip

By \eqref{dimension.def} and \eqref{dimensionU.eq}, we have 

  \begin{proposition}\label{sisequality.prop}
Let   $U_1$ and $U_2$ be GSISs and satisfy $U_1\subset U_2$. Then $U_1=U_2$
if and only if their dimension functions are the same. 
\end{proposition}

Taking $U_2=R^V$ in  Proposition \ref{sisequality.prop}, we conclude that a GSIS $U$ is the whole space $R^V$ if and only if
  \begin{equation}
  \dim_U({\pmb \gamma}(m))=\dim W_m,\quad 1\le m\le M.
  \end{equation}

\subsection{Graph shift-invariant spaces and shift-invariant filters}
\label{filterrange.section}
Given a graph filter ${\bf H}$, define its range and kernel spaces by
\begin{equation}\label{range.def}R({\bf H})=\big\{ {\bf H}{\bf x}\ |\  {\bf x}\in {\mathbb R}^V\big\}
\end{equation}
and
\begin{equation}\label{kernel.def} K({\bf H})=\big\{ {\bf x} \in {\mathbb R}^V \ | \ {\bf H}{\bf x}={\bf 0}\big\}
\end{equation}
respectively. 
One may  verify that the
range space $R({\bf H})$ and the kernel space $K({\bf H})$
associated with a shift-invariant filter ${\bf H}$ are  GSISs. In the following theorem, we show that the converse is also true; see Section \ref{filterrange.thm.pfsection}
for the detailed proof.

\begin{theorem}\label{filterrange.thm} Let $U$ be a linear space of graph signals. Then 
$U$ is a GSIS if and only if it is the range space of a shift-invariant filter if and only if it is the kernel space of another shift-invariant filter.
\end{theorem}

For a shift-invariant linear operator $T$ between two GSISs,  by \eqref{projection.eq1}, \eqref{projection.eq2}, \eqref{shiftinvariantoperator.def} and Theorem \ref{projection.polynomialfilter.thm}, we  have the following spectral decomposition
    \begin{equation}\label{shiftinvariantoperator.eq0} T=\sum_{m=1}^M {\bf P}_m T {\bf P}_m.
    \end{equation}
Then following the standard argument, we have the following properties
for shift-invariant operators.

\begin{proposition}\label{sioperator.prop} Let $U_1, U_2$ be GSISs,
and  $T: U_1\to U_2$ be a shift-invariant linear operator.
Then 
\begin{itemize}
\item[{(i)}] $T$ is surjective if and only if ${\bf P}_mT{\bf P}_m: W_m\cap U_1\longmapsto W_m\cap U_2$ is surjective for every $1\le m\le M$.

\item[{(ii)}] $T$ is injective if and only if ${\bf P}_mT{\bf P}_m: W_m\cap U_1\longmapsto W_m\cap U_2$ is injective for every $1\le m\le M$.

\item[{(iii)}] $T$ is isometric if and only if ${\bf P}_mT{\bf P}_m: W_m\cap U_1\longmapsto W_m\cap U_2$ is isometric for every $1\le m\le M$.

\end{itemize}

\end{proposition}

\subsection{Super-shift-invariant spaces and bandlimited spaces}
\label{bandlimited.section}

Clearly, super-shift-invariant spaces and bandlimited spaces are shift-invariant. 
In the following theorem, we establish the equivalence between
 super-shift-invariance and  bandlimitedness of a linear space.

\begin{theorem}\label{bandlimited.thm}
Let $U$ be a linear space of graph signals. Then  it is super-shift-invariant if and only if it is bandlimited.
\end{theorem}

Following the argument used in Theorem \ref{filterrange.thm} and applying the conclusion in Theorem \ref{bandlimited.thm}, we have the following result to characterize bandlimited spaces via polynomial filters.

\begin{corollary}\label{bandlimitedpolynomialfilter.cor}
Let $U$ be a linear space of graph signals.
Then $U$ is bandlimited if and only if it is the range space of a polynomial filter if and only if it is the kernel space of another polynomial filter. 
\end{corollary}

By Theorem \ref{bandlimited.thm}, we recover the following result about the equivalence between shift-invariance and bandlimitedness in \cite{Chung2024}.

\begin{corollary}\label{bandlimited.cor} Let $U$ be a linear space of graph signals. 
If the graph shifts ${\bf S}_1, \ldots, {\bf S}_d$ satisfies  Assumption \ref{graphshift.assumption} and the distinct joint spectrum condition \eqref{distincteigenvalue.assumption}, then 
$U$ is shift-invariant if and only if it is super-shift-invariant if and only if it is bandlimited.
\end{corollary}

We say that $B$ is a finite Blaschke product with real zeros if
it is a finite product of the form $(t-c_{k})/(1- {c}_{k}t)$ for some real numbers $c_k$ with $|c_k|<1$, i.e., 
$$B(t)=\prod_{k}   \frac{t-c_{k}}{1-c_{k}t}. $$
Under the  normalization assumption on
 eigenvalues of the graph shifts ${\bf S}_1, \ldots, {\bf S}_d$,
\begin{equation}\label{graphshift.productassumption}
 \|{\pmb \gamma}(m)\|<1, \ \ 1\le m\le M,
 \end{equation}
 combining \eqref{projection.eq3} with the  bandlimited characterization \eqref{bandlimit.def2} yields the following Beurling-type conclusion.

\begin{theorem}\label{Beurling.thm}
 Let the graph shifts ${\bf S}_1, \ldots, {\bf S}_d$
 satisfy \eqref{graphshift.productassumption} and Assumption \ref{graphshift.assumption}, 
 and let $U$ be a bandlimited space.
Then there exist $a_1, \ldots a_d$ with $\sum_{l=1}^d|a_l|^2=1$,  and a  Blaschke product $B(t)$
with zeros in $(-1, 1)$ such that
\begin{equation} \label{Beurling.cor.eq1} U=B({\bf T}) {\mathbb R}^V,
\end{equation}
where ${\bf T}= a_1{\bf S}_1+\ldots+a_d {\bf S}_d$.
\end{theorem}

In particular, we may select ${\bf a}=[a_1, \ldots, a_d]^T$  and the  Blaschke product  $B(t)$ so that $\sum_{l=1}^d|a_l|^2=1$ and
\begin{equation}\label{FGSIS.thm.eq1-} \mu(m):=\sum_{l=1}^d a_l \gamma_l(m), 1\le m\le M, {\rm \ are\  distinct},\end{equation}
and 
the root set of $B(t)$ is $\mu(m)\in (-1,1),  m\not\in \Omega$,  
i.e.,  $$B(t)=\prod_{m\not\in \Omega} \frac{t-\mu(m)}{1- \mu(m) t}, $$
where 
 $\Omega$ 
is  the spectrum \eqref{spectrumsis.def} of the GSIS $U$.

\subsection{Maximal invariant subspaces and bi-directional shift-invariant spaces}
\label{maximalsis.section}

In the following theorem, we provide a characterization for the equalities in \eqref{maximalsis.eq0} to hold; see Section \ref{sisequality.thm.pfsection} for the detailed proof. 

\begin{theorem}\label{sisequality.thm} Let $U$ be a  GSIS. Then the following statements are equivalent:
\begin{itemize}
\item[{(i)}] $U$  satisfies \eqref{maximalsis.eq0}, i.e., ${\bf S}_1U=\cdots={\bf S}_dU=U$.

\item[{(ii)}] ${\bf S}_1\ldots {\bf S}_d U=U$.

\item[{(iii)}] $\dim_U({\pmb \gamma}(m))=0$
if $\gamma_1(m)\cdots \gamma_d(m)=0$. 

\end{itemize}
\end{theorem}

Clearly, a bidirectional-shift-invariant space is shift-invariant.
By 
Theorem 
\ref{sisequality.thm}, 
the converse holds.

\begin{corollary}
    \label{bidirectionsis.cor}
    Let ${\bf S}_1, \ldots, {\bf S}_d$ be invertible graph shifts satisfying  Assumption \ref{graphshift.assumption} and $U$ be a linear space of graph shifts on the graph ${\mathcal G}$. Then 
    $U$ is shift-invariant if and only if it is bidirectional-shift-invariant. 
\end{corollary}

Given a GSIS $U$, we set 
\begin{equation}\label{maximalsis.eq2} U_M={\bf S}_1\ldots {\bf S}_d U.\end{equation}  By \eqref{projection.eq3}, one may verify that
 $\dim_{U_M}({\pmb \gamma}(m))=0$
for all $1\le m\le M$ with
$\gamma_1(m)\cdots\gamma_d(m)=0$. 
Therefore 
\begin{equation*}
{\bf S}_l U_M=U_M,\quad  1\le l\le d,
\end{equation*}
by Theorem \ref{sisequality.thm}.
On the other hand, for any  shift-invariant subspace $W$ of $U$ satisfying \eqref{maximalsis.eq1},
 we have
$$W={\bf S}_1\cdots{\bf S}_dW\subset U_M.$$
This shows that the GSIS
$U_M$ in \eqref{maximalsis.eq2} is the maximal subspace satisfying 
\eqref{maximalsis.eq1},

\begin{theorem}\label{maximalsis.thm}
Let $U$ be a GSIS. Then 
${\bf S}_1\cdots {\bf S}_dU$ is its maximal shift-invariant subspace  satisfying
\eqref{maximalsis.eq1}, i.e., if
$W$ is a shift-invariant subspace of $U$ satisfying
\eqref{maximalsis.eq1}, then 
$W\subset {\bf S}_1\cdots {\bf S}_dU$.
\end{theorem}

   For the scenario that the distinct  joint spectrum
   assumption \eqref{distincteigenvalue.assumption} holds, 
 a GSIS $U$ can be characterized by its supporting set in the Fourier domain; see \eqref{bandlimit.def2}.
Then the maximal shift-invariant space is described as follows.

\begin{corollary}
Let $U$ be a GSIS. If 
 the graph shifts ${\bf S}_1, \ldots, {\bf S}_d$ satisfies \eqref{distincteigenvalue.assumption},
then 
$${\bf S}_1\cdots {\bf S}_d U=\{{\bf x}\in U\ | \ {\rm supp} \ \widehat {\bf x}\subset \tilde \Omega\}
$$
where  $\Omega\subset \{1, \ldots, M\}$ is the set in \eqref{bandlimit.def2}
and  $\tilde \Omega=\{m\in \Omega\ | \ \gamma_l(m)\ne  0\ {\rm for \ all} \ 1\le l\le d\}$.
\end{corollary}

\section{Finitely-generated  graph shift-invariant spaces}
\label{FGSIS.section}

 Let $\Phi=\{\phi_1, \ldots, \phi_r\}$ be a  set of nonzero graph signals
on the graph $\G$ and set ${\bf S}^{\pmb\alpha}={\bf S}_1^{\alpha_1}\cdots{\bf S}_d^{\alpha_d}$ for $\pmb\alpha=[\alpha_1,\ldots, \alpha_d]\in {\mathbb Z}_+^d$.
 We say that a linear space of graph signals is a {\em GSIS generated by $\Phi$} if it is the minimal GSIS  containing $\Phi$, i.e.,
\begin{equation}\label{FGSIS.def}
\mathcal S(\Phi):=\span\{{\bf S}^{\pmb\alpha}\phi\mid \pmb\alpha\in  {\mathbb Z}_+^d, \phi\in\Phi\},
\end{equation}
and a {\em Principal  Graph Shift-Invariant Space}, PGSIS for abbreviation,  if it is generated by a single graph signal. 

Given a finite family  $\Phi=\{\phi_1, \ldots, \phi_r\}$ of graph signals, we define the {\em graph fiber} $F$, a $N\times r$-matrix-valued function defined on the frequencies,  by
\begin{equation}\label{fiber.def}
    F({\pmb \gamma}(m))= \big[\widehat{\phi}(m) 
    \big]_{\phi\in\Phi}, 1\le m\le M.
\end{equation}
In Theorem \ref{fiber.thm} of Section \ref{fiber.section}, we show that the graph range function of the GSIS 
${\mathcal S}(\Phi)$
 coincides with the range space of the graph fiber associated with its generator $\Phi$. Furthermore, in Theorem \ref{FGSIS.thm} of Section \ref{fiber.section}, we prove that every GSIS is finitely generated. 
For a GSIS $U$, we may define
its {\em length}, to be denoted by $L(U)$,
as the smallest integer
$L\ge 1$ such that there exist nonzero signals $\phi_1, \ldots, \phi_L$ such that
$U={\mathcal S}(\phi_1, \ldots, \phi_L)$.
This definition is well-posed by Theorem \ref{FGSIS.thm}. Moreover, we show that the length of a GSIS $U$
equals the maximal bound of its dimension function; see 
Corollary \ref{length.cor}. 
The conclusion in Theorem \ref{FGSIS.thm} is also established in \cite{Chung2024} under the additional assumption that
 the graph shifts ${\mathbf S}_1, \ldots, {\mathbf S}_d$  satisfy \eqref{distincteigenvalue.assumption}.
 A similar fiber characterization result to Theorem \ref{FGSIS.thm} has been obtained in the  real line setting \cite{bownik2000, bownik2003, deBoor1994, ron1997, ron1995}. However, the finite-generator property for GSISs does not always hold in the real line case; cf.  the finite-generator property for locally finite-dimensional shift-invariant spaces on the line \cite{akram2002}.

In Section \ref{frame.section}, we investigate
 frame  property for the GSIS  ${\mathcal S}(\Phi)$ generated by $\Phi$. To this end, we normalize the graph shifts  ${\bf S}_1, \ldots, {\bf S}_d$ so that 
their eigenvalues  satisfy
\begin{equation}\label{frame.thm.eq0}
|\gamma_l(m)|< 1 \ \ {\rm for \ all} \  1\le m\le M\ {\rm and}\  1\le l\le d.
\end{equation}
In Theorems \ref{frame.thm} and \ref{frameoperator.thm} of Section \ref{frame.section}, we establish upper and lower bound estimates for the frame system ${\mathbf S}^{\pmb \alpha} \phi, {\pmb \alpha}\in {\mathbb Z}_+^d, \phi\in \Phi$,  and provide a necessary and sufficient condition
for the shift-invariance of the corresponding frame operator.
We  observe that, unlike in the real line setting where the frame bound is governed solely by the Grammian of the fiber associated with the generator, the frame bound in Theorem \ref{frame.thm} is expressed as the product of {\bf two} quantities: one determined by the Grammian of the graph fiber, and the other 
 depending on  the distribution of distinct eigenvalues of the graph shifts.
More notably,  the frame operator associated with a finite-generated SIS is {\bf always} shift-invariant in the real line setting, while this property typically fails in the graph setting; see Theorem \ref{frameoperator.thm} and  Proposition \ref{dualframeCompletegraph.pr}.

We use $\langle \cdot, \cdot\rangle$
to denote the standard inner product on ${\mathbb R}^V$
and also the one on ${\mathbb R}^N$ in the absence of notation.
 We say that  $\tilde\Phi=\{\tilde \phi_1, \ldots, \tilde \phi_r\}\in {\mathcal S}(\Phi)$ generates a {\em shift-invariant dual frame} if
\begin{equation}\label{dualframe.thm.eq1}
{\bf x}=\sum_{i=1}^r \sum_{{\pmb \alpha}\in {\mathbb Z}_+^d} \langle \x, {\bf S}^{\pmb \alpha}\tilde \phi_i\rangle {\bf S}^{\pmb \alpha} \phi_i
    \end{equation}
    hold for all ${\bf x}\in {\mathcal S}(\Phi)$.
 In Theorem \ref{dualframe.thm}, we characterize the shift-invariant dual frame via the fiber matrices associated with the generators
$\Phi$ and $\tilde \Phi$.
Define the Grammian of the graph filter $F$ in \eqref{fiber.def} by
\begin{equation}\label{Cm.def}{\bf C}(m)=(F({\pmb \gamma}(m)))^T  F({\pmb \gamma}(m))\in {\mathbb R}^{r\times r}, \end{equation}
and denote  the range space of ${\bf C}(m)$ by
$X(m), 1\le m\le M$.
In  Theorem \ref{dualframeexistence.thm}, we show that  a necessary and sufficient condition for the existence of a shift-invariant dual frame is that
the sum of $X(m), 1\le m\le M$, are {\bf direct}.
 As a  consequence, Corollary \ref{dualframeexistence.cor} shows that a necessary condition for such existence is that the dimension of the GSIS does not exceed the number of generators.
We note that, in the  real line setting,  a shift-invariant dual frame {\bf always} exists \cite{christensen2004, gabardo2007}. 
In contrast, this is {\bf not} always the case in the graph setting; see Proposition \ref{dualframeCompletegraph.pr} in Example \ref{uncertainty.example.part2} for a detailed analysis of the existence of shift-invariant dual frames on a complete graph.

In Section \ref{riesz.section}, we study  Riesz property of a shift-invariant subsystem of  the generating family
$\{{\bf S}^{\pmb\alpha}\phi\mid \pmb\alpha\in  {\mathbb Z}_+^d, \phi\in\Phi\}$ of the GSIS ${\mathcal S}(\Phi)$.
First, we show that the generating family forms a Bessel sequence and provide an upper bound estimate; see Theorem \ref{bessel.thm}. We then observe that the system is finitely linearly dependent and therefore cannot serve as a Riesz basis for the generating GSIS 
${\mathcal S}(\Phi)$.
On the other hand, Theorem \ref{FGSIS.thm} shows that the shift-invariant space 
${\mathcal S}(\Phi)$
 can be generated by 
 ${\bf T}^k\phi, 0\le k\le M-1, \phi\in \Phi$, where ${\bf T}$ is given in 
 \eqref{FGSIS.thm.eq1}.
 Building on this, Theorem \ref{riesz.thm} provides a necessary and sufficient condition under which the above finite system constitutes a Riesz basis for the GSIS 
 ${\mathcal S}(\Phi)$.

\subsection{Graph shift-invariant spaces are always finitely-generated}
\label{fiber.section}

In the following theorem, we provide
a characterization of
the  GSIS ${\mathcal S}(\Phi)$
via the  graph fiber
of its generators; see Section \ref{fiber.thm.pfsection} for the detailed proof. 

\begin{theorem}\label{fiber.thm} Let $\Phi$ be a finite family of graph signals. Then
 the range function $J_{{\mathcal S}(\Phi)}$ of the GSIS ${\mathcal S}(\Phi)$ is  the range space of the graph fibers, i.e., 
 \begin{equation} \label{fiber.thm.eq1}
 J_{{\mathcal S}(\Phi)} ({\pmb \gamma}(m))={\rm span} \big\{\widehat{{\phi}}(m)\ | \ \phi\in\Phi\big\}, \ 1\le m \le M.
 \end{equation}
 \end{theorem}

Denote the maximal and minimal positive singular value of a  matrix ${\bf A}$ by 
$\sigma_{\max}({\bf A})$ and
$\sigma_{\min}^+({\bf A})$ respectively.
By Theorems \ref{rangecharacterization.thm} and \ref{fiber.thm}, for any ${\bf x}\in {\mathcal S}(\Phi)$,
 there exist  coefficients $c_{\phi, m}, \phi\in \Phi, 1\le m\le M$,  such that
\begin{equation}\label{fiber.thm.eq3}
\widehat \x (m)=\sum_{\phi\in \Phi} c_{\phi, m} \widehat \phi(m)
\end{equation}
and
\begin{equation} \label{fiber.thm.eq4}
 \frac{ \|\widehat {\bf x}(m)\|_2}{ \sigma_{\max}(F({\pmb \gamma}(m))) } 
 \le \Big(\sum_{\phi\in \Phi} |c_{\phi, m}|^2\Big)^{1/2} \le \frac{ \|\widehat {\bf x}(m)\|_2}{\sigma_{\min}^+(F({\pmb \gamma}(m))) }. 
\end{equation}
We remark that  the maximal and  minimal positive singular values
$\sigma_{\max}
    (F({\pmb \gamma}(m)))$
and $\sigma_{\min}^+(F({\pmb \gamma}(m)))$ of the fiber matrix $F({\pmb \gamma}(m))$ at frequency ${\pmb \gamma}(m)$  also serve as the upper and lower frame bound for the fiber frame $\{\widehat \phi(m), \phi\in \Phi\}$ in its spanning subspace
$ J_{{\mathcal S}(\Phi)} ({\pmb \gamma}(m))$.
This representation establishes a quantitative  link between signals in the GSIS ${\mathcal S}(\Phi)$ 
 and their corresponding fibers in the Fourier domain at each frequency.

As a consequence of Theorem \ref{fiber.thm}, we see that
the  dimension function of the GSIS ${\mathcal S}(\Phi)$ is the same as the rank function of the graph fiber $F$ of its generator $\Phi$, i.e., 
\begin{equation}\label{fiberdimension.def}
\dim_{{\mathcal S}(\Phi)}({\pmb \gamma}(m))={\rm rank}\ F({\pmb \gamma}(m)), \ 1\le m\le M. 
\end{equation}

In the following theorem,  we show that every GSIS is finitely-generated; see Section \ref{FGSIS.thm.pfsection} for the detailed proof.

\begin{theorem}\label{FGSIS.thm}
Let  $U$ be linear space of graph signals,
and set
\begin{equation}\label{FGSIS.thm.eq1}
{\bf T}= a_1{\bf S}_1+\cdots+ a_d {\bf S}_d,
\end{equation}
where coefficients $a_l, 1\le l\le d$, are selected so that \eqref{FGSIS.thm.eq1-} holds.
Then the following statements are equivalent:
\begin{itemize}
\item[(i)]
The linear space $U$ is  shift-invariant.
 
 \item[(ii)] 
 The linear space $U$ is the GSIS ${\mathcal S}(\Phi)$
 generated by a  finite family $\Phi$
 of graph signals.

 \item[(iii)] The linear space $U$ is the GSIS 
 generated by a  finite family $\Phi$
 of graph signals with respect to the graph shift ${\bf T}$, i.e.,  
 $$U={\rm span}\{{\bf T}^k \phi\ | \ 0\le k\le M-1, \phi\in \Phi\}.$$
 \end{itemize}
\end{theorem}

The conclusion  in Theorem \ref{FGSIS.thm} is established in \cite{Chung2024}, where it is assumed that the condition \eqref{distincteigenvalue.assumption} is satisfied.

For a GSIS $U$, define 
\begin{equation}\label{maximaldimension.def} {\rm Maxdim}(U)=\max_{1\le m\le M} \dim_U({\pmb \gamma}(m)),\end{equation}
the  maximal bound of its dimension function. From the proof of Theorem 
\ref{FGSIS.thm}, we see that  for any given GSIS $U$, 
 there exist generators $\Phi=\{\phi_1, \ldots, \phi_r\}$ with $r={\rm Maxdim}(U)$
 such that $U={\mathcal S}(\Phi)$. On the other hand, if
 $U= {\mathcal S}(\Phi)$, it follows from \eqref{fiber.thm.eq1} that
the cardinality $\Phi$  
is no less than  ${\rm Maxdim}(U)$.
Consequently, we have the following result about the length of a GSIS. 

\begin{corollary} \label{length.cor} Let  $U$ be a GSIS,  and 
denote 
 its dimension function and length by $\dim_U$ and $L(U)$ respectively. Then 
\begin{equation}\label{length.cor.eq1}
L(U)= \max_{1\le m\le M} \dim_U({\pmb \gamma}(m)).
\end{equation}
\end{corollary}

\noindent and 

\begin{corollary}\label{wholespacesis.cor}
The space ${\mathbb R}^V$ is a  PGSIS if and only if 
 the graph
shifts ${\bf S}_1, \ldots, {\bf S}_d$  satisfies the distinct joint spectrum condition \eqref{distincteigenvalue.assumption}.
\end{corollary}

  \smallskip

\subsection{Frame property for shift-invariant systems}
\label{frame.section}

In the following theorem, we provide a frame bound estimate for the GSIS generated by $\Phi$; see Section \ref{frame.thm.pfsection} for the detailed proof.

\begin{theorem}\label{frame.thm}
Let ${\bf S}_1, \ldots, {\bf S}_d$ be the graph shifts that satisfy  Assumption \ref{graphshift.assumption}
and the normalization condition \eqref{frame.thm.eq0}. 
Set 
\begin{equation}\label{frame.thm.eq1}
{\bf A}=\left[\prod_{l=1}^d 
\big(1-\gamma_l(m)\gamma_l(m')\big)^{-1}\right]_{1\le m, m'\le M},
\end{equation}
where ${\pmb \gamma}(m)=[\gamma_1(m), \ldots, \gamma_d(m)]^T, 1\le m\le M$.
Let $\Phi=\{\phi_1,\ldots, \phi_r\}$,  ${ F}$
be its graph fiber in  
\eqref{fiber.def}, and 
$\mathcal S(\Phi)$ be the GSIS generated by $\Phi$.
Then 
\begin{eqnarray*} 
& &  \sigma_{\min}^+({\bf A})\Big(\inf_{1\le m\le M} \sigma_{\min}^+\big(F({\pmb \gamma}(m)\big)
\Big)  \|\x \|_2^2  \le     
\sum_{\phi\in \Phi} \sum_{{\pmb \alpha}\in {\mathbb Z}_+^d} |\langle \x, {\bf S}^{\pmb \alpha}\phi\rangle|^2\nonumber\\
&  & \qquad\le   \sigma_{\max}({\bf A})\Big(\sup_{1\le m\le M} \sigma_{\max} \big(F({\pmb \gamma}(m)\big) 
\Big)  \|\x \|_2^2 
\end{eqnarray*}
hold for all $\x\in {\mathcal S}(\Phi)$.
\end{theorem}

Denote the {\em frame operator} by
\begin{equation}\label{frameoperator.def}
    S\x=\sum_{\phi\in \Phi} \sum_{\pmb\alpha\in {\mathbb Z}_+^d}\langle\x, {\bf S}^{\pmb\alpha}\phi\rangle {\bf S}^{\pmb\alpha}\phi, \quad \x\in {\mathcal S}(\Phi). 
\end{equation}
By  Theorem \ref{frame.thm}, we have
\begin{equation}
\sigma_{\min}^+({\bf A})\Big(\inf_{1\le m\le M} \sigma_{\min}^+\big(F({\pmb \gamma}(m)\big)
\Big)  {\bf I} \preceq   S\preceq   
\sigma_{\max}({\bf A})\Big(\sup_{1\le m\le M} \sigma_{\max} \big(F({\pmb \gamma}(m)\big) 
\Big)  {\bf I}.
\end{equation}
We remark that, unlike in the  real line setting, the frame operator 
$S$
 in \eqref{frameoperator.def} does {\bf not} always commute with the graph shifts 
 ${\mathbf S}_1, \ldots, {\mathbf S}_d$.
The possible reasons could be  that, in the  real line setting, the unit shift operator is both invertible and isometric, and the summation in the frame operator 
is taken over all shifts in the {\bf group} ${\mathbb Z}^d$. In contrast,  in the graph setting, 
 graph shifts are not necessarily invertible and, in general, cannot be assumed to be isometric, and moreover  the summation
 in \eqref{frameoperator.def} is taken over shifts in the
 {\bf semi-group} 
 ${\mathbb Z}_+^d$. In the following theorem, we provide a characterization for the shift-invariance of the frame operator $S$ in \eqref{frameoperator.def}; see Section \ref{frameoperator.thm.pfsection} for the detailed proof.

 \begin{theorem}\label{frameoperator.thm}
 Let  $\Phi=\{\phi_1, \ldots, \phi_r\}$,  ${\bf S}_1, \ldots, {\bf S}_d$ be the graph shifts that satisfy  Assumption \ref{graphshift.assumption}
and the normalization condition \eqref{frame.thm.eq0}, 
 and ${\bf C}(m), 1\le m\le M$, be the Grammian  of the  graph fiber   defined as in   \eqref{Cm.def}.
The  the frame operator  $S$  in \eqref{frameoperator.def} is  shift-invariant if and only if
 \begin{equation}\label {frameoperator.thm.eq1}
 {\bf C}(m) {\bf C}(m')=0 \ \ {\rm for \ all} \ 1\le m\ne m'\le M.
 \end{equation}
 \end{theorem}

 By \eqref{dimensionU.eq}, \eqref{fiberdimension.def}
and Theorem \ref{frameoperator.thm}, we have the following necessary condition for the shift-invariance of the frame operator in \eqref{frameoperator.def}, which  is {\bf not} required in the  real line setting, cf. Corollary \ref{dualframeexistence.cor} for a necessary condition of the existence of  shift-invariant dual frames.

\begin{corollary}
\label{frameoperator.cor}
     Let  $\Phi=\{\phi_1, \ldots, \phi_r\}$, and ${\bf S}_1, \ldots, {\bf S}_d$ be the graph shifts that satisfy  Assumption \ref{graphshift.assumption} and the normalization condition \eqref{frame.thm.eq0}.
     Then a necessary condition for the 
     shift-invariance of the frame operator $S$ in \eqref{frameoperator.def} is
     $$\dim {\mathcal S}(\Phi)\le r.$$
\end{corollary}

\begin{remark}
{\rm 
By the classical Cayley-Hamilton theorem, for every $1\le l\le d$, ${\bf S}_l^M$ is the linear combination of ${\bf S}_l^m, 0\le m\le M-1$. Therefore the GSIS generated by $\Phi$ is given by 
\begin{equation} \label{FGSIS.def2}
    \mathcal S(\Phi)=\span\{{\bf S}^{\pmb \alpha}\phi\mid {\bf 0}\le {\pmb \alpha}\le (M-1) {\bf 1}, \phi\in\Phi\}.
\end{equation}
Here ${\bf 0}\le {\pmb \alpha}\le (M-1) {\bf 1}$ means that all components of indices ${\pmb \alpha}$ are integers between $0$ and $M-1$.
Following the argument in Theorem \ref{frame.thm.eq1}, we have a similar frame bound estimate 
 with the matrix  ${\bf A}$ in \eqref{frame.thm.eq1} by 
$$A_M=\left[\prod_{l=1}^d \frac{1-\gamma_l(m)^M \gamma_l(m')^M}{1-\gamma_l(m)\gamma_l(m')}\right]_{1\le m, m'\le M}.$$
In particular,  for any  $\x\in {\mathcal S}(\Phi)$, we have
\begin{eqnarray}\label{frame.thm.ramarkeq1}
 & & \sigma_{\min}^+({\bf A}_M)\Big(\inf_{1\le m\le M} \sigma_{\min}^+\big(F({\pmb \gamma}(m)\big)
\Big)  \|\x \|_2^2
 \le 
\sum_{\phi\in \Phi} \sum_{ 0\le {\pmb \alpha}\le (M-1) {\bf 1}} |\langle \x, {\bf S}^{\pmb \alpha}\phi\rangle|^2\nonumber\\
&  & \qquad\le   \sigma_{\max}({\bf A}_M)\Big(\sup_{1\le m\le M} \sigma_{\max} \big(F({\pmb \gamma}(m)\big) 
\Big)  \|\x \|_2^2.
\end{eqnarray}
Define the frame operator
associated with the generating system in 
\eqref{FGSIS.def2} by
$$S_M{\bf x}=\sum_{\phi\in \Phi}\sum_{{\bf 0}\le {\pmb \alpha}\le (M-1){\bf 1}} \langle {\bf x}, {\bf S}^{\pmb \alpha} \phi\rangle {\bf S}^{\pmb \alpha} \phi,\ \  {\bf x}\in {\mathcal S}(\Phi).$$
One may follow the argument used in Theorem \ref{frameoperator.thm} and show that the frame operator ${\bf S}_M$ just-defined  is shift-invariant if and only if \eqref{frameoperator.thm.eq1} holds.
}
\end{remark}

In the following theorem, we provide a characterization for a shift-invariant dual frame;
see Section \ref{dualframe.thm.pfsection} for the detailed proof. 

\begin{theorem}\label{dualframe.thm}
 Let ${\bf S}_1, \ldots, {\bf S}_d$ be the graph shifts that satisfy \eqref{frame.thm.eq0} and  Assumption \ref{graphshift.assumption}.
Let $\Phi=\{\phi_1,\ldots, \phi_r\}$ and $\tilde \Phi=\{\tilde \phi_1, \ldots, \tilde \phi_r\}\in S(\Phi)$, and set
\begin{equation}\label{dualframe.thm.eq0} {\bf B}(m)=\big[\big\langle \widehat \phi_i(m), \widehat {\tilde \phi}_j(m)\big\rangle\big]_{1\le i, j\le r}.
\end{equation} 
Then  $\tilde \Phi$ generates a shift-invariant dual frame 
if and only if
 \begin{equation}\label{dualframe.thm.eq2}   S(\tilde \Phi)= S(\Phi),
 \end{equation}
and
\begin{equation} \label{dualframe.thm.eq3}
    {\bf B}(m) {\bf B}(m')=\left\{
    \begin{array}{ll} {\bf O}  & {\rm if}\ m\ne m'\\
    \prod_{l=1}^d [(1-\gamma_l(m)^2)] {\bf B}(m) & {\rm if}
    \ m'=m.
    \end{array}\right. 
\end{equation}
\end{theorem}

Let  the Grammian  ${\bf C}(m), 1\le m\le M$, of the graph filter $F$ be as in \eqref{Cm.def},
and $X(m)$ be  the range space of
${\bf C}(m), 1\le m\le M$.
By Theorem \ref{fiber.thm}, we have
\begin{equation}\label{Xm.def}
X(m)=\Big\{\big[\big\langle\widehat \phi_i(m), \widehat \x(m)\big\rangle\big]_{1\le i\le r}\ \big | \ {\x}\in {\mathcal S}(\Phi)\Big\}
\end{equation}
and
\begin{equation} \label{Xm.eq1}
\dim X(m)=\dim_{{\mathcal S}(\Phi)}({\pmb \gamma}(m)), \quad  1\le m\le M.
\end{equation}
With the help of the characterization of shift-invariant dual frames in 
Theorem \ref{dualframe.thm}, we have the following result about the existence of shift-invariant dual frames; see Section \ref{dualframeexistence.thm.pfsection} for the detailed proof. 

\begin{theorem}\label{dualframeexistence.thm}
 Let ${\bf S}_1, \ldots, {\bf S}_d$ be graph shifts that satisfy \eqref{frame.thm.eq0} and  Assumption \ref{graphshift.assumption}.
Let $\Phi=\{\phi_1,\ldots, \phi_r\}$  and $X(m), 1\le m\le M$,
be as in \eqref{Xm.def}.
Then there exists $\tilde \Phi=\{\tilde \phi_1, \ldots, \tilde \phi_r\}\subset {\mathcal S}(\Phi)$ that generates a shift-invariant dual frame if
and only if the sum of $X(m), 1\le m\le M$, is direct.  
\end{theorem}

By \eqref{dimensionU.eq}, \eqref{Xm.eq1} and Theorem \ref{dualframeexistence.thm}, we have  the following necessary condition for the existence of shift-invariant dual frame, 
which  is {\bf not} required in the  real linear setting, cf. Corollary \ref{frameoperator.cor} for a necessary condition
for the shift-invariance of the frame operator in \eqref{frameoperator.def}.

\begin{corollary} \label{dualframeexistence.cor}
 Let  $\Phi=\{\phi_1,\ldots, \phi_r\}$ and ${\bf S}_1, \ldots, {\bf S}_d$ be the graph shifts that satisfy \eqref{frame.thm.eq0} and  Assumption \ref{graphshift.assumption}. Then a necessary condition for the existence of a shift-invariant dual frame for the shift-invariant space ${\mathcal S}(\Phi)$ is
  \begin{equation} \label{dualframeexistence.cor.eq1}
  \dim {\mathcal S}(\Phi)\le r.
  \end{equation}
\end{corollary}

For a PGSIS  ${\mathcal S}(\phi)$ generated by a single nonzero graph signal $\phi$, we obtain the following result by Corollary \ref{dualframeexistence.cor}. 

\begin{corollary}\label{dualframe.PGSIS.cor}
Let   ${\mathcal S}(\phi)$ be a PGSIS. Then it has a shift-invariant dual frame if and only if it is one-dimensional, in other words, 
$\widehat \phi(m)\ne 0$ for only one $1\le m\le M$.
\end{corollary}

Given a shift-invariant space generated by  $\Phi$, we remark that the existence of a shift-invariant dual frame is closely related to the shift-invariance of the frame operator $S$ in \eqref{frameoperator.def}. Specifically, it follows from Theorem \ref{frameoperator.thm} that the frame operator $S$  is shift-invariant if and only if the subspaces $X(m),  1 \le m \le M$, in \eqref{Xm.def} are mutually orthogonal, which is stronger than the direct sum requirement imposed  in Theorem \ref{dualframeexistence.thm} on those subspaces. 
In the following example, we provide a detailed analysis of the shift-invariance of the frame operator and the existence of a shift-invariant dual frame for a finitely generated GSIS on a complete graph.

\smallskip

\begin{example}\label{uncertainty.example.part2} {\rm (Continuation of Example \ref{uncertainty.example})\quad  Let $\Phi=\{\phi_1, \ldots, \phi_r\}$ be a family of graph signals on the complete graph $K_N$. As the spectrum of the symmetric Laplacian ${\bf L}^{\rm sym}(K_N)$ has only two different eigenvalues $0$ and $N/(N-1)$. 
 By \eqref{uncertainty.example.eq2} and \eqref{Xm.def}, we have
\begin{equation} \label{uncertainty.example.part2.eq1}
    X(1)=\Big\{\big[\langle \widehat\phi_i(1), \widehat\x(1)\rangle\big]_{1\le i\le r} \Big| 
    \ \x\in\mathcal S(\Phi)\Big\}=\span\big\{[\bar{\phi}_1,\ldots, \bar{\phi}_r]^T\big\},
\end{equation}
and 
\begin{eqnarray} \label{uncertainty.example.part2.eq2}
    X(2) & = & \Big\{\big[\langle \widehat\phi_i(2), \widehat\x(2)\rangle\big]_{1\le i\le r}\Big| \ \x\in\mathcal S(\Phi)\Big\}\nonumber\\
    & = & \span\Big\{\big[\langle \phi_i-\bar \phi_i {\bf 1},\phi_{j}-\bar \phi_j {\bf 1} \rangle \big]_{1\le i\le r}\ \Big|\ 1\le j\le r\Big\}, 
\end{eqnarray}
where we denote the average of $\phi_i$ by $\bar \phi_i=\langle \phi_i, {\bf 1}\rangle/N, 1\le i\le r$.
By Theorem  \ref{dualframeexistence.thm}, a shift-invariant dual frame exists when the spectrum of the GSIS ${\mathcal S}(\Phi)$ contains only one frequency. In particular, one may verify that $\dim X(1)=0$ if and only if $\bar \phi_i=0, 1\le i\le r$, while $\dim X(2)=0$
if and only if $\phi_i, 1\le i\le r$ are parallel to the vector ${\bf 1}$, or equivalently  $\phi_i=c_i{\bf 1}$ for some numbers $c_i\in {\mathbb R}, 1\le i\le r$.  

Now let us consider the existence of shift-invariant dual frame when both $X(1)$ and $X(2)$ are nontrivial.  
By Theorem \ref{dualframeexistence.thm}, it suffices to verify whether
$X(1)\cap X(2)$ is trivial or not.
Then  by \eqref{uncertainty.example.part2.eq1} and the observation $\langle \widehat {\bf x}(2), {\bf 1}\rangle=0$ for all $\x\in {\mathcal S}(\Phi)$, the problems reduces  to whether we can find ${\bf z}\in {\rm span} \{\phi_1-\bar \phi_1 {\bf 1}, \ldots, \phi_r-\bar \phi_r {\bf 1}\}$ such that
\begin{equation}  \label{uncertainty.example.part2.eq3}
\langle \phi_i, {\bf z}\rangle=\bar \phi_i,\quad  1\le i\le r.
\end{equation}
Let $\psi_1, \ldots, \psi_s$ be an orthonormal basis for the space
spanned by $\phi_1, \ldots, \phi_r$. Then we reformulate
\eqref{uncertainty.example.part2.eq3}
as follows:  
\begin{equation}  \label{uncertainty.example.part2.eq4}
\langle \psi_i, {\bf z}\rangle=\bar \psi_i,\quad  1\le i\le s.
\end{equation}
Write ${\bf z}=\sum_{i=1}^s c_i (\psi_i- \bar \psi_i {\bf 1}\rangle $. Then the equation  \eqref{uncertainty.example.part2.eq4} can be rewritten as 
a linear system, 
\begin{equation} \label{uncertainty.example.part2.eq5}
c_i- N\bar \psi_i \Big(\sum_{j=1}^s c_j \bar\psi_j\Big)=\bar \psi_i, \quad 1\le i\le s.
\end{equation}
If $ N\sum_{i=1}^s \bar \psi_i^2\ne 1$, then
we see that $c_i=\bar \psi_i (1-N \sum_{i=1}^s \bar \psi_i^2)^{-1}, 1\le i\le s$, is a  nonzero solution of the linear system
\eqref{uncertainty.example.part2.eq5} and hence $X(1)\cap X(2)$ is nontrivial. 
On the other hand,  if $ N\sum_{i=1}^s \bar \psi_i^2=1$,
multiplying $\bar \psi_i$ at both sides of the equation \eqref{uncertainty.example.part2.eq5} and sum up over all $1\le i\le s$ yields
$$\Big(1-N  \sum_{i=1}^s\bar \psi_i^2\Big) \sum_{j=1}^s c_j \bar\psi_j =\sum_{i=1}^s\bar \psi_i^2, $$
which is contradiction as $\sum_{i=1}^s \bar \psi_i^2\ne 0$
by the assumption $\dim X(1)\ne 0$.
On the other hand, we obtain from by the orthonormal property for $\psi_i, 1\le i\le s$, that
$$\sum_{i=1}^s \bar \psi_i^2=N^{-2}\sum_{i=1}^s |\langle \psi_i, {\bf 1}\rangle|^2\le N^{-1}
$$
and the equality hold if and only if ${\bf 1}\in {\rm span}\{\psi_1, \ldots, \psi_s\}={\rm span} \{\phi_1, \ldots, \phi_r\}$.

By similar argument, we can show that the linear subspaces
$X(1)$ and $X(2)$ in \eqref{uncertainty.example.part2.eq1} and
\eqref{uncertainty.example.part2.eq2}
are orthogonal if and only if
$$\sum_{j=1}^r \langle \phi_i-\bar{\phi_i}{\bf 1},
 \phi_j-\bar{\phi_j}{\bf 1}\rangle \bar \phi_j=0, \ \ 1\le i\le r,
$$
if and only if 
$$\Big\langle \phi_i, \sum_{j=1}^r \bar\phi_j \phi_j -\Big(\sum_{j'=1}^r (\bar \phi_{j'})^2\Big){\bf 1}\Big\rangle=0, \ \ 1\le i\le r,$$
if and only if 
\begin{equation}
\label{completegraph.orthogonalsum}
\sum_{j=1}^r  \bar {\phi_j} \phi_j = \Big(\sum_{j=1}^r  
 \big(\bar{\phi_j}\big)^2\Big) {\bf 1}.
\end{equation}
Therefore we conclude that 

\begin{proposition}\label{dualframeCompletegraph.pr}
Let $\Phi=\{\phi_1, \ldots, \phi_r\}$ be a family of graph signals on the complete graph $K_N$.
Then the following statements hold.
\begin{itemize}
    \item[{(i)}] There exists a shift-invariant dual frame for the GSIS ${\mathcal S}(\Phi)$ if and only if either ${\bf 1}\in {\rm span} \{\phi_1, \ldots, \phi_r\}$  or $\bar \phi_i=0$ for all $1\le i\le r$.

    \item[{(ii)}] The frame operator $S$ in \eqref{frameoperator.def} is shift-invariant if and only if \eqref{completegraph.orthogonalsum} hold.

\end{itemize}  
\end{proposition}
}
\end{example}

\subsection{Riesz property for finitely-generated shift-invariant spaces}
\label{riesz.section}

In the following theorem, we show that
$\{{\bf S}^{\pmb\alpha}\phi \ | \ \pmb\alpha\in {\mathbb Z}_+^d, \phi\in\Phi\}$ forms a Bessel sequence for the GSIS $S(\Phi)$; see Section \ref{bessel.thm.pfsection} for the detailed proof.

\begin{theorem}\label{bessel.thm}
Let ${\bf S}_1, \ldots, {\bf S}_d$ be the graph shifts satisfying Assumption \ref{graphshift.assumption} and
the normalization condition 
\eqref{frame.thm.eq0}, and $\Phi$ be a finite family of graph signals,  Then 
 $\{{\bf S}^{\pmb\alpha}\phi \ | \ \pmb\alpha\in {\mathbb Z}_+^d, \phi\in\Phi\}$ forms a Bessel sequence for the  GSIS $\mathcal S(\Phi)$, i.e.,
\begin{equation}\label{bessel.thm.eq2}
\Big\|\sum_{\phi\in \Phi} \sum_{{\pmb \alpha}\in {\mathbb Z}_+^d} c_{\pmb\alpha,\phi}{\bf S}^{\pmb\alpha}\phi\Big\|_2\le 
\Bigg( \sum_{m=1}^M  \frac{\big(\sigma_{\max}(F({\pmb \gamma}(m))\big )^2}
{\prod_{l=1}^{d} (1-\gamma_l(m)^2)}\Bigg)^{1/2}
 \Big(\sum_{\phi\in \Phi} \sum_{{\pmb \alpha}\in {\mathbb Z}_+^d} |c_{{\pmb \alpha}, \phi}|^2\Big)^{1/2}
\end{equation} 
hold for all sequences $c=[c_{{\pmb \alpha}, \phi}]_{{\pmb \alpha}\in {\mathbb Z}_+^d, \phi\in \Phi}$,
where $\sigma_{max}(F({\pmb \gamma}(m)))$
is the maximal singular value of the fiber matrix $F({\pmb \gamma}(m)), 1\le m\le M$ of the generator $\Phi$.
\end{theorem}

Let ${\bf T}$ be the graph shift so chosen that \eqref{FGSIS.thm.eq1} holds. 
By the classical Cayley-Hamilton theorem, for every $1\le l\le d$, ${\bf T}^M$ is the linear combination of ${\bf T}^m, 0\le m\le M-1$. Therefore
${\bf S}^{\pmb\alpha}\phi,  \pmb\alpha\in {\mathbb Z}_+^d, \phi\in\Phi$, are linearly dependent and cannot form a Riesz basis for the GSIS ${\mathcal S}(\Phi)$. However we have
\begin{equation}\label{riesz.thm.eq0}
{\mathcal S}(\Phi)={\rm span} \{ {\bf T}^k \phi\ |\  0\le k\le M-1, \phi\in \Phi\}
\end{equation}
by Theorem \ref{FGSIS.thm}.
In the following theorem, we provide a necessary and sufficient condition  on the Riesz  of  the generator $\Phi$; see Section \ref{riesz.thm.pfsection} for the detailed proof.

\begin{theorem}\label{riesz.thm}
Let ${\bf T}$ be the graph shift so chosen  that \eqref{FGSIS.thm.eq1} holds.  Then
$\{{\bf T}^k \phi\ |\  0\le k\le M-1, \phi\in \Phi\}$ form a Riesz basis for the GSIS  ${\mathcal S}(\Phi)$ if and only if
\begin{equation}\label{riesz.thm.eq1}\dim_{{\mathcal S}(\Phi)}({\pmb \gamma}(m))=r \  {\rm for \ all}\  1\le m\le M.\end{equation}
\end{theorem}

By  Theorem \ref{fiber.thm}, we see that an equivalent statement for the requirement
 \eqref{riesz.thm.eq1} is that 
$\widehat \phi_1(m), \ldots, \widehat \phi_r(m)$ is a Riesz basis for the space $W_m\cap {\mathcal S}(\Phi)$  for every $1\le m\le M$. Therefore there exist positive numbers $r_0(m)$ and $r_1(m)$ such that
\begin{equation}\label{riesz.thm.rem1}
r_0(m)\Big(\sum_{l=1}^r|c_l|^2\Big)^{1/2}
\le \Big\|\sum_{l=1}^r c_l \widehat \phi_l(m)\Big\|_2\le 
r_1(m)\Big(\sum_{l=1}^r|c_l|^2\Big)^{1/2}
\end{equation}
hold for all sequence $\{c_l\}_{l=1}^r$, cf. 
the frame representation \eqref{fiber.thm.eq3}
and \eqref{fiber.thm.eq4}
of signals in $W_m\cap {\mathcal S}(\Phi)$.

For the graph shift ${\bf T}$ in  Theorem \ref{riesz.thm}, we define
the Vandermonde matrix
\begin{equation}
\label{riesz.thm.rem2} {\bf V}:=[(\mu(m))^{k}]_{1\le m\le M, 0\le k\le M-1}.\end{equation}
Then ${\bf V}$ is nonsingular by \eqref{FGSIS.thm.eq1}.
From the proof of Theorem \ref{riesz.thm}, we have the following Riesz bound estimate:
\begin{eqnarray}  \label{riesz.thm.rem3}
 & &  \sigma_{\min}({\bf V}) \Big(\min_{1\le m\le M}  r_0(m)\Big)
\Big(\sum_{k=0}^{M-1} \sum_{l=1}^r |c_{k,l}|^2\Big)^{1/2}
\le 
\Big\|\sum_{k=0}^{M-1} \sum_{l=1}^r c_{k,l} {\bf T}^k \phi_l\Big\|_2\nonumber\\
& & \qquad \le 
\sigma_{\max}({\bf V}) \Big(\max_{1\le m\le M}  r_1(m)\Big)
\Big(\sum_{k=0}^{M-1} \sum_{l=1}^r |c_{k,l}|^2\Big)^{1/2}
\end{eqnarray}
holds for all sequences $\{c_{k,l}\}_{0\le k\le M-1, 1\le l\le r}$, where
$\sigma_{\max}({\bf V})$ and $\sigma_{\min}({\bf V})$
are the maximal and minimal singular value of the matrix ${\bf V}$ in \eqref{riesz.thm.rem2} respectively.

We finish this  section with the Riesz basis property for a GSIS on a complete graph. 

\begin{example}\label{uncertainty.example.part3} {\rm (Contiuation of Examples \ref{uncertainty.example} and \ref{uncertainty.example.part2})\quad  Let $\Phi=\{\phi_1, \ldots, \phi_r\}$ be a family of graph signals on the complete graph $K_N$.  
 By \eqref{uncertainty.example.part2.eq1},  \eqref{uncertainty.example.part2.eq2}
 and  Theorem \ref{riesz.thm},
 a sufficient and necessary condition for  $\Phi$ and ${\bf L}^{\rm sym}(K_N)\Phi$ form a Riesz basis for the GSIS ${\mathcal S}(\Phi)$ is that
 $$\dim X(1)=\dim X(2)=r=1.$$
 In the other words, $r=1$ and the sole generator $\phi_1$ satisfies $\bar\phi_1\ne 0$ and
 $\phi_1\ne \bar\phi_1 {\bf 1}$.
}
\end{example}

\section{Graph reproducing kernel Hilbert spaces with shift-invariant reproducing kernels}
\label{rkhs.section}

Let ${\mathcal G}=(V, E, {\bf W})$ be a undirected simple graph of order $N\ge 2$, and ${\bf S}_1, \ldots, {\bf S}_d$ be graph shifts satisfying Assumption \ref{graphshift.assumption}.
Let $H$ be a Hilbert space of graph signals on the graph ${\mathcal G}$ and denote its inner product and norm  by $\langle \cdot, \cdot\rangle_H$ and $\|\cdot\|_H$ respectively.
We say that the Hilbert space $H$  is a {\em Graph Reproducing Kernel Hilbert Space} (GRKHS) if for every vertex $v\in V$ there exists a positive constant $C_v$ such that
\begin{equation}
|x(v)|\le C_v\|{\bf x}\|_H \ \ {\rm hold \ for \ all}  \ {\bf x}=[x(v)]_{v\in V}\in H.
\end{equation}
  For a GRKHS, there exists a reproducing kernel ${\bf K}=[K(v, v')]_{v, v'\in V}$ such that
$$x(v)=\langle {\bf x}, K(\cdot, v)\rangle_{H} \ \
 {\rm hold \ for \ all}  \ {\bf x}=[x(v)]_{v\in V}\in H \ {\rm and} \ v\in V.$$
  For a GRKHS  $H$ with the reproducing kernel ${\bf K}$, we can express any graph signal in $H$ as a  linear combination of the columns of ${\bf K}$, i.e.,
\begin{equation}\label{SIGRKHS.def}
H=\big\{{\bf Kc}\ |\ {\bf c}\in \R^V\big\}.
\end{equation}

RKHSs on the real line have been widely adopted in kernel-based learning for function estimation \cite{scholkopf2002,  shalev2014,  zhang2009}. For efficient learning of signals in a GRKHS  on an undirected graph, the reproducing kernels are typically chosen to be shift-invariant \cite{nikolentzos2021, rmero2017,  seto2014,  smola2003}.
Here we  say that the reproducing kernel ${\bf K}$ of a GRKHS $H$ is  {\em shift-invariant} if
\begin{equation} \label{rksis.def}
{\bf K} {\bf S}_l={\bf S}_l {\bf K} \ \ {\rm hold \  for \ all}\ 1\le l\le d,
\end{equation}
and that $H$ is a {\em Shift-Invariant GRKHS} of graph signals, SIGRKHS for abbreviation, if its reproducing kernel is shift-invariant. 
By \eqref{SIGRKHS.def}, an SIGRKHS $H$ is the range space of a shift-invariant filter and hence it is shift-invariant, i.e.,
\begin{equation}
{\bf S}_l {  H}\subset { H} \ {\rm hold \ for \ all}\ 1\le l\le d.
\end{equation}
In the following theorem, we show that the converse holds as well; see Section
\ref{rkhs.thm1.pfsection} for the detailed proof.

\begin{theorem} \label{rkhs.thm1}
Let $U$
 be a linear space of graph signals. Then 
$U$
 is shift-invariant if and only if it is an SIGRKHS  embedded with the standard inner product in 
${\mathbb R}^V$.
\end{theorem}

A GSIS may not be bandlimited. In the following theorem, we provide another characterization for  bandlimited spaces via GRKHSs with reproducing kernel being  polynomial filters; see Section \ref{rkhs.bandlimited.thm.pfsection} for the detailed proof.

\begin{theorem}\label{rkhs.bandlimited.thm}
Let $U$ be a linear space of graph shifts.  Then $U$ is a bandlimited space if and only if it is 
an SIGRKHS  embedded with the standard inner product in 
${\mathbb R}^V$ and the reproducing kernel being a polynomial filter. 
\end{theorem}

Under the  additional assumption that \eqref{distincteigenvalue.assumption} holds, 
 the conclusions in Theorems \ref{rkhs.thm1}  and 
\ref{rkhs.bandlimited.thm}  become equivalent by Corollaries \ref{shiftinvariantfilter.cor} and \ref{bandlimited.cor}, which has been proved in \cite{Chung2024}.

 In the following theorem, we show that the inner product of any SIGRKHS can be defined by a generalized inner product in the Fourier domain; see Section \ref{decomposition.thm.pfsection} for the detailed proof.

\begin{theorem}\label{decomposition.thm}
Let $H$ be a GRKHS of graph signals. Then  $H$ 
is an SIGRKHS  if and only if there exist
 GRKHSs $H_m\subset W_m$ with inner
product $\langle \cdot, \cdot\rangle_{H_m}, 1\le m\le M,$ such that
\begin{equation}\label{decomposition.thm.eq1}
H=\bigoplus_{m=1}^M H_m
\end{equation}
and
\begin{equation}\label{decomposition.thm.eq2}
\langle {\bf x}, {\bf y}\rangle_H=\sum_{m=1}^M \langle \widehat {\bf x}(m), \widehat {\bf y}(m)\rangle_{H_m}\ {\rm for \ all} \ {\bf x}, {\bf y}\in H.\end{equation}
\end{theorem}

From the proof of Theorem \ref{decomposition.thm}, given a SIGRKHS $H$, we notice that the  GRKHSs $H_m$ and their associated inner product can be selected as follows: 
\begin{equation}\label{decomposition.thm..eq5} H_m=H\cap W_m \ {\rm and} \ \langle {\bf x}, {\bf y}\rangle_{H_m}= \langle {\bf x}, {\bf y}\rangle_H \ {\rm for \ all} \ {\bf x}, {\bf y}\in H_m, 1\le m\le M.\end{equation}

From the proof of Theorem \ref{decomposition.thm}, we have the following result when  the RKHS has a reproducing kernel being a polynomial filter.

\begin{theorem}
\label{decomposition.polynomial.thm}
Let $H$ is SIGRKHS with the reproducing kernel ${\bf K}$ being a polynomial filter if and only if
there exist 
 GRKHSs $H_m\subset W_m, 1\le m\le M$
with inner
product being a multiple of standard Euclidean dot product 
such that
\begin{equation}\label{decomposition.polynomial.thm.eq1}
H=\bigoplus_{m=1}^M H_m
\end{equation}
and
\begin{equation}\label{decomposition.polynomialthm.eq2}
\langle {\bf x}, {\bf y}\rangle_H=\sum_{m=1}^M  
\langle \widehat {\bf x}(m), \widehat {\bf y}(m)\rangle_{H_m}=\sum_{m=1}^M
\mu_m 
 (\widehat {\bf x}(m))^T \widehat {\bf y}(m)\ \ {\rm for \ all} \ {\bf x}, {\bf y}\in H,\end{equation}
where  $\mu_m\ge 0, 1\le m\le M$, are  nonnegative numbers.
    \end{theorem}

As a consequence of Theorem \ref{decomposition.polynomial.thm}, we recover a conclusion in \cite{Chung2024} on the inner product of a 
SIGRKHS.

\begin{corollary}
Let $H$ be a SIGRKHS and assume that the distinct joint spectrum requirement \eqref{distincteigenvalue.assumption}
is satisfied. Then 
there exist nonnegative numbers $\mu_m, 1\le m\le M$, such that
\begin{equation}
\langle {\bf x}, {\bf y}\rangle_H=\sum_{m=1}^M \mu_m 
(\widehat {\bf x}(m))^T \widehat {\bf y}(m)\ {\rm for \ all} \ {\bf x}, {\bf y}\in H.
\end{equation}

\end{corollary}

We finish this section with the characterization to 
an isometric shift-invariant linear operator $T$ between two  SIGRKHSs; see Section \ref{rkhsisometric.thm.pfsection} for the detailed proof and  cf. Proposition \ref{sioperator.prop}. 

\begin{theorem}\label{rkhsisometric.thm} Let $H_1$ and $H_2$ be SIGRKHSs, and $T$ be a shift-invariant linear operator from $H_1$ to $H_2$, represented by a shift-invariant filter. Then  $T$ is isometric,  i.e., 
\begin{equation}
\|T{\bf x}\|_{H_2}=\|{\bf x}\|_{H_1} \ {\rm for \ all} \ {\bf x}\in H_1, 
\end{equation}
if and only if  ${\bf P}_mT{\bf P}_m: W_m\cap H_1\longmapsto W_m\cap H_2, 1\le m\le M$,
are isometric.
\end{theorem}

From the proof of  Theorem \ref{rkhsisometric.thm}, we see that  a
shift-invariant linear operator $T$ from a SIGRKHS $H_1$ to another
SIGRKHS $H_2$ is bounded, and its operator norm
$\|T\|$ is the same as the  maximal bound of the operator norm $\|{\bf P}_mT{\bf P}_m\|$
of ${\bf P}_m T{\bf P}_m$ from $H_1\cap W_m$ to $H_2\cap W_m, 1\le m\le M$, i.e., 
$\|T\|= \max_{1\le m\le M} \|{\bf P}_mT{\bf P}_m\|$.

\section{Proofs}

In this section, we collect the proofs of  Theorems \ref{bandapproximation.thm}, \ref{projection.polynomialfilter.thm},  \ref{polynomialfilter.thm},
\ref{polynomialinvariance.thm},  
  \ref{shiftinvariantfilter.thm}, \ref{center.thm},
 \ref{rangecharacterization.thm}, \ref{filterrange.thm}, \ref{bandlimited.thm}, \ref{sisequality.thm}, \ref{fiber.thm},  \ref{FGSIS.thm},  \ref{frame.thm}, \ref{frameoperator.thm}, \ref{dualframe.thm}, \ref{dualframeexistence.thm}, \ref{bessel.thm}, \ref{riesz.thm},
 \ref{rkhs.thm1}, \ref{rkhs.bandlimited.thm},   \ref{decomposition.thm}, and \ref{rkhsisometric.thm}.

\subsection{Proof of Theorem \ref{bandapproximation.thm}}
\label{bandapproximation.thm.pfsection}
By \eqref{projection.eq1}, \eqref{projection.eq2} and \eqref{projection.eq3}, we have
\begin{eqnarray*}
    \sum_{l=1}^d \|{\bf S}_l {\bf x}\|_2^2
   & = & \sum_{l=1}^d \sum_{m=1}^M |\gamma_l(m)|^2 \|{\bf P}_m {\bf x}\|_2^2
     =  \sum_{m=1}^M \|{\pmb \gamma}(m)\|^2 \| {\bf P}_m {\bf x}\|^2\\
    & \ge &  \|{\pmb \gamma}(K+1)\|^2 \sum_{m=K+1}^M \| {\bf P}_m {\bf x}\|^2
    = \|{\pmb \gamma}(K+1)\|^2  \|{\bf x}-{\bf x}_K\|^2.
\end{eqnarray*}
Dividing both sides of the inequality by
$\|{\pmb \gamma}(K+1)\|^2$
 and subsequently taking square roots yield
 the desired approximation error estimate in 
\eqref{bandapproximation.thm.eq2}.

 \subsection{Proof of Theorem \ref{projection.polynomialfilter.thm}}
 \label{projection.polynomialfilter.thm.pfsection}

Clearly it suffices to prove \eqref{projection.polynomialfilter.thm.eq1}.
As ${\pmb \gamma}(m'), 1\le m'\le M$, are distinct, we can find a multivariate polynomial $h_{m}$ of degree at most $M-1$ for every $1\le m\le M$ such that
 \begin{equation} \label{shiftinvariantfilter.thm.pfeq1}
 h_m({\pmb \gamma}(m'))=\delta_{mm'}, \ 1\le m'\le M,
 \end{equation}
\cite[Theorem 1 on p. 58]{cheney2000}.
 Then for  arbitrary graph signal $\x$ on $\G$, we have
 \begin{eqnarray*}
 \big({\bf P}_m-h_m({\bf S}_1, \ldots, {\bf S}_d)\big)\x \hskip-0.1in&=\hskip-0.1in& {\bf P}_m{\bf x}-
\sum_{m'=1}^M  h_m({\bf S}_1, \ldots, {\bf S}_d) {\bf P}_{m'}\x
\\
 \hskip-0.1in&=\hskip-0.1in& {\bf P}_m{\bf x}- \sum_{m'=1}^{M}  h_m({\pmb \gamma}(m')) {\bf P}_{m'}\x=\bf 0,
 \end{eqnarray*}
 where  the first equality holds by  \eqref{projection.eq2}, the second equality follows from  \eqref{projection.eq3} and \eqref{polynomialGFT}, and the last equality is true by  \eqref{shiftinvariantfilter.thm.pfeq1}.  Therefore the orthogonal projections ${\bf P}_m, 1\le m\le M$, are polynomial filters.

\subsection{Proof of Theorem \ref{polynomialfilter.thm}}
\label{polynomialfilter.thm.pfsection}

  
 The necessity of the  characterization \eqref{shiftinvariantfilter.thm.eq1} follows easily from \eqref{polynomialGFT}.
In particular, for the polynomial filter ${\bf H}$ in \eqref{filter.def}, we have
\begin{equation} \label{shiftinvariantfilter.thm.pfeq2}
\mu_m= h({\pmb \gamma}(m)),\ \ 1\le m\le M.\end{equation}

Now the sufficiency of the  requirement \eqref{shiftinvariantfilter.thm.eq1}.  As ${\pmb \gamma}(m), 1\le m\le M$, are distinct, we can find a multivariate polynomial $h$ of degree at most $M-1$ such that \eqref{shiftinvariantfilter.thm.pfeq2} holds \cite[Theorem 1 on p. 58]{cheney2000}. 
Then for  arbitrary graph signal $\x$ on $\G$, we obtain 
 \begin{eqnarray*}
 \big({\bf H}-h({\bf S}_1, \ldots, {\bf S}_d)\big)\x \hskip-0.1in&=\hskip-0.1in& 
\sum_{m=1}^M \big ({\bf H}-h({\bf S}_1, \ldots, {\bf S}_d)\big) {\bf P}_m\x
\\
 \hskip-0.1in&=\hskip-0.1in& \sum_{m=1}^{M} \mu_m {\bf P}_m{\bf x}- h({\pmb \gamma}(m)) {\bf P}_m\x={\bf 0},
 \end{eqnarray*}
 where  the first equality holds by  \eqref{projection.eq2}, the second equality follows from  \eqref{polynomialGFT} and \eqref{shiftinvariantfilter.thm.eq1}, 
 and the last equality is true by  \eqref{shiftinvariantfilter.thm.pfeq2}. This proves the sufficiency and hence completes the proof. 

\subsection{Proof of Theorem \ref{polynomialinvariance.thm}}
\label{polynomialinvariance.thm.pfsection}

Take a polynomial filter ${\bf H}$ in \eqref{filter.def}. For any orthogonal matrix ${\bf U}$
in the eigendecomposition \eqref{eigen.dec} of the graph shifts
${\bf S}_1, \ldots, {\bf S}_d$, applying \eqref{eigen.dec} repeatedly yields 
\begin{equation} \label{polynomialinvariance.thm.pfeq1}
    {\bf U}^T {\bf H} {\bf U}= {\bf U}^T h({\bf S}_1, \ldots, {\bf S}_d) {\bf U}=
    h({\pmb  \Lambda}_1, \ldots, {\pmb \Lambda}_d). 
\end{equation}
This proves the necessity.

Next  the sufficiency. Take an orthogonal matrix  ${\bf U}=[{\bf u}_1, \ldots, {\bf u}_N]$ in the eigendecomposition \eqref{eigen.dec} of the graph shifts
${\bf S}_1, \ldots, {\bf S}_d$, and define
${\bf W}= {\bf U}^T{\bf H} {\bf U}=[w(n, n')]_{1\le n, n'\le N}$. 
First we  prove the following: 

 {\em Claim 1:  ${\bf W}$ is a  diagonal matrix.}

Let  ${\bf V}={\bf US}$, where ${\bf S}={\rm diag}[\epsilon_1, \ldots, \epsilon_N]$ is a diagonal matrix with
$\epsilon_n\in \{-1, 1\}, 1\le n \le N$, arbitrary chosen. One  may verify that ${\bf V}$  is an  orthogonal matrix that can be  used in the eigendecomposition \eqref{eigen.dec}. Then by the invariance assumption, we have
${\bf S}^T {\bf W} {\bf S}={\bf W}$, or equivalently
$$ w(n, n')\epsilon(n)\epsilon(n')=w(n, n'), \ \ 1\le n, n'\le N. $$
This together with   arbitrary selections of $\epsilon_n\in \{-1, 1\}, 1\le n\le N$, proves  Claim 1.

\smallskip

By Claim 1, we conclude that
\begin{equation} \label{polynomialinvariance.thm.pfeq3}
{\bf H}=\sum_{n=1}^N w(n, n) {\bf u}_n {\bf u}_n^T.
\end{equation}
Next we show that some of diagonal entries of the matrix ${\bf W}$ must be the same. 

\noindent {\em Claim 2:  $w(n, n)=w(n', n')$ if ${\pmb \lambda}(n)={\pmb \lambda}(n')$.
}

Let ${\bf V}$ be the orthogonal matrix by swapping the $n$-th and $n'$-th columns of the matrix ${\bf U}$. Then by the assumption that ${\pmb \lambda}(n)={\pmb \lambda}(n')$, 
${\bf V}$ is an orthogonal matrix that can be  used in the eigendecomposition \eqref{eigen.dec}. Therefore the desired conclusion  in Claim 2 follows directly from ${\bf U}^T {\bf H} {\bf U}={\bf V}^T {\bf H} {\bf V}$.

 Set $\mu(m)=w(n, n), 1\le m\le M$, where 
${\pmb \lambda}(n)={\pmb \gamma}(m)$. The above definition of $\mu(m), 1\le m\le M$, is well-defined by Claim 2. 
Therefore
\begin{equation} \label{polynomialinvariance.thm.pfeq4}
{\bf H}=\sum_{n=1}^N w(n, n) {\bf u}_n {\bf u}_n^T=
\sum_{m=1}^M \mu(m) \sum_{{\pmb \lambda}(n)={\pmb \gamma}(m)} {\bf u}_n {\bf u}_n^T=
\sum_{m=1}^M \mu(m){\bf P}_m,
\end{equation}
where the first equality holds by
 \eqref{polynomialinvariance.thm.pfeq3}, the second one follows from 
 Claim 2 and definition of $\mu(m), 1\le m\le M$, and the last equality
 is true by  the definition of the orthogonal projections ${\bf P}_m, 1\le m\le M$
 in 
\eqref{projection.eq0}.
Therefore the desired  conclusion for the filter ${\bf H}$
follows from  \eqref{shiftinvariantfilter.thm.eq3}, \eqref{polynomialinvariance.thm.pfeq4} and Theorem \ref{polynomialfilter.thm}.

 \subsection{Proof of Theorem \ref{shiftinvariantfilter.thm}}
\label{shiftinvariantfilter.thm.pfsection}

    Take an arbitrary shift invariant filter ${\bf H}\in \mathcal T$. Then 
    \begin{equation*}        {\bf H} =\sum_{m=1}^{M}{\bf P}_m {\bf H}=\sum_{m=1}^{M}{\bf P}_m^2 {\bf H}=
    \sum_{m=1}^{M}{\bf P}_m {\bf H} {\bf P}_m.
    \end{equation*}
    where the first two equalities follows from \eqref{projection.eq1} and \eqref{projection.eq2}, and the third  equation holds by \eqref{shiftinvariant.def} and  Theorem \ref{projection.polynomialfilter.thm}. This proves the necessity. 
    
    Take a filter ${\bf H}=\sum_{m=1}^{M}{\bf P}_m{\bf H}_m{\bf P}_m$ for some filters ${\bf H}_m, 1\le m\le M$.
 By   \eqref{projection.eq3},
   we have  
     \begin{equation*}
 {\bf H} {\bf S}_l=     \sum_{m=1}^{M}{\bf P}_m{\bf H}_m{\bf P}_m {\bf S}_l=\sum_{m=1}^{M}\gamma_l(m)  {\bf P}_m{\bf H}_m{\bf P}_m = \sum_{m=1}^{M} {\bf S}_l {\bf P}_m{\bf H}_m{\bf P}_m={\bf S}_l {\bf H}, \ \ 1\le l\le d.
    \end{equation*}
 This proves ${\bf H}$ is shift-invariant and hence the sufficiency.

\subsection{Proof of Theorem \ref{center.thm}}
\label{center.thm.pfsection}

Applying \eqref{shiftinvariant.def} repeatedly
yields $\mathcal P\subset C(\mathcal T)$. Therefore it suffices to prove
\begin{equation} \label{center.thm.pfeq1}C(\mathcal T)\subset \mathcal P.\end{equation}
Take an arbitrary  filter ${\bf G}\in C({\mathcal T})$. Then there exist filters ${\bf G}_m, 1\le m\le M$, by Theorem \ref{shiftinvariantfilter.thm} such that ${\bf G}=\sum_{m=1}^M{\bf P}_m{\bf G}_m{\bf P}_m$. By \eqref{projection.eq1},  \eqref{shiftinvariantcenter.def} and Theorem \ref{shiftinvariantfilter.thm},
we have 
   \begin{equation*}
      \sum_{m=1}^M {\bf P}_m{\bf B}_m{\bf P}_m{\bf G}_m{\bf P}_m= {\bf HG}={\bf GH}=\sum_{m=1}^M {\bf P}_m{\bf G}_m{\bf P}_m{\bf B}_m{\bf P}_m
   \end{equation*}
   for arbitrary matrices ${\bf B}_m, 1\le m\le M$.
   Therefore
   \begin{equation} \label{center.thm.pfeq2}
   {\bf P}_m{\bf B}_m{\bf P}_m{\bf G}_m{\bf P}_m= {\bf P}_m{\bf G}_m{\bf P}_m{\bf B}_m{\bf P}_m, \quad 1\le m\le M,
   \end{equation}
   holds for all matrices ${\bf B}_m, 1\le m\le M$.

   Let ${\bf u}_n, 1\le n\le N$, be an orthonormal basis in the    eigendecomposition
 \eqref{eigen.dec} of graph shifts ${\bf S}_1, \ldots, {\bf S}_d$,
 $I_m=\{n\ | \ {\pmb \lambda}(n)={\pmb \gamma}(m)\}$,
 and write
 $\tilde {\bf B}_m=[{\bf u}_n^T {\bf B}_m {\bf u}_{n'}]_{n, n'\in I_m}$ and
 $ \tilde {\bf G}_m=[{\bf u}_n^T {\bf G}_m {\bf u}_{n'}]_{n, n'\in I_m}, 1\le m\le M$.
 Then
 by \eqref{projection.eq0}, we can write 
 \eqref{center.thm.pfeq2}
 as
 \begin{equation}
 \tilde {\bf  B}_m \tilde {\bf G}_m=\tilde {\bf G}_m \tilde {\bf B}_m
 \end{equation}
 hold for all matrices $\tilde {\bf B}_m$ of size $|I_m|, 1\le m\le M$. This implies that $\tilde {\bf G}_m$ is a multiple of the identity matrix, which in turn proves that 
 $${\bf P}_m {\bf G}_m{\bf P}_m=\mu_m {\bf P}_m$$
 for some $\mu_m, 1\le m\le M$. Therefore ${\bf G}$ is a polynomial filter by Theorem \ref{polynomialfilter.thm}.

\subsection{Proof of Theorem \ref{rangecharacterization.thm}}
\label {rangecharacterization.thm.pfsection}


First the sufficiency. Let $U$ be the linear space of graph signals in \eqref{rangecharacterization.thm.eq0} and take ${\bf x}\in U$. For $1\le l\le d$ and $1\le m\le M$, we have
\begin{equation*}
 \widehat {{\mathbf S}_l {\bf x}}(m)=\gamma_l(m) 
\widehat {{\bf x}}(m)\in J({\pmb \gamma}(m)),
\end{equation*}
where the first equality hold by \eqref{multiplier.prop}
and the inclusion follows from 
\eqref{rangecharacterization.thm.eq0} and the linear space property for the range spaces of the range function  $J$.  Hence $U$ is shift-invariant and the sufficiency follows.

Next the necessity.  Define the range function $J$ by
\begin{equation} \label{rangecharacterization.thm.pfeq1}
J({\pmb\gamma}(m))=\{\widehat{{\bf x}}(m)\mid {\bf x}\in U\}=\{ {\bf P}_m {\bf x}\ \mid\ {\bf x}\in U\}, \  1\le m\le M,\end{equation}
where ${\bf P}_m, 1\le m\le M$, are the orthogonal projections in 
\eqref{projection.eq0}.  Let $W_m$
be as in \eqref {projection.eq4}. By  Theorem \ref{projection.polynomialfilter.thm} and the shift-invariance of the linear space $U$, we conclude that
$J({\pmb \gamma}(m))$ is a linear subspace of $W_m\cap U$
for every $1\le m\le M$, and hence $J$ is a graph range function.

For the  graph range function $J$ in \eqref{rangecharacterization.thm.pfeq1}, define
    \begin{equation*}\label{rangecharacterization.thm.pfeq2}
      \widetilde   U=\{\x\in {\mathbf R}^V\ \mid \ \widehat{{\bf x}}(m)\in J(\pmb\gamma(m)),\ 1\le m\le M\}.
    \end{equation*}
    Clearly  $U\subset \tilde U$. On the other hand, 
    for any ${\bf x}\in \tilde U$, we have 
    $${\bf x}=\sum_{m=1}^M {\bf P}_m {\bf x}\in 
    J({\pmb\gamma}(1))+ \cdots+ J({\pmb\gamma}(M))\subset
    U.$$
    This proves that $\tilde U=U$ and hence  completes the proof of the necessity. 

\subsection{Proof of Theorem \ref{filterrange.thm}}
\label{filterrange.thm.pfsection}
  Clearly it suffices to show that
 a  GSIS $U$ is the range space of some shift-invariant filter  ${\bf H}$ and also the kernel space of some shift-invariant filter ${\bf G}$.

Let  ${\bf Q}_m, 1\le m\le M$, be the orthogonal projections onto $W_m\cap U$.
Then for every $1\le m\le M$, we have
\begin{equation}\label{filterrange.thm.pfeq1}
{\bf Q}_m^2={\bf Q}_m, \ {\bf Q}_m^T={\bf Q}_m\ {\rm and}\  {\bf Q}_m{\mathbb R}^V=W_m\cap U.
\end{equation}
Define
\begin{equation}\label{filterrange.thm.pfeq2} {\bf H}=\sum_{m=1}^M {\bf Q}_m.\end{equation}
Then for every $1\le l\le d$,
\begin{equation}
{\bf S}_l {\bf H}=\sum_{m=1}^M \gamma_l(m) {\bf P}_m {\bf Q}_m=
\sum_{m=1}^M \gamma_l(m)  {\bf Q}_m
=\sum_{m=1}^M \gamma_l(m)  {\bf Q}_m{\bf P}_m= {\bf H} {\bf S}_l,
\end{equation}
where the first and last equalities follow from 
\eqref{projection.eq1} and \eqref{projection.eq3}, and the second and third equalities hold by \eqref{filterrange.thm.pfeq1}. This proves that the filter ${\bf H}$ in \eqref{filterrange.thm.pfeq2} is shift-invariant.

By \eqref{filterrange.thm.pfeq2}, we see that
\begin{eqnarray} R({\bf H}) \hskip-0.1in  & = \hskip-0.05in&\hskip-0.05in \Big\{\sum_{m=1}^M {\bf Q}_m {\bf x}\ |\ {\bf x}\in {\mathbb R}^V\Big\}\nonumber\\
\hskip-0.1in  & = \hskip-0.05in&\hskip-0.05in \Big\{\sum_{m=1}^M {\bf Q}_m {\bf x}_m\ |\ {\bf x}_m\in W_m, 1\le m\le M\Big\}
=\sum_{m=1}^M U\cap W_m=U,
\end{eqnarray}
where the second inequality hold
by \eqref{filterrange.thm.pfeq1}
and \eqref{projection.eq2}. Therefore ${\bf H}$ is a shift-invairant filter with $U$ as its range space.

Define ${\bf G}={\bf I}-{\bf H}$. Then since ${\bf H}$ is orthogonal projection onto $U$, one may verify that ${\bf G}$
is a shift-invariant filter  with $U$ as its kernel space.

\subsection{Proof of Theorem \ref{bandlimited.thm}}
\label{bandlimited.thm.pfsection}


First the necessity.  By
\eqref{rangefunction.def0} and Theorem
\ref{rangecharacterization.thm}, it suffices to prove
\begin{equation}\label{bandlimited.thm.pfeq1}
U\cap W_m=W_m \ \ {\rm if }\ \ U\cap W_m\ne\{0\}. 
\end{equation}
By \eqref{superSIS.def} and Theorem
\ref{projection.polynomialfilter.thm},
\begin{equation}\label{bandlimited.thm.pf.eq2}
{\bf H} (U\cap W_m)\subset U\cap W_m, \ 1\le m\le M, 
\end{equation}
hold for all shift-invariant filters ${\bf H}\in {\mathcal T}$.
From   \eqref{projection.eq1}, \eqref{bandlimited.thm.pf.eq2}
and Theorem \ref{shiftinvariantfilter.thm}, it follows that
\begin{equation}\label{bandlimited.thm.pf.eq3}
{\bf P}_m {\bf H}_m {\bf P}_m (U\cap W_m)\subset U\cap W_m
\end{equation}
hold for all filters ${\bf H}_m$. As  the only nontrivial linear space invariant under
all transformation on $W_m$ is the whole space $W_m$ itself. This proves 
\eqref{bandlimited.thm.pfeq1}.

Now the sufficiency. Take a shift-invariant filter ${\bf H}\in {\mathcal T}$ and $\x\in U$.  By Theorem \ref{projection.polynomialfilter.thm} and the shift-invariance of $U$, we have 
\begin{equation}\label{bandlimited.thm.pf.eq5}
{\bf P}_m\x\in W_m\cap U,\ 1\le m\le M.
\end{equation}
Then 
\begin{eqnarray}
{\bf H}\x= \sum_{m=1}^M {\bf P}_m {\bf H}{\bf P}_m\x= \sum_{{\bf P}_m\x\ne {\bf 0}} 
{\bf P}_m {\bf H} {\bf P}_m {\bf x}
\in \bigoplus_{{\bf P}_m\x\ne {\bf 0}}  W_m\subset U,
\end{eqnarray}
where the first equality holds by \eqref{shiftinvariantfilter.commutative.eq2},  and the last inclusion is true by \eqref{bandlimited.thm.pfeq1}.  This proves the super-shift-invariance for the linear space $U$.

\subsection{Proof of Theorem \ref{sisequality.thm}}
\label{sisequality.thm.pfsection}

%
%
%

(i)$\Longrightarrow$(ii): \quad The conclusion follows from  applying
\eqref{maximalsis.eq0} repeatedly. 

(ii)$\Longrightarrow$(iii): \quad Take
${\bf x}\in U$ and $1\le m\le M$ with $\gamma_l(m)=0$
for some $1\le l\le d$. Then by the assumption on the GSIS $U$,  there exists
 ${\bf y}\in U$ such that
${\bf x}={\bf S}_1\ldots {\bf S}_d {\bf y}$. Therefore
\begin{equation*} 
    {\bf P}_m{\bf x}= \gamma_1(m) \ldots \gamma_d(m) {\bf P}_m {\bf y}=0,\end{equation*}
    where the first equality follows from \eqref{projection.eq3} and the second equality holds as 
    $\gamma_1(m) \ldots \gamma_d(m)=0$ by the selection of $1\le m\le M$.
This proves the conclusion in  (iii).

(iii)$\Longrightarrow$(i):\quad Take $1\le l\le d$. By the shift-invariance assumption on the linear space $U$, we have ${\bf S}_l U\subset U$. Then it suffices to prove
\begin{equation} \label{sisequality.thm.pfeq1} U\subset {\bf S}_l U.
\end{equation}
Take ${\bf x}\in U$.
By the assumption  (iii), 
\begin{equation} \label{sisequality.thm.pfeq2} {\bf P}_m{\bf x}=0\end{equation}
for all $1\le m\le M$ such that $\gamma_l(m)=0$.
Set ${\bf y}_l=\sum_{\gamma_l(m)\ne 0} (\gamma_l(m))^{-1} {\bf P}_m {\bf x}$.
Then ${\bf y}_l\in U$ by Theorem \ref{projection.polynomialfilter.thm} and the shift-invariance  assumption on  the linear space $U$.
Also we have
\begin{equation}
{\bf S}_l {\bf y}_l= \sum_{\gamma_l(m)\ne 0} (\gamma_l(m))^{-1} {\bf S}_l {\bf P}_m {\bf x}=
\sum_{\gamma_l(m)\ne 0}  {\bf P}_m {\bf x}=
\sum_{m=1}^M  {\bf P}_m {\bf x}={\bf x},
\end{equation}
where the second equality follows from \eqref{projection.eq3}, the third one holds by  \eqref{sisequality.thm.pfeq2} and the last one is true by \eqref{projection.eq2}. This proves 
\eqref{sisequality.thm.pfeq1}  and completes the proof. 

\subsection{Proof of Theorem \ref{fiber.thm}}     \label{fiber.thm.pfsection}

 Take $1\le m\le M$.
By Theorem \ref{projection.polynomialfilter.thm} and
the shift-invariance of the linear space ${\mathcal S}(\Phi)$, we clearly have
$${\rm span} \big\{  \widehat {\phi}(m)\ | \ \phi\in\Phi\big\} \subset {\mathcal S}(\Phi)\cap W_m=J_{{\mathcal S}(\Phi)} ({\pmb \gamma}(m)).$$
Then it suffices to prove that
\begin{equation} \label{fiber.thm.pfeq1}
J_{{\mathcal S}(\Phi)} ({\pmb \gamma}(m)) \subset {\rm span} 
\big\{\widehat{\phi}(m)\ | \ \phi\in\Phi\big\}.
\end{equation}

 Take $\x\in {\mathcal S}(\Phi)$. Then 
$${\bf x}=\sum_{{\pmb \alpha}\in {\mathbb Z}_+^d} \sum_{\phi\in \Phi}c_{\phi,\pmb\alpha}{\bf S}^{\pmb\alpha}\phi$$
for some finite collection of coefficients
 $c_{\phi, {\pmb \alpha}}$. This together with  \eqref{projection.eq3} implies that 
$$\widehat{\bf x}(m)=\sum_{\phi\in\Phi}h_\phi(\pmb\gamma(m)) \widehat \phi(m),$$ 
where
$h_{\phi}(\pmb\gamma(m))=\sum_{\pmb\alpha \in {\mathbb Z}_+^d}c_{\phi,\pmb\alpha}\pmb\gamma(m)^{\pmb\alpha}$.
This proves \eqref{fiber.thm.pfeq1}
and completes the proof.

\subsection{Proof of Theorem \ref{FGSIS.thm}} 
\label{FGSIS.thm.pfsection}

   
(i)$\Longrightarrow$(ii): Let  $r 
 = \max_{1\le m\le M} \dim U\cap W_m$, and let
\begin{equation} \label{FGSIS.thm.pfeq1} \phi_j=\sum_{m=1}^M {\bf e}_{m, j}, \ 1\le j\le r
    \end{equation}
    be chosen so that for every  $1\le m\le M$, 
    ${\bf e}_{m, j}, 1\le j\le r$, forms a spanning set
    for the linear space $U\cap W_m$. The existence 
    of such spanning sets follows 
    as $\dim U\cap W_m\le r$ for all $1\le m\le M$.
Set $\Phi=\{\phi_j\}_{j=1}^r$. Then it reduces to establishing
\begin{equation} \label{FGSIS.thm.pfeq2} 
U={\mathcal S}(\Phi).
\end{equation}
By  \eqref{FGSIS.thm.pfeq1} and the shift-invariance of the linear space $U$, we have 
${\mathcal S}(\Phi)\subset U$.
Then it suffices to prove
\begin{equation}  \label{FGSIS.thm.pfeq4} U\subset
{\mathcal S}(\Phi).
\end{equation}
Take ${\bf x}\in U$. By \eqref{FGSIS.thm.pfeq1}, there exists coefficients $c_{m,j}, 1\le j\le r$, for every $1\le m\le M$, such that
\begin{equation} \label{FGSIS.thm.pfeq5}
{\bf P}_m {\bf x}=\sum_{j=1}^r c_{m, j} {\bf e}_{m,j}=\sum_{j=1}^r c_{m, j} {\bf P}_m \phi_j,
\end{equation}
where the last equality holds by \eqref{projection.eq1}.
Let  $h_{m}, 1\le m\le M$,  be  multivariate polynomials in \eqref{shiftinvariantfilter.thm.pfeq1}
and define
$Q_j=\sum_{m=1}^M c_{m,j} h_m$.
Then  
\begin{equation}  \label{FGSIS.thm.pfeq6}
Q_j({\pmb \gamma}(m))= c_{m,j},\quad  1\le m\le M.
\end{equation}
By  \eqref{projection.eq3}, 
 \eqref{FGSIS.thm.pfeq5} and
\eqref{FGSIS.thm.pfeq6}, we have
$$ {\bf P}_m {\bf x}=
\sum_{j=1}^r Q_j({\pmb \gamma}(m)) {\bf P}_m \phi_j=\sum_{j=1}^r Q_j({\bf S}_1, \ldots, {\bf S}_d) {\bf P}_m \phi_j.
$$
Hence
$$
{\bf x}=\sum_{m=1}^M {\bf P}_m {\bf x}= \sum_{j=1}^r Q_j({\bf S}_1, \ldots, {\bf S}_d) \phi_j\in {\mathcal S}(\Phi).$$
This proves \eqref{FGSIS.thm.pfeq4} and hence the conclusion in (ii) holds.  

 (ii)$\Longrightarrow$(iii):  For any $k\ge 0$, define 
 \begin{equation}\label{FGSIS.thm.pfeq9}
 {\bf T}^k=\Big(\sum_{l=1}^d{a_l}{\bf S}_l\Big)^k=\sum_{|\pmb\alpha|=k}c_{\pmb\alpha}{\bf S}^{\pmb\alpha}=:P_k({\bf S}_1, \ldots, {\bf S}_d),
 \end{equation} where $|\pmb\alpha|=\alpha_1+\ldots+\alpha_d$ and 
 $c_{\pmb\alpha}=\frac{k!}{\alpha_1!\cdots \alpha_d!}a_1^{\alpha_1}\cdots a_d^{\alpha_d}$. Then for any $\phi\in\Phi$, 
 \begin{equation*}
     {\bf T}^k\phi=P_k({\bf S}_1, \ldots, {\bf S}_d)\phi\in U. 
 \end{equation*}
  It implies that \begin{equation}\label{FGSIS.thm.pfeq7}
      \widetilde U:=\span\{ {\bf T}^k\phi\mid k\ge 0, \phi\in\Phi\}\subset U.
  \end{equation}
  On the other hand, by \eqref{FGSIS.thm.eq1},  for any $\pmb\alpha\in\Z_+^d$,  there exist  real numbers $p_{\pmb\alpha,1},\ldots, p_{\pmb\alpha,M-1}$ such that 
  \begin{equation*} 
\pmb\gamma(m)^{\pmb\alpha}=\sum_{k=0}^{M-1}p_{\pmb\alpha,k}\mu(m)^{k}, 1\le m\le M,
  \end{equation*}
  where $\mu(m)$, $1\le m\le M$, are defined in \eqref{FGSIS.thm.eq1-}.
  Then for any graph signal $\x=\sum_{\pmb\alpha\in\Z_+^d}\sum_{\phi\in\Phi}\tilde{c}_{\pmb\alpha, \phi}\mathcal S^{\pmb\alpha}\phi\in U$, we have
  \begin{eqnarray*} 
  \x&=&\sum_{m=1}^{M}{\bf P}_m\x=\sum_{m=1}^M\sum_{\pmb\alpha\in\Z_+^d}\sum_{\phi\in\Phi}\tilde{c}_{\pmb\alpha,\phi}\pmb\gamma(m)^{\pmb\alpha}\widehat\phi(m)\nonumber\\
&=&\sum_{m=1}^M\sum_{\pmb\alpha\in\Z_+^d}\sum_{\phi\in\Phi}\sum_{k=0}^{M-1}\tilde{c}_{\pmb\alpha,\phi}p_{\pmb\alpha,k}\mu(m)^{k}\widehat\phi(m)\in \bigoplus _{1\le m\le M}{\bf P}_m{\widetilde U}=\widetilde U.
  \end{eqnarray*}
  where the inclusion follows from $P_m({\bf T}^k\phi)=\mu(m)^k\widehat\phi(m)$, and the last equality holds since the eigenvalues $\{\mu(m)\}_{m=1}^M$ of ${\bf T}$ are pairwise distinct.
  Therefore $U\subset \widetilde U$. This together with \eqref{FGSIS.thm.pfeq7} proves the conclusion in (iii).

   (iii)$\Longrightarrow$(i): For any $1\le l\le d$ and $\phi\in \Phi$, by \eqref{FGSIS.thm.pfeq9},
   \begin{equation*}
       {\bf S}_l{\bf T}^k\phi={\bf S}_l\sum_{|\pmb\alpha|=k}c_{\pmb\alpha}{\bf S}^{\pmb\alpha}\phi=\sum_{|\pmb\alpha|=k}c_{\pmb\alpha}{\bf S}_1^{\alpha_1}\ldots {\bf S}_{l-1}^{\alpha_{l-1}}{\bf S}_{l}^{\alpha_{l}+1}{\bf S}_{l+1}^{\alpha_{l+1}}\ldots {\bf S}_{d}^{\alpha_d}\phi\in U.
   \end{equation*}
   This implies that $U$ is shift-invariant and completes the proof.

\subsection{Proof of Theorem \ref{frame.thm}}
\label{frame.thm.pfsection}

To prove Theorem \ref{frame.thm}, we need a technical lemma.

\begin{lemma}\label{frame.lemma}
Let  ${\bf A}$ be as in \eqref{frame.thm.eq1}. 
Then ${\bf A}$ is nonsingular.
\end{lemma}

\begin{proof} Suppose, on the contrary, that
the matrix ${\bf A}$ in \eqref{frame.thm.eq1} is singular. Then there exists ${\bf 0}\ne {\bf u}=[u_1, \ldots, u_M]\in {\mathbb R}^M$ such that
${\bf A} {\bf u}={\bf 0}$.
Observe that
$${\bf u}^T {\bf A} {\bf u}
=\sum_{{\pmb \alpha}\in {\mathbb Z}_+^d}
\Big|\sum_{m=1}^M  u_m {\pmb \gamma}(m)^{\pmb \alpha}\Big|^2.$$
Therefore 
\begin{equation}
\label{frame.lemma.pf.eq1}
\sum_{m=1}^M  u_m  {\pmb \gamma}(m)^{\pmb \alpha}=0 \ \ {\rm for \ all}\ {\pmb \alpha}\in {\mathbb Z}_+^d.
\end{equation}
Let $h_m, 1\le m\le M$, be multivariate polynomials in  \eqref{shiftinvariantfilter.thm.pfeq1}.
By \eqref{frame.lemma.pf.eq1}, we obtain
$$u_m =\sum_{m'=1}^M u_{m'} h_m({\pmb \gamma}(m'))=0,\quad  1\le m\le M.$$
This is a contradiction.
\end{proof}

Now we prove Theorem \ref{frame.thm}.

\begin{proof} [Proof of Theorem \ref{frame.thm}] Take $\x\in {\mathcal S}(\Phi)$.
Then we have
\begin{eqnarray} \label{frame.thm.eq2}  \sum_{\phi\in \Phi} \sum_{{\pmb \alpha}\in {\mathbb Z}_+} |\langle \x, {\bf S}^{\pmb \alpha}\phi\rangle|^2 
& = & 
\sum_{\phi\in \Phi} \sum_{{\pmb \alpha}\in {\mathbb Z}_+} \Big|\sum_{m=1}^M \langle \widehat {\bf x}(m), \widehat \phi(m)\rangle {\pmb \gamma}(m)^{\pmb \alpha}\Big|^2\nonumber\\
& = & 
\sum_{\phi\in \Phi}
\sum_{m, m'=1}^M \langle \widehat {\bf x}(m), \widehat \phi(m)\rangle 
\langle \widehat {\bf x}(m'), \widehat \phi(m')\rangle \prod_{l=1}^d [1-\gamma_l(m) \gamma(m')]^{-1}\nonumber\\
&\le & 
\sigma_{\max}({\bf A}) \sum_{\phi\in \Phi} \sum_{m=1}^M |\langle \widehat \x(m), \widehat \phi(m)\rangle|^2
\nonumber\\
& \le & 
\sigma_{\max}({\bf A})\sum_{m=1}^M \sigma_{\max}(F({\pmb\gamma}(m)) 
\|\widehat {\bf x}(m)\|^2\nonumber\\
& \le & 
\sigma_{\max}({\bf A})\Big(\sup_{1\le m\le M} \sigma_{\max}(F({\pmb\gamma}(m))  \Big)
\|{\bf x}\|^2,
\end{eqnarray}
where the first equation  \eqref{GFT.def} and \eqref{projection.eq3},
and the last inequality holds by 
the Paserval’s identity \eqref{parsevelidentity}.
By similar argument, we have
\begin{equation} \label{frame.thm.eq3}\sum_{\phi\in \Phi} \sum_{{\pmb \alpha}\in {\mathbb Z}_+} |\langle \x, {\bf S}^{\pmb \alpha}\phi\rangle|^2\ge 
\sigma_{\min}^+({\bf A})
\Big(\min_{1\le m\le M} \sigma_{\min}^+(F(\pmb\gamma(m))) \Big)
\|{\bf x}\|^2.
\end{equation}
Combining \eqref{frame.thm.eq2} and \eqref{frame.thm.eq3} completes the proof. 
\end{proof}

\subsection{Proof of Theorem \ref{frameoperator.thm}}
\label{frameoperator.thm.pfsection}
By \eqref{shiftinvariantoperator.eq0}, one may verify that  the frame operator $S$ is shift-invariant if and only if 
 $S=\sum_{m=1}^M {\bf P}_m S {\bf P}_m$ if and only if
 ${\bf P}_M  S {\bf P}_{m'}=0$ for all $1\le m\ne m'\le M$. On the other hand, for $1\le m\ne m'\le M$, we obtain from \eqref{projection.eq3} 
  that
 $${\bf P}_M  S {\bf P}_{m'} {\bf x}=
 \sum_{\phi\in\Phi} \sum_{{\pmb \alpha}\in {\mathbb Z}_+} \langle {\bf P}_{m'}{\bf x}, {\bf S}^{\pmb \alpha} \phi\rangle {\bf P}_m {\bf S}^{\pmb \alpha}\phi=
 \prod_{l=1}^d 
\big(1-\gamma_l(m)\gamma_l(m')\big)^{-1}
\sum_{\phi\in \Phi}  \langle  {\bf x}, \widehat \phi(m')\rangle  \widehat \phi(m).
 $$ 
This together with Theorem \ref{fiber.thm} implies that
 the frame operator $S$ 
 in \eqref{frameoperator.def} is shift-invariant if and only if \begin{equation*}\sum_{\phi\in \Phi}  \langle  {\bf x},
 \widehat \phi(m')\rangle 
 \langle \widehat \phi(m), {\bf y}\rangle =0, \ \ {\rm for \ all}\ \   {\bf x}, {\bf y} \in  {\mathcal S}(\Phi)
 \end{equation*}
 hold for all $1\le m\ne m'\le M$, or equivalently
 \begin{equation*}\sum_{\phi\in \Phi}  \langle  \widehat {\bf x}(m'),
 \widehat \phi(m')\rangle 
 \langle \widehat \phi(m), \widehat {\bf y}(m)\rangle =0
 \end{equation*}
 hold for all $\widehat {\bf x}(m')\in W_{m'}\cap {\mathcal S}(\Phi)$, $\widehat {\bf y}(m) \in W_m\cap  {\mathcal S}(\Phi)$ and $1\le m\ne m'\le M$.
This together with \eqref{Juspace.def} and Theorem \ref{fiber.thm} proves the desired conclusion.

\subsection{Proof of Theorem \ref{dualframe.thm}}
\label{dualframe.thm.pfsection} $\Longrightarrow$:   Taking Fourier transform at both side of \eqref{dualframe.thm.eq1} yields
\begin{equation} \label{dualframe.thm.pfeq1}
\widehat \x(m)=\sum_{i=1}^r \sum_{m'=1}^M
\big\langle \widehat \x(m'), \widehat {\tilde \phi}_i(m')\big\rangle
\prod_{l=1}^d 
[1-\gamma_l(m)\gamma_l(m')]^{-1} \widehat \phi_i(m),\quad  1\le m\le M.
\end{equation}
As $\widehat \x(m)\in {\mathcal S}(\Phi)\cap W_m, 1\le m\le M$, can be chosen independently by Theorem \ref{rangecharacterization.thm}, we have
\begin{equation} \label{dualframe.thm.pfeq2}
\sum_{i=1}^r 
\big\langle \widehat \x(m'), \widehat{\tilde \phi}_i(m')\big\rangle
\widehat \phi_i(m)=0,\quad  1\le m'\ne m\le M, 
\end{equation}
and 
\begin{equation} \label{dualframe.thm.pfeq3}
\widehat \x(m)=\sum_{i=1}^r \big\langle \widehat \x(m), \widehat {\tilde \phi}_i(m)\big\rangle 
\prod_{l=1}^d [1-\gamma_l(m)^2]^{-1} \widehat\phi_i(m), \quad 1\le m\le M.
\end{equation}
By \eqref{dualframe.thm.pfeq3}, Theorem \ref{fiber.thm} and  the assumption $\tilde \phi_i\in {\mathcal S}(\Phi), 1\le i\le r$, we conclude that
\begin{equation}  \label{dualframe.thm.pfeq4}
{\rm span}\Big \{\widehat{\tilde \phi}_i(m)\ |\ 1\le i\le r\Big\}=
{\rm span} \big\{\widehat{\phi}_i(m)\ |\ 1\le i\le r\big\}={\mathcal S}(\Phi)\cap W_m, \quad 1\le m\le M.
\end{equation}
This together with Theorem \ref{fiber.thm} proves
\eqref{dualframe.thm.eq2}.

Replacing  ${\bf x}$ by $\phi_j, 1\le j\le r$ in 
\eqref{dualframe.thm.pfeq2} and \eqref{dualframe.thm.pfeq3}, and then taking inner product with $\tilde \phi_{j'}, 1\le j'\le r$, proves \eqref{dualframe.thm.eq3}.

\smallskip
 $\Longleftarrow$:\quad By Theorems \ref{rangecharacterization.thm} and \ref{fiber.thm}, it suffices to verify
 \eqref{dualframe.thm.pfeq1} for all ${\bf x}\in {\mathcal S}(\Phi)$.
 Therefore it reduces to verify that
 \begin{equation} \label{dualframe.thm.pfeq5} 
{\bf u}_j(m)=0
\end{equation}
hold for all $1\le j\le r$ and $1\le m\le M$, where
\begin{equation}\label{dualframe.thm.pfeq6}  {\bf u}_j(m):=\widehat \phi_j(m)-\sum_{i=1}^r \sum_{m'=1}^M
\big\langle \widehat{\phi}_j(m'), \widehat {\tilde \phi}_i(m')\big\rangle
\prod_{l=1}^d 
[1-\gamma_l(m)\gamma_l(m')]^{-1} \widehat \phi_i(m).
\end{equation}
By \eqref{dualframe.thm.eq3} and \eqref{dualframe.thm.pfeq6}, we have
\begin{equation} \label{dualframe.thm.pfeq7} 
{\bf u}_j(m)\in {\mathcal S}(\Phi)\cap W_m,\ 1\le j\le r,\end{equation}
and
\begin{equation} \label{dualframe.thm.pfeq8}
\big\langle {\bf u}_j(m), \widehat {\tilde \phi}_{j'}(m)\big\rangle =0 \ {\rm for \ all}\ 1\le j, j'\le r \ {\rm and} \ 1\le m\le M.
\end{equation}
  By \eqref{dualframe.thm.eq2} and Theorem \ref{fiber.thm}, we see that
 \eqref{dualframe.thm.pfeq4} holds. Combining 
  \eqref{dualframe.thm.pfeq4},  \eqref{dualframe.thm.pfeq7}  and \eqref{dualframe.thm.pfeq8} proves \eqref{dualframe.thm.pfeq5}, and hence the  frame reconstruction formula \eqref{dualframe.thm.eq1}.

\subsection{Proof of Theorem \ref{dualframeexistence.thm}}
\label{dualframeexistence.thm.pfsection}

To  prove Theorem \ref{dualframeexistence.thm}, we require a technical lemma on oblique projections.

\begin{lemma}\label{sumprojection.lem}
Let $X(1), \ldots, X(M)$ be linear subspaces of ${\mathbb R}^r$ such that their sum is direct.  Then there exist oblique projections ${\bf Q}_m, 1\le m\le M$, such that 
\begin{equation} \label{sumprojection.lem.eq1} {\bf Q}_m^2={\bf Q}_m\  {\rm and}\   {\bf Q}_m {\mathbb R}^r=X(m)\ {\rm for\ all} \ 1\le m\le M,\end{equation}
and 
\begin{equation}\label{sumprojection.lem.eq2} {\bf Q}_mX(m')=\{0\} \ {\rm for \ all}\ 1\le m\ne m'\le M.
\end{equation}
\end{lemma}

Although the conclusion in Lemma \ref{sumprojection.lem} appears in the literature, we provide a sketch of its proof  for the completeness of this paper.

\begin{proof} [Proof of Lemma \ref{sumprojection.lem}] Let $1\le m\le M$ and  $w_{m, i}, 1\le i\le \dim X(m)$, be a Riesz basis of the linear space  $X(m)$. Then by the direct sum assumption, $\{w_{m,i} \ | \ 1\le i\le \dim X(m), 1\le m\le M\}$ is a Riesz basis for $X(1)+\cdots+X(M)$. Then there exist 
$\tilde w_{m, i}, 1\le i\le \dim X(m), 1\le m\le M$, such that
\begin{equation}\label{dualframeexistence.lem.pfeq1}
    \langle \tilde w_{m,i}, w_{m', i'}\rangle=0 
\end{equation}
for all $1\le i\le \dim X(m), 1\le i'\le \dim X({m'}), 1\le m\ne m'\le M$,
and
\begin{equation}  \label{dualframeexistence.lem.pfeq2}
    \langle \tilde w_{m,i}, w_{m, i'}\rangle =\delta_{ii'}\end{equation}
for all $1\le i, i'\le \dim X(m)$ and $1\le m\le M$,    where $\delta$ is the conventional Kronecker symbol.
The above set $\{\tilde w_{m, i}, 1\le i\le \dim X(m), 1\le m\le M\}$
can be chosen to be part of the dual Riesz basis corresponding to a Riesz basis of  ${\mathbb R}^r$ obtained from extending the  basis  $\{w_{m,i}, 1\le i\le \dim X(m), 1\le m\le M\}$ for  the  linear subspace $X(1)+\ldots+X(M)$.
Define
$${\bf Q}_m=\sum_{i=1}^{\dim X(m)}  w_{m, i}  {\tilde w}_{m, i}^T, \quad 1\le m\le M.$$ 
Then one may verify that ${\bf Q}_m, 1\le m\le M$, are the oblique projections satisfying
the desired properties \eqref{sumprojection.lem.eq1} and \eqref{sumprojection.lem.eq2}. 
\end{proof}
Now we prove Theorem \ref{dualframeexistence.thm}.
\begin{proof}[Proof of Theorem  \ref{dualframeexistence.thm}]
$\Longrightarrow$:\quad Let $\tilde \Phi=\{\tilde \phi_1, \ldots, \tilde \phi_r\}$ generates a shift-invariant dual frame. 
By Theorem \ref{dualframe.thm}, ${\bf B}(m)$ are a multiple of an oblique projection on its range space for every $1\le m\le M$. On the other hand,  the range space of the matrix ${\bf B}(m)$ and ${\bf C}(m)$ are the same by 
\eqref{dualframe.thm.pfeq4}. 
Hence
${\bf B}(m)$ is the  multiple of an oblique projection operator on  $X(m)$.
By \eqref{dualframe.thm.eq3}, we see that
\begin{equation} \label{dualframeexistence.thm.pfeq1} {\bf B}(m') {\bf y}=0\ {\rm  for\ all} \ {\bf y}\in X(m)  \ {\rm and} \ 1\le m'\ne m\le M\end{equation}
and
\begin{equation} \label{dualframeexistence.thm.pfeq2} {\bf B}(m) {\bf y}= \Big(\prod_{l=1}^d \big(1-\gamma_l(m)^2\big)\Big) {\bf y}\ {\rm  for\ all} \ {\bf y}\in X(m).\end{equation}
Then 
the direct  sum conclusion follows from \eqref{dualframeexistence.thm.pfeq1} and \eqref{dualframeexistence.thm.pfeq2}, since for any ${\bf y}=\sum_{m=1}^M  {\bf y}_m$
with  ${\bf y}_m\in X(m), 1\le m\le M$, one may verify that
$$  {\bf y}_m= \Big(\prod_{l=1}^d(1-\gamma_l(m)^2)^{-1}\Big) {\mathbf B}(m) {\bf y}, \ 1\le m\le M,$$ 
are uniquely determined from ${\bf y}\in X(1)+\cdots+X(M)$.

\smallskip 

$\Longleftarrow$:\quad  Let $X(m)$ be the range space of matrices ${\bf C}(m), 1\le m\le M$
 in \eqref{Cm.def}. By Lemma \ref{sumprojection.lem} and  the assumption on $X(m), 1\le m\le M$, there exist oblique projections ${\bf Q}_m, 1\le m\le M$ such that
 \eqref{sumprojection.lem.eq1} and \eqref{sumprojection.lem.eq2} hold.
 Set 
 \begin{equation} \label{dualframeexistence.thm.pfeq3} 
 {\bf D}(m)=\Big(\prod_{l=1}^d [1-\gamma_l(m)^2]\Big)  {\bf C}(m)^\dag {\bf Q}_m \end{equation}
 and
write
 ${\bf D}(m)=[D_m(j,j')]_{1\le j, j'\le r}$, 
 where ${\bf C}(m)^\dag$ is the pseudo inverse of the matrix ${\bf C}(m)$.
 Define $\tilde \phi_j$ with the help of GFT by
 \begin{equation} \label{dualframeexistence.thm.pfeq4}
 \widehat{\tilde \phi}_j(m)=  \sum_{j'=1}^r D_m(j',j) \widehat \phi_{j'}(m), \quad 1\le j\le r. 
\end{equation} 
 Then it follows from \eqref{dualframeexistence.thm.pfeq1}
 and \eqref{dualframeexistence.thm.pfeq2} that 
 the matrices ${\bf B}(m), 1\le m\le M$, in \eqref{dualframe.thm.eq0}
 corresponding to $\tilde \Phi=\{\tilde \phi_1, \ldots, \tilde \phi_r\}$
 satisfy
 \begin{equation} \label{dualframeexistence.thm.pfeq5}
     {\bf B}(m)=[\langle \widehat \phi_j(m), \widehat {\tilde \phi}_{j'}\rangle]_{1\le j, j'\le r}
     = {\bf C}(m) {\bf D} (m)= \Big(\prod_{l=1}^d [1-\gamma_l(m)^2]\Big) {\bf Q}_m,\quad 1\le m\le M,
 \end{equation}
where the last equality hold as the range space of ${\bf Q}_m$ is the same as ${\bf C}(m)$. 
Hence \eqref{dualframe.thm.eq3} holds by  \eqref{sumprojection.lem.eq1}, \eqref{sumprojection.lem.eq2} and \eqref{dualframeexistence.thm.pfeq5}.
This  proves  $\tilde \Phi=\{\tilde \phi_1, \ldots, \tilde \phi_r\}$ generates a dual shift-invariant  frame by Theorem \ref{dualframe.thm}. This completes the proof. 
\end{proof}

\subsection{Proof of Theorem \ref{bessel.thm}}
\label{bessel.thm.pfsection}
  Set
${\bf x}=\sum_{\phi\in \Phi} \sum_{{\pmb \alpha}\in {\mathbb Z}_+^d} c_{\pmb\alpha,\phi}{\bf S}^{\pmb\alpha}\phi$.
Then 
\begin{equation}\label{bessel.thm.pfeq1}
\widehat {\bf x} (m)= \sum_{\phi\in \Phi}
\sum_{{\pmb \alpha}\in {\mathbb Z}_+^d}  c_{\pmb\alpha,\phi} {\pmb \gamma}(m)^{\pmb \alpha} \widehat \phi(m)
\end{equation}
by \eqref{projection.eq3}.
Therefore
\begin{eqnarray}\label{bessel.thm.pfeq2}
\|{\bf x}\|_2^2
& = & \sum_{m=1}^M \|\widehat {\bf x}(m)\|_2^2\nonumber\\
& \le &  \sum_{m=1}^M 
\big(\sigma_{\max}(F({\pmb \gamma}(m))\big )^2  \sum_{\phi\in \Phi}
\Big|\sum_{{\pmb \alpha}\in {\mathbb Z}_+^d} c_{\phi, {\pmb \alpha}} {\pmb \gamma}(m)^{\pmb \alpha}\Big|^2\nonumber \\
& \le & \sum_{m=1}^M \big(\sigma_{\max}(F({\pmb \gamma}(m))\big )^2 
 \sum_{\phi\in \Phi}
\Big(\sum_{{\pmb \alpha}\in {\mathbb Z}_+^d} |c_{\phi, {\pmb \alpha}}|^2\Big)\times
\Big(\sum_{{\pmb \alpha}\in {\mathbb Z}_+^d}
|{\pmb \gamma}(m)|^{2\pmb \alpha}\Big)\nonumber \\
&\le & 
 \sum_{m=1}^M \big(\sigma_{\max}(F({\pmb \gamma}(m))\big )^2 
 \Big( \prod_{l=1}^d (1-(\gamma_l(m))^2)^{-1}\Big) \times 
\Big(\sum_{\phi\in \Phi} \sum_{{\pmb \alpha}\in {\mathbb Z}_+^d} |c_{\phi, {\pmb \alpha}}|^2\Big),\quad 
\end{eqnarray}
where the first equality follows from
the the Paserval’s identity \eqref{parsevelidentity} and the second one holds by \eqref{bessel.thm.pfeq1} and the property for maximal singular value of a matrix.
Taking square roots at both sides of \eqref{bessel.thm.pfeq2} proves 
the desired Bessel bound estimate \eqref{bessel.thm.eq2}.

\subsection{Proof of Theorem \ref{riesz.thm}}
\label{riesz.thm.pfsection}
 $\Longrightarrow$: \quad  Suppose, on the contrary, that
\eqref{riesz.thm.eq1} does not hold. Then by Theorem \ref{fiber.thm},
 there exist $1\le m\le M$ and  $c_1,\ldots, c_r$ being not all zeros such that
\begin{equation*}
c_1 \widehat \phi_1(m)+\cdots+ c_r \widehat\phi_r(m)=0.
\end{equation*}
Following the proof of Theorem \ref{FGSIS.thm},  we can find a nonzero polynomial
$Q_m$ of order at most $M-1$ such that
${\bf P}_m= Q_m({\bf T})$.
Therefore
$$c_1  Q_m({\bf T})\phi_1+\cdots+ c_r Q_m({\bf T})\phi_r=0.$$
This contradicts to  the Riesz property for $ {\bf T}^k\phi_l,  0\le k\le M-1, 1\le l\le r$.

$\Longleftarrow$: \  Take $\x\in S(\Phi)$. Then there exist coefficients $c_{k, l}$ such that
\begin{equation*}
\x=\sum_{k=0}^{M-1} \sum_{l=1}^r c_{k,l} {\bf T}^k \phi_l
\end{equation*}
by Theorem \ref{FGSIS.thm}.
Taking Fourier transform at both sides of the above equation yields
\begin{equation*}
\widehat \x (m)=\sum_{l=1}^r \Big(\sum_{k=0}^{M-1} 
c_{k,l} (\mu(m))^k\Big) \widehat \phi_l(m).
\end{equation*}
This together with 
the Riesz property for $\widehat \phi_l(m), 1\le l\le r$, gives
\begin{eqnarray}\label{rieszpf.eq1}
\|\x\|_2^2 \hskip-0.08in& = \hskip-0.05in&\hskip-0.05in  \sum_{m=1}^M \Big\|\sum_{l=1}^r \Big(\sum_{k=0}^{M-1} 
c_{k,l} (\mu(m))^k\Big) \widehat \phi_l(m)\Big\|_2^2\nonumber\\
\hskip-0.08in& \ge \hskip-0.05in&\hskip-0.05in  \sum_{l=1}^r \sum_{m=1}^M ( r_0(m))^2
\Big| \sum_{k=0}^{M-1} 
c_{k,l} (\mu(m))^k \Big|^2\nonumber\\
\hskip-0.08in& \ge\hskip-0.05in& \hskip-0.05in
\Big(\min_{1\le m\le  M} (r_0(m))^2\Big) (\sigma_{\min}({\bf V}))^2
\sum_{k=1}^{M-1}\sum_{l=1}^r |c_{k,l}|^2,
\end{eqnarray}
where $r_0(m)$ is the lower Riesz bound in \eqref{riesz.thm.rem1}, and $\sigma_{\min}({\bf V})$ is the minimal singular value for the Vandermonde matrix ${\bf V}$ in 
\eqref{riesz.thm.rem2}.   By similar arguments, we have 
\begin{equation}\label{rieszpf.eq2}
  \|\x\|_2^2\le  \Big( \max_{1\le m\le  M} (r_1(m))^2\Big) (\sigma_{\max}({\bf V}))^2
\sum_{k=1}^{M-1}\sum_{l=1}^r |c_{k,l}|^2.
\end{equation}
Combining \eqref{rieszpf.eq1} and \eqref{rieszpf.eq2} completes the proof of sufficiency.

\subsection{Proof of Theorem \ref{rkhs.thm1}}
\label{rkhs.thm1.pfsection}
The sufficiency follows directly from \eqref{SIGRKHS.def} and \eqref{rksis.def}. Now we prove the  necessity.  Let $U$ be a shift-invariant space.
By Theorem \ref{fiber.thm}, we have
\begin{equation}
U=\bigoplus_{m=1}^M (W_m\cap U),
\end{equation}
where $W_m, 1\le m\le M$, are mutually orthogonal subspaces of ${\mathbb R}^V$  in  \eqref{projection.eq5}.
Define 
$${\bf K}=\sum_{m=1}^M \sum_{i=1}^{\dim W_m\cap U} 
{\bf u}_{m,i} {\bf u}_{m,i}^T,
$$
where  ${\bf u}_{m,i}, 1\le i\le \dim W_m\cap U=\dim_U({\pmb \gamma}(m))$, form an orthonormal basis of the subspace $W_m\cap U, 1\le m\le M$.
Then one may verify that
\begin{equation}
{\bf K} {\bf S}_l= {\bf S}_l {\bf K}=\sum_{m=1}^M \gamma_l(m)
\sum_{i=1}^{\dim W_m\cap U} 
{\bf u}_{m,i} {\bf u}_{m,i}^T, \quad 1\le l\le d
\end{equation}
by \eqref{projection.eq3} and Assumption \ref{graphshift.assumption}. Hence
${\bf K}$ is a shift-invariant. 

Define 
$$H=\{{\bf K} {\bf c} \ |\  {\bf c}\in {\mathbb R}^V\}.$$
Clearly  $H\subset U$ holds as ${\bf u}_{m, i}\in U$ for all $1\le i\le \dim W_m\cap U, 1\le m\le M$. On the other hand, for any
$\x\in U$, we have
\begin{equation}
\x=\sum_{m=1}^M {\bf P}_m \x= \sum_{m=1}^M \sum_{i=1}^{\dim W_m\cap U} {\bf u}_{m, i} {\bf u}_{m, i}^T {\bf P}_m {\bf x}=
 \sum_{m=1}^M \sum_{i=1}^{\dim W_m\cap U} {\bf u}_{m, i} {\bf u}_{m, i}^T  {\bf x}= {\bf K} \x\in H.
\end{equation}
This proves the desired conclusion that $U$ is a GRKHS with standard inner product on ${\mathbb R}^V$.

\subsection{Proof of Theorem \ref{rkhs.bandlimited.thm}}
\label{rkhs.bandlimited.thm.pfsection}
 First we prove the necessity. 
By \eqref{bandlimited.def}, there exists $\Omega \subset \{1, \ldots, N\}$ such that
$$W_m\cap U=\left\{ \begin{array}{ll} W_m  & {\rm if} \ m\in \Omega\\
\{0\} & {\rm if} \ m\not\in \Omega.
\end{array}\right.$$
Define 
$${\bf K}=\sum_{m\in \Omega} {\bf P}_m.$$
Then ${\bf K}$ is a polynomial filter by  Theorem \ref{projection.polynomialfilter.thm}.
Moreover, one may verify that
$$H=\{ {\bf K}{\bf c}\ |\ {\bf c}\in {\mathbb R}^V\}=\bigoplus_{m\in \Omega} W_m=U$$
and 
$${\bf K} {\bf x}={\bf x}\ {\rm for \ all} \ \x\in U.$$
This show that $U$ is the GRKHS with the reproducing kernel  $K$ being a polynomial filter.

Now the sufficiency.
Let ${\bf K}$ be the reproducing kernel. Then  there exists a polynomial $h$ such that
${\bf K}=h({\mathbf S})$. This together with \eqref{projection.eq2} and \eqref{projection.eq3} implies that
\begin{equation}  \label{rkhs.bandlimited.thm.pfeq5}
{\bf K}=\sum_{m=1}^M {\bf P}_m h({\mathbf S})=\sum_{m=1}^M h ({\pmb \gamma}(m)) {\bf P}_m.
\end{equation}
Set $\Omega=\{1\le m\le M\ | \ h({\pmb \gamma}(m))\ne 0\}$.
By \eqref{SIGRKHS.def} and  \eqref{rkhs.bandlimited.thm.pfeq5}, we see that
\begin{equation*}
   {\bf Kc}=\sum_{m=1}^{M}h(\pmb\gamma(m)){\bf P}_m{\bf c}=\sum_{m\in\Omega}h(\pmb\gamma(m)){\bf P}_m{\bf c}\in \bigoplus_{m\in \Omega} W_m,
\end{equation*}
holds for all ${\bf c}\in \R^V$, and hence 
\begin{equation}\label{rkhs.bandlimited.thm.pfeq6}
   H=\{ {\bf K}{\bf c}\ |\ {\bf c}\in {\mathbb R}^V\}\subset\bigoplus_{m\in \Omega} W_m. 
\end{equation}
On the other hand, take an arbitrary  ${\bf y}=\sum_{m\in \Omega}{\bf y}_m\in \bigoplus_{m\in \Omega} W_m $ with ${\bf y}_m\in {\bf W}_m$. Define ${\bf c}=\sum_{m\in \Omega}\dfrac{{\bf y}_m}{h(\pmb\gamma(m))}$. By \eqref{projection.eq1}, \eqref{SIGRKHS.def} and  \eqref{rkhs.bandlimited.thm.pfeq5},  we have 
\begin{equation*}
   h(\pmb\gamma(m)) {\bf P}_m{\bf c}={\bf y}_m, m\in \Omega,
\end{equation*}
and 
\begin{equation*}
  {\bf y}=\sum_{m\in\Omega}{\bf y}_m=\sum_{m\in \Omega} h(\pmb\gamma(m)){\bf P}_m{\bf c}=\sum_{m=1}^M h(\pmb\gamma(m)){\bf P}_m{\bf c}= {\bf Kc}\in H.
\end{equation*}
Therefore $\bigoplus_{m\in \Omega} W_m\subset H$. This together with \eqref{rkhs.bandlimited.thm.pfeq6} completes the proof.

\subsection{Proof of Theorem \ref{decomposition.thm}}
\label{decomposition.thm.pfsection}

 First we prove the necessity. Let $H_m=H\cap W_m$  equipped with the inner product induced from the space 
 $H$.
By \eqref{Juspace.def} and Theorem \ref{rangecharacterization.thm},  
\eqref{decomposition.thm.eq1} holds for the RKHSs $H_m, 1\le m\le M$, defined above.
By \eqref{projection.eq1}, Theorem \ref{shiftinvariantfilter.thm} and the shift-invariance assumption for the reproducing kernel ${\bf K}$ of the SIGRKHS $H$, we have
\begin{equation}\label{decomposition.thm.pfeq1}
{\bf K}=\sum_{m=1}^M {\bf P}_m {\bf K} {\bf P}_m.
\end{equation}
Then for any ${\bf x}, {\bf y}\in H$ and 
$1\le m\ne m'\le M$, we have
$${\bf x}={\bf K} {\bf c}\  {\rm and} \ {\bf y}={\bf K}{\bf d}$$ for some ${\bf c}, {\bf d}\in {\mathbb R}^V$ by \eqref{SIGRKHS.def},
and 
$$\langle {\bf P}_m {\bf x}, {\bf P}_{m'} {\bf y}\rangle_H= 
\langle {\bf P}_m {\bf K} {\bf c}, {\bf P}_{m'} {\bf K} {\bf d}\rangle_H=
\langle {\bf K} ({\bf P}_m {\bf c}), {\bf K} ({\bf P}_{m'}{\bf d})\rangle_H=
({\bf P}_m{\bf c})^T {\bf K} ({\bf P}_{m'}{\bf d})=0
$$
where the second equality holds since ${\bf K}$
is shift-invariant and ${\bf P}_m, 1\le m\le M$, are polynomial filters by Theorem \ref{projection.polynomialfilter.thm}, and 
the last equality follows from \eqref{decomposition.thm.pfeq1}.
Therefore
$$\langle {\bf x}, {\bf y}\rangle_H= \sum_{1\le m, m'\le M} \langle {\bf P}_m{\bf x}, {\bf P}_{m'}{\bf y}\rangle_H
=\sum_{m=1}^M \langle {\bf P}_m{\bf x}, {\bf P}_m {\bf y}\rangle_{H_m}.
$$
This proves \eqref{decomposition.thm.eq2}.

Next we prove  the sufficiency. Denote the reproducing kernels of GRKHSs $H_m$ by ${\bf K}_m, 1\le m\le M$ and set
\begin{equation}\label{decomposition.thm.pfeq2}
{\bf K}=\sum_{m=1}^M{\bf P}_m{\bf K}_m {\bf P}_m.
\end{equation}
Then it suffices to prove that ${\bf K}$
is a shift-invariant reproducing kernel of the inner product space  $H=\oplus_{m=1}^M H_m$.

Set
$$\tilde  H=\big\{ {\bf K} {\bf c} \   | \ {\bf c}\in {\mathbb R}^V\}.$$ 
First we prove 
\begin{equation} \label{decomposition.thm.pfeq3-}
\tilde H=H.
\end{equation}
Clearly $\tilde H\subset H$ by \eqref{decomposition.thm.pfeq2}. Then  it suffices to prove
\begin{equation} \label{decomposition.thm.pfeq3+}
H\subset \tilde H.
\end{equation}
Take arbitrary ${\bf x}\in H$. Then 
\begin{eqnarray*}
{\bf x}\hskip-0.1in & = \hskip-0.05in& \hskip-0.05in\sum_{m=1}^M {\bf K}_m {\bf c}_m =
\sum_{m=1}^M {\bf P}_m {\bf K}_m {\bf P}_m {\bf c}_m\nonumber\\
\hskip-0.1in & = \hskip-0.05in& \hskip-0.05in
\sum_{m=1}^M {\bf P}_m {\bf K}_m {\bf P}_m \Big(\sum_{m'=1}^M {\bf P}_{m'} {\bf c}_{m'}\Big)
={\bf K}\Big(\sum_{m'=1}^M {\bf P}_{m'} {\bf c}_{m'}\Big)\in \tilde H.
\end{eqnarray*}
where the existence of ${\bf c}_m\in W_m, 1\le m\le M$, follows from \eqref{SIGRKHS.def} and \eqref{decomposition.thm.eq1},  the second equality holds as $H_m\subset W_m, 1\le m\le M$, and the third one is true by \eqref{projection.eq1}.
This proves \eqref{decomposition.thm.pfeq3+}
and hence 
\begin{equation} \label{decomposition.thm.pfeq4}
H=\tilde H=\{ {\bf K}{\bf c}\ |\  {\bf c}\in {\mathbb R}^V\}.
\end{equation}

Write ${\bf A}=[{\bf A}(v, v')]_{v, v'\in V}$
for a matrix ${\bf A}\in {\mathbb R}^{V\times V}$
and ${\bf x}=[{\bf x}(v)]_{v\in V}$ for ${\bf x}\in {\mathbb R}^V$.
Take ${\bf x}\in H$ and $v\in V$. 
Then  we have 
\begin{eqnarray*} 
\langle {\bf x}, {\bf K}(\cdot, v)\rangle_H
\hskip-0.1in & = \hskip-0.05in& \hskip-0.05in \sum_{m=1}^M \langle {\bf P}_m {\bf x},
({\bf P}_m {\bf K})(\cdot, v)\rangle_{H_m}\nonumber\\
\hskip-0.1in & = \hskip-0.05in& \hskip-0.05in \sum_{m=1}^M
\langle {\bf P}_m {\bf x},
{\bf K}_m (\cdot, v)\rangle_{H_m}= \sum_{m=1}^M  ({\bf P}_m{\bf x})(v) ={\bf x} (v),
\end{eqnarray*}
where the first equality holds by \eqref{decomposition.thm.eq2}, the second equality is true  by \eqref{decomposition.thm.pfeq2}, the third equality  follows since $K_m$ are the reproducing kernels of  the GRKHS $H_m, 1\le m\le M$.

Finally, we prove the reproducing kernel property.
Take ${\bf x}\in H$ and $v\in V$.
By the reproducing property of the space $H_m, m\in \Omega$, there exist constants $c_m(v)$ such that
\begin{equation*} 
    | {\bf x}_m(v)|\le c_m(v) \|{\bf x}_m\|_{H_m} \ {\rm for \ all}\ {\bf x}_m\in  H_m, 1\le m\le M.
\end{equation*}
Therefore
\begin{eqnarray*}
    |\x(v)|\hskip-0.1in & \le \hskip-0.05in& \hskip-0.05in \sum_{m=1}^M |({\bf P}_m \x)(v)|
\le \sum_{m=1}^M  c_m(v) \|{\bf P}_mx\|_{H_m}\\
\hskip-0.1in & \le \hskip-0.05in& \hskip-0.05in\Big(\sum_{m=1}^M | c_m(v)|^2\Big)^{1/2}
\times \Big(\sum_{m=1}^M \|{\bf P}_m\x\|_{H_m}^2\Big)^{1/2}\le \Big(\sum_{m=1}^M |c_m(v)|^2\Big)^{1/2} \|{\bf x}\|_H,
\end{eqnarray*}
where the first inequality follows from \eqref{decomposition.thm.pfeq4}.
This proves that $H$ is a GRKHS with  reproducing kernel $K$.

On the other hand, we have
$${\bf S}_k{\bf K}=\sum_{m\in \Omega}
\gamma_l(m) {\bf P}_m {\bf K}_m {\bf P}_m=
{\bf K} {\bf S}_l, \quad 1\le l\le d.$$
by \eqref{projection.eq3} and \eqref{decomposition.thm.pfeq2}. Therefore the reproducing kernel $\bf K$ is shift-invariant. This completes the proof. 

\subsection{Proof of Theorem \ref{rkhsisometric.thm}}
\label{rkhsisometric.thm.pfsection}


$\Longrightarrow$: \  Take $1\le m\le M$ and  ${\bf x}\in W_m\cap H_1$. Observe  from \eqref{shiftinvariantfilter.projection}
and the shift-invariance assumption on the linear operator $T$ that
${\bf P}_mT {\bf P}_m {\bf x}=T{\bf x}\in W_m\cap H_2$. Then the isometric property
for the linear operator ${\bf P}_mT{\bf P}_m$ follows from the isometric property for the linear operator $T$.

$\Longleftarrow$: \ Take ${\bf x}\in H_1$. By \eqref{projection.eq1}, \eqref{shiftinvariantfilter.projection}, \eqref{decomposition.thm..eq5} and
Theorem \ref{decomposition.thm},
we have
\begin{eqnarray}
\|T{\bf x}\|_{H_2}^2\hskip-0.1in & =\hskip-0.05in &\hskip-0.05in  \sum_{m=1}^M 
\langle {\bf P}_m T {\bf x}, {\bf P}_m T {\bf x}\rangle_{H_2}=
\sum_{m=1}^M 
\langle {\bf P}_m T {\bf P}_m ({\bf P}_m{\bf x}), {\bf P}_m T {\bf P}_m ({\bf P}_m{\bf x})\rangle_{H_2}\nonumber\\
\hskip-0.1in & =\hskip-0.05in &\hskip-0.05in \sum_{m=1}^M 
\langle  {\bf P}_m {\bf x},  {\bf P}_m {\bf x}\rangle_{H_1}= \|{\bf x}\|_{H_1}.
\end{eqnarray}
This proves the isometric property for the linear operator $T$.

\begin{appendix}

\section{Graph Fourier transform on circulant graphs}
\label{circulantgraph.section} 

Let   $N\ge 1$ and 
write  $a=b\ {\rm mod }\ N$ if $(a-b)/N$ is an integer.
In this appendix, we consider the unweighted  circulant graph
${\mathcal C}:={\mathcal C}(N, Q)$  generated by  $Q=\{q_1, \ldots, q_d\}$ of positive integers 
 with  the vertex set   $V_N=\{0, 1, \ldots, N-1\}$  and the edge set
$E_N(Q)=\{(i, i\pm q\ {\rm mod}\ N),\  i\in V_N, q\in Q\}$,
where   $q_i\in Q, 1\le i\le d$ are ordered with $1\le q_1<\cdots< q_d\le N/2$ and they satisfy the condition that 
  $q_1, \ldots, q_d,N$  are  coprime \cite{emirov2022, kotzag2019}. 
  Take the symmetrically normalized Laplacians on the circulant subgraphs ${\mathcal C}(N, \{q_l\})$
 as  graph shifts
${\bf S}_l=(S_l(i ,j))_{i,j\in V_N},1\le l\le d$, where their entries are defined by $S_l(i,j)=1$ if $j=i$, $S_l(i,j)=-1/2$ if $i-j=\pm q_l\ {\rm mod} \ N$ and $S_l(i,j)=0$ otherwise. 
In this appendix, we  provide explicit description
for the joint spectrum  of the above graph shifts ${\bf S}_l, 1\le l\le d$, on the circulant graph ${\mathcal C}(N, Q)$.
 In addition, we derive the corresponding spectral decomposition \eqref{projection.eq0} of the identity operator, together with the associated GFT.

As the graph shifts ${\bf S}_l, 1\le l\le d$, are symmetric circulant matrices, they can be diagonalized simultaneously,
\begin{equation}\label{circulantgraph.example.eq1}
{\bf S}_l= {\bf U} {\pmb \Lambda}_l {\bf U}^T,
\end{equation}
where
\begin{equation} \label{circulantgraph.example.eq2}
{\bf U}=\big[{\bf u}_0, {\bf u}_1, {\bf v}_1, \ldots, {\bf u}_{(N-1)/2}, {\bf v}_{(N-1)/2}\big] 
\end{equation}
and
\begin{equation} \label{circulantgraph.example.eq3}
{\pmb \Lambda}_l ={\rm diag} \Big(0, 1-\cos \theta_l, 1-\cos \theta_l, 
\cdots, 1-\cos \frac{N-1}{2}\theta_l, 1-\cos\frac{N-1}{2} \theta_l\Big)
\end{equation}
for odd $N$, and
\begin{equation} \label{circulantgraph.example.eq4}
{\bf U} =\big[{\bf u}_0, {\bf u}_1, {\bf v}_1, \ldots, {\bf u}_{(N-2)/2}, {\bf v}_{(N-2)/2}, {\bf u}_{N/2}\big ]
\end{equation}
and 
\begin{equation} \label{circulantgraph.example.eq5} {\pmb \Lambda}_l ={\rm diag} \Big(0, 1-\cos \theta_l, 1-\cos \theta_l, 
\cdots, 1-\cos \frac{N-2}{2}\theta_l, 1-\cos  \frac{N-2}{2} \theta_l, 1-\cos \frac{N}{2}\theta_l\Big)\end{equation}
for even $N$,
where   $\theta_l= 2\pi q_l/N$ for $1\le l\le d$, 
$${\bf v}_k=(2/N)^{1/2}\Big[0, \sin \frac{2\pi k}{N}, \ldots, 
 \sin \frac{2\pi  k (N-1) }{N}\Big]^T$$
 if $1\le k\le  
 \lfloor (N-1)/2\rfloor$ (the largest integer no larger than $(N-1)/2$), 
 and  $  {\bf u}_0= N^{-1/2} {\bf 1}$, 
$${\bf u}_k=(2/N)^{1/2}\Big[1, \cos \frac{2\pi k}{N}, \ldots, 
 \cos \frac{2\pi  k (N-1) }{N}\Big]^T$$ 
 if $1\le k\le  \lfloor (N-1)/2\rfloor$, and
 ${\bf u}_{N/2}= N^{-1/2} [1, -1, \ldots, 1, -1]^T$ if $N$ is even.
Then the joint spectrum of the commutative graph shifts
${\bf S}_1, \ldots, {\bf S}_d$  are given by 
\begin{equation}  \label{circulantgraph.example.eq6}
{\pmb \Lambda}:=\big\{{\pmb \lambda}(n)\ | \ 1\le n\le N\big\} \subset {\mathbb R}^d,
\end{equation}
where
${\pmb \lambda}(1)=\left[0, \ldots, 0\right]^T$,
${\pmb \lambda}(2k)= {\pmb \lambda}(2k+1)=
\left[1-\cos k \theta_1, \ldots, 1-\cos k\theta_{d}\right]^T
$ for $1\le k\le \lfloor (N-1)/2\rfloor$, and 
${\pmb \lambda}(N)=[1-(-1)^{q_1}, \ldots,  1-(-1)^{q_d}]^T$
if $N$ is even.

Denote the distinct elements in the joint spectrum by ${\pmb \gamma}(m), 1\le m\le M$.
By the coprime assumption among $q_1, \ldots, q_d$ and $N$, we see that
${\pmb \lambda}(1)\ne {\pmb \lambda}(2k)$ for all $1\le k\le
 \lfloor (N-1)/2\rfloor$
for odd $N$, and ${\pmb \lambda}(1)\ne {\pmb \lambda}(N)$, and
$ {\pmb \lambda}(2k)\ne {\pmb \lambda}(1)$
and ${\pmb \lambda}(2k)\ne {\pmb \lambda}(N)$ hold for all $1\le k\le  \lfloor (N-1)/2\rfloor$
for even $N$.
We remark that ${\pmb \lambda}(2k), 1\le k\le \lfloor (N-1)/2\rfloor$,
are not necessarily distinct; see Proposition 
\ref{doubledimension.pr} for a necessary and sufficient condition under which this distinctness condition is satisfied.
Then we may order ${\pmb \mu}(m), 1\le m\le M$,
as follows: 
\begin{equation}
{\pmb \gamma}(m)=\left\{
\begin{array}{ll}
    \left[0, \ldots, 0\right]^T &  m=1 \\ 
     \left[1-\cos k_m \theta_1, \ldots, 1-\cos k_m\theta_{d}\right]^T&  2\le m\le M,
\end{array}
\right.
\end{equation} where $1\le k_m\le \lfloor (N-1)/2\rfloor, 2\le m\le M$ for odd $N$, and 
\begin{equation}
{\pmb \gamma}(m)=\left\{
\begin{array}{ll}
    \left[0, \ldots, 0\right]^T &  m=1 \\ 
     \left[1-\cos k_m \theta_1, \ldots, 1-\cos k_m\theta_{d}\right]^T&  2\le m\le M-1\\
     \left[1-(-1)^{q_1}, \ldots,  1-(-1)^{q_d}\right]^T & m=M
\end{array}
\right.
\end{equation} for even $N$, where $1\le k_m\le \lfloor (N-1)/2\rfloor, 2\le m\le M-1$.
The eigenspaces $W_m,1\le m\le M$, corresponding to the joint spectrum $\pmb\mu(m)$ are given by 
\begin{equation}\label{rangecirculant.def}
W_m=\left\{
\begin{array}{ll}
    \span\{ {\bf 1}\} &  m=1 \\ 
     \span\{{\bf v}_k,{\bf u}_k \ {\rm with} \ k \ {\rm satisfying }\ \pmb\lambda(2k)=\pmb\gamma(m)\}&  2\le m\le M,
\end{array}
\right.
\end{equation} where ${\bf u}_k,{\bf v}_k, 1\le k\le \lfloor (N-1)/2\rfloor$ are given in \eqref{circulantgraph.example.eq2} for odd $N$, and 
\begin{equation}
W_m=\left\{
\begin{array}{ll}
    \span \{\bf 1\} &  m=1 \\ 
     \span\{{\bf v}_k,{\bf u}_k \ {\rm with} \ k \ {\rm satisfying }\ \pmb\lambda(2k)=\pmb\gamma(m)\}&  2\le m\le M-1\\
     \span\{{\bf u}_{N/2}\} & m=M
\end{array}
\right.
\end{equation} for even $N$, where ${\bf u}_k,{\bf v}_k, 1\le k\le \lfloor (N-1)/2\rfloor$ and ${\bf u}_{N/2}$ are given in \eqref{circulantgraph.example.eq4}.

Denote the  greatest common divisor between two integers
$m$ and $n$ by ${\rm gcd}(m, n)$.
We say that the family of two nonempty subsets $Q_1$ and $Q_2$ form a {\em partition} of the set $Q=\{q_1, \ldots, q_d\}$ if  $Q_1\cap Q_2=\emptyset$ and $Q_1\cup Q_2=Q$. We conclude this appendix with  a necessary and sufficient condition on the set $Q$ for the  distinct property for
${\pmb \lambda}(1)$ and ${\pmb \lambda}(2k), 1\le k\le (N-1)/2 $ in
\eqref{circulantgraph.example.eq6} for odd $N$ and 
 ${\pmb \lambda}(1), {\pmb \lambda}(N)$ and
${\pmb \lambda}(2k), 1\le k\le N/2-1$ for even $N$.

\begin{proposition}\label{doubledimension.pr}
Let  $Q=\{q_1, \ldots, q_d\}$  be of positive integers ordered with $1\le q_1<\ldots< q_d\le N/2$ and $q_1, \ldots, q_d,N$ being coprime, and let ${\pmb \lambda}(n), 1\le n\le N$, be as in \eqref{circulantgraph.example.eq6}.  Then the following statements hold.

\begin{itemize}
\item[{(i)}]
For odd $N$,
${\pmb \lambda}(1)$ and ${\pmb \lambda}(2k), 1\le k\le (N-1)/2 $,
are distinct if and only if for any partition $\{Q_1, Q_2\}$ of  the set $Q$,
either 
the greatest  common divisor among N and all $q\in Q_1$  is one, or
the greatest common divisor among  $N$ and all $q\in Q_2$  is one. 

\item[{(ii)}]
For even $N$, ${\pmb \lambda}(1), {\pmb \lambda}(N/2)$ and
${\pmb \lambda}(2k), 1\le k\le N/2-1$,
are distinct if and only if 
for any partition $Q_1, Q_2$ of the set $Q$ one of the following requirements are satisfied: 
\begin{itemize}
    \item [{(a)}] ${\rm gcd}(r_1, N)=1$;
    
    \item[{(b)}] ${\rm gcd}(r_2, N)=1$;

  \item[{(c)}]  $N\ne 0 \ {\rm mod} \ 4$; 
  ${\rm gcd}(r_1, N)=2$ and ${\rm gcd}(r_2, N)\ge 3$;
  
  \item[{(d)}]  $N\ne 0 \ {\rm mod} \ 4$;
 ${\rm gcd}(r_1, N)\ge 3$ and ${\rm gcd}(r_2, N)= 2$;

    \item[{(e)}]  $N\ne 0{\ \rm  mod }\ 12$,   ${\rm gcd}(r_1, N)=3$ and ${\rm gcd}(r_2, N)=4$; and
    
    \item[{(f)}] $N\ne 0{\ \rm  mod }\ 12$,
    ${\rm gcd}(r_1, N)=4$ and ${\rm gcd}(r_2, N)=3$,
 \end{itemize}
where    $r_1$ and $r_2$ are the greatest common divisors  of all $q\in Q_1$ and 
$q\in Q_2$ respectively. 
\end{itemize}
\end{proposition}


\begin{proof}   (i):  
First the necessity.  Suppose, on the contrary,  that there is a partition of the set $Q$  into two nonoverlapping groups $Q_1$ and $Q_2$ such that the maximal divisor between $N$ and all two groups in the partition is larger than $2$. Without loss of generality, 
we assume that  $Q_1=\{q_1, \ldots, q_e\}$ and  $Q_2=\{q_{e+1}, \ldots q_d\}$. By the coprime assumption on  $q_1, \ldots, q_d$ and $N$, we have that $1\le e<d$. Denote the greatest common divisor of the first group  and the second group by $r_1={\rm gcd}(q_1, \ldots, q_e)$ and  $r_2={\rm gcd}(q_{e+1}, \ldots, q_d)$
respectively.  By the coprime assumption among $q_1, \ldots, q_d$ and $N$, we may,  without loss of generality, that
\begin{equation}\label{doubledimension.pr.pfeq1}
{\rm gcd}(r_2, N)> {\rm gcd}(r_1, N)\ge 2.
\end{equation}
Set 
$$k_1=\frac{1}{2}\Big(\frac{N}{{\rm gcd}(r_1, N)}+ \frac{N}{{\rm gcd}(r_2, N)}\Big) \ 
\ {\rm and} \ \  k_2=\frac{1}{2}\Big(\frac{N}{{\rm gcd}(r_1, N)}-\frac{N}{{\rm gcd}(r_2, N)}\Big).$$
Then $k_1, k_2$ are integers satisfying $1\le k_2< k_1\le (N-1)/2$  by \eqref{doubledimension.pr.pfeq1}
and the assumption $N$ is odd. 

Let   $\theta_i= 2\pi q_i/N$ for $1\le i\le d$. Then   we observe that 
 $k_1\theta_i=-k_2\theta_i \ {\rm mod} \  2\pi$ for $1\le i\le e$
 and  $k_1\theta_i=k_2\theta_i \ {\rm mod} \  2\pi$ for $e+1\le i\le d$. 
 Hence
$\cos k_1 \theta_i=\cos k_2 \theta_i$ for all $1\le i\le d$.
This together with \eqref{circulantgraph.example.eq3}  and \eqref{circulantgraph.example.eq6} implies that
${\pmb \lambda}(2k_1)={\pmb \lambda}(2k_2)$,
which is a contradiction. This proves the necessity.

Now we prove the sufficiency.  Suppose, on the contrary, that
two of ${\pmb \lambda}(1)$ and ${\pmb \lambda}(2k), 1\le k\le (N-1)/2 $,
are the same.  By the coprime assumption among $q_1, \ldots, q_d$ and $N$, we see that ${\pmb \lambda}(1)\ne {\pmb \lambda}(2k)$ for all $1\le k\le (N-1)/2$. Therefore there exist integers $1\le k_1\ne k_2\le (N-1)/2$ such that
${\pmb \lambda}(2k_1)={\pmb \lambda}(2k_2)$. 
Following the argument in the necessity, we can find a partition of the set $Q$ into $Q_1$ and $Q_2$ such that $Q_1\cap Q_2=\emptyset, Q_1\cup Q_2=Q$,
$k_1 q=-k_2 q \ {\rm mod} \ N$ for $q\in Q_1$
and  $k_1 q=k_2 q\ {\rm mod} \ N$ for $q\in Q_2$.

\smallskip 
{\em  Case 1: \   $Q_1=\emptyset$.} 

In this case,  $(k_1-k_2) q_l=0 \ {\rm mod} \ N$ for all $1\le l\le d$. 
This together with the coprime assumption among $q_1, \ldots, q_d$ and $N$
implies that $k_1-k_2=0 \ {\rm mod} \ N$, which is a contradiction as
$-N/2\le k_1-k_2\le N/2$ and $k_1\ne k_2$.

\smallskip
{\em  Case 2: \ $Q_2=\emptyset$.}

In this case,  $(k_1+k_2) q_l=0 \ {\rm mod} \ N$ for all $1\le l\le d$. 
This together with the coprime assumption among $q_1, \ldots, q_d$ and $N$
implies that $k_1+k_2=0 \ {\rm mod} \ N$, which is a contradiction as
$2\le k_1+k_2\le N-1$.

\smallskip

{\em  Case 3:  \ $Q_1\ne \emptyset$ and $Q_2\ne \emptyset$.}

Denote the greatest common divisor of integers in $Q_1$  and $ Q_2$ 
by $r_1$ and $r_2$ respectively. Following the arguments in Cases 1 and 2, we can show that
\begin{equation}\label{case3.eq1}
    (k_1+k_2) r_1=0 \ {\rm mod} \ N \ \ {\rm  and} \ \
(k_1-k_2) r_2=0 \ {\rm mod} \ N.
\end{equation}
By  the assumption of the sufficiency, either  ${\rm gcd}(r_1, N)=1$ or 
${\rm gcd}(r_2, N)=1$. This together with \eqref{case3.eq1}
implies that
either $k_1+k_2=0 \ {\rm mod} \ N$
or $k_1-k_2=0 \ {\rm mod} \ N$, which is a contradiction 
as $k_1+k_2\in [2, N-1]$ and
$k_1-k_2\in [-(N-1)/2, (N-1)/2]\backslash \{0\}$.

\medskip

(ii)\ \
  First the necessity. Suppose, on the contrary, that there is a partition  $\{Q_1, Q_2\}$ of the set $Q$ such that
the  greatest common divisor  $r_1$ of $q\in Q_1$ and
the  greatest common divisor  $r_2$ of $q\in Q_2$ satisfy
\begin{equation}\label{N2.eq1}
{\rm gcd}(r_1, N)\ge 2\ \  {\rm and} \ \ {\rm gcd}(r_2, N)\ge 2.
\end{equation}
Without loss of generality, we assume that
${\rm gcd}(r_2, N)>{\rm gcd}(r_1, N)$.

\smallskip 

{\em Case N1: \ ${\rm gcd}(r_1, N)\ge 4$.}

Set 
\begin{equation}\label{case4.eq1}
    k_1= \frac{N}{ {\rm gcd}(r_1, N)}+\frac{N}{ {\rm gcd}(r_2, N)}\ \  {\rm and}\ \ 
 k_2= \frac{N}{ {\rm gcd}(r_1, N)}-\frac{N}{ {\rm gcd}(r_2, N)}.
\end{equation}
Then  one may verify that 
$1\le k_1\ne k_2<N/2$ and 
${\pmb \lambda}(2k_1)={\pmb \lambda}(2k_2)$.
This is a contradiction. 

\smallskip 

{\em Case N2:\ ${\rm gcd}(r_1, N)=3$ and ${\rm gcd}(r_2, N)\ge 7$.}

One may verify that integers $k_1$ and $k_2$ in \eqref{case4.eq1} 
satisfy $1\le k_1\ne k_2<N/2$ and 
${\pmb \lambda}(2k_1)={\pmb \lambda}(2k_2)$. This is a contradiction.

\smallskip

{\em Case N3:\  ${\rm gcd}(r_1, N)=3$ and ${\rm gcd}(r_2, N)=6$}.

This is a contradiction as ${\rm gcd}(r_1, r_2, N)={\rm gcd}(q_1, \ldots, q_d, N)=1$.

\smallskip

{\em Case N4: \ ${\rm gcd}(r_1, N)=3$ and ${\rm gcd}(r_2, N)=5$.}

In this case, we set
\begin{equation}\label{case7.eq1}
k_1= \frac{1}{2}\Big(\frac{N}{ {\rm gcd}(r_1, N)}+\frac{N}{ {\rm gcd}(r_2, N)}\Big)\ \ 
{\rm and} \  \ 
k_2=\frac{1}{2}\Big(\frac{N}{ {\rm gcd}(r_1, N)}-\frac{N}{ {\rm gcd}(r_2, N)}\Big).
\end{equation}
Then we may verify that $1\le k_1\ne k_2\le (N-1)/2$ and 
${\pmb \lambda}(2k_1)={\pmb \lambda}(2k_2)$. This is a contradiction.

\smallskip

{\em Case N5:\ ${\rm gcd}(r_1, N)=3$ and ${\rm gcd}(r_2, N)=4$.}

In this case, we have $N=0 \ {\rm mod} \ 12$.
Set
$k_1={ 5N}/{12}$ and $k_2={N}/{12}$.
 Then 
$1\le k_1\ne k_2\le N/2-1$, 
  $4(k_1+k_2)=2N$ and $3(k_1-k_2)=N$, which implies that
 ${\pmb \lambda}(2k_1)={\pmb \lambda}(2k_2)$.
 This is a contradiction.  

\smallskip 

{\em Case N6:\ ${\rm gcd}(r_1, N)=2$, ${\rm gcd}(r_2, N)\ge 3$.}

In this case, we have  $N=0 \ {\rm mod} \ 4$ by Assumption (c) and (d), and 
${\rm gcd}(r_2, N)$ is an odd integer by the coprime assumption among $q_1, \ldots, q_N$ and $N$.
 Set 
$$k_1= \frac{N}{4 {\rm gcd}(r_2, N)}\ \ {\rm and} \ \ k_2= \frac{N}{2}-k_1.$$
Then one may verify that $1\le k_1\ne k_2\le N/2-1$,
$2(k_1+k_2)=N$ and ${\rm gcd}(r_2, N)(k_2-k_1)= ({\rm gcd}(r_2, N)-1)N/2=0 \ {\rm mod}\ N$.
This implies  that
 ${\pmb \lambda}(2k_1)={\pmb \lambda}(2k_2)$, which is  a contradiction.

\medskip

Now we prove the sufficiency.  Suppose, on the contrary, that
two of ${\pmb \lambda}(1), {\pmb \lambda}(N)$ and ${\pmb \lambda}(2k), 1\le k\le N/2-1 $
are the same.  By the coprime assumption among $q_1, \ldots, q_d$ and $N$, we see that ${\pmb \lambda}(1)\ne {\pmb \lambda}(N)$,
${\pmb \lambda(1)}\ne {\pmb \lambda}(2k)$ for all $1\le k\le N/2-1$
and ${\pmb \lambda}(N/2)\ne {\pmb \lambda}(2k)$ for all $1\le k\le N/2-1$.
Then there exist integers $1\le k_1\ne k_2\le N/2-1$ such that
${\pmb \lambda}(2k_1)={\pmb \lambda}(2k_2)$.  Following the argument in the sufficiency of the first conclusion, we conclude that for $1\le l\le d$, 
$$ {\rm either} \ (k_1-k_2) q_l=0  \ {\rm mod} \ N \ {\rm  or}\ 
(k_1+k_2) q_l= 0  \ {\rm mod} \ N.$$ 
Denote the set of all $q_l\in Q$ satisfying
 $(k_1+k_2) q_l= 0  \ {\rm mod} \ N$ by $Q_1$ and all other $q_l\in Q$
 by $Q_2$.  Following the argument used in the proof of the sufficiency in the conclusion for odd $N$, we can show that $Q_1$ and $Q_2$ are not empty sets. 
 Then  $Q_1, Q_2$ forms a partition of the set $Q$.
 Denote the greatest common divisor among $q\in Q_1$ by $r_1$
 and the greatest common divisor among $q\in Q_2$  by $r_2$.
Then by the sufficiency assumption, it suffices to one of the following conditions are satisfied: 1) either ${\rm gcd}(r_1, N)=1$ or  ${\rm gcd}(r_2, N)=1$, or 2)  ${\rm gcd}(r_1, N)=2$, ${\rm gcd}(r_2, N)\ge 3$ and $N\ne 0 \ {\rm mod} \ 4$,  or 3) ${\rm gcd}(r_1, N)\ge 3$, ${\rm gcd}(r_2, N)=2$ and $N\ne 0 \ {\rm mod} \ 4$, or 4) 
${\rm gcd}(r_1, N)=3$, ${\rm gcd}(r_2, N)\ge 4$ and $N\ne 0 \ {\rm mod} \ 12$, or 5) 
${\rm gcd}(r_1, N)\ge 4$, ${\rm gcd}(r_2, N)\ge 3$ and $N\ne 0 \ {\rm mod} \ 12$.

\smallskip 

 {\em Case S1: Either ${\rm gcd}(r_1, N)=1$ or ${\rm gcd}(r_2, N)=1$. }
 
  By the construction of integers $k_1$ and $k_2$, we see that
  $(k_1-k_2) {\rm gcd}(r_1, N)= 0  \ {\rm mod} \ N$ and
 $(k_1+k_2) {\rm gcd}(r_2, N)= 0  \ {\rm mod} \ N$, which
 is contradiction as  $k_1-k_2\in [-N/2+1, N/2-1]\backslash \{0\}$
 and $2\le k_1+k_2\le N-2$.

\smallskip

{\em Case S2: ${\rm gcd}(r_1, N)=2$, ${\rm gcd}(r_2, N)\ge 3$ and $N\ne 0 \ {\rm mod} \ 4$.}

In this case,  we obtain from the construction of integers that
$2(k_1-k_2)=0 \  {\rm mod} \ N$, which is a contradiction as
$k_1-k_2\in [-N/2+1, N/2-1]\backslash \{0\}$.

\smallskip 

{\em Case S3: ${\rm gcd}(r_1, N)\ge 3$, ${\rm gcd}(r_2, N)=2$ and $N\ne 0 \ {\rm mod} \ 4$.}

In this case,  we obtain from the construction of integers that
\begin{equation}\label{case10.eq1}
    {\rm gcd}(r_1, N)(k_1-k_2)=0 \  {\rm mod} \ N\   \ {\rm and}\ \ 2(k_1+k_2)=0 \ {\rm mod} \ N.\end{equation}
Therefore $k_1+k_2=N/2$ as  $2\le k_1+k_2\le N-2$.
Substituting the above expression into \eqref{case10.eq1}, we have
$$  {\rm gcd}(r_1, N) \frac{N}{2}- 2 {\rm gcd}(r_1, N) k_2= 0 \ {\rm mod} \ N,$$
which is a contradiction as the left hand side of the above equality is an odd integer and
the right hand side is an even integer. 

\smallskip 

{\em Case S4:
${\rm gcd}(r_1, N)=3$, ${\rm gcd}(r_2, N)\ge 4$ and $N\ne 0 \ {\rm mod} \ 12$.}

In this case,  we obtain from the construction of integers that
\begin{equation}\label{case12.eq1}
    3(k_1-k_2)=0 \  {\rm mod} \ N\   \ {\rm and}\ \ 4(k_1+k_2)=0 \ {\rm mod} \ N.\end{equation}
  By the assumption $k_1-k_2\in [-N/2+2, N/2-2]\backslash \{0\}$, we see that either  
  $k_1-k_2=N/3$ or $-N/3$.   Therefore either
  $4(N/3+2k_2)=0 \ {\rm mod}\  N$ or  $4(N/3+2k_1)=0 \ {\rm mod}\  N$.
  This together with  the range assumption on $k_1$ and $k_2$, we see that
  $4(N/3+2k_2)\in \{N, 2N, 3N\}$ or $4(N/3+2k_1)\in \{N, 2N, 3N\}$, which
  has no integer solution under the assumption that $N\ne 0 \ {\rm mod} \ 12$.  This leads to a contradiction. 

  \smallskip
  
{\em Case S5: ${\rm gcd}(r_1, N)\ge 4$, ${\rm gcd}(r_2, N)=3$ and $N\ne 0 \ {\rm mod} \ 12$. 
}

In this case,  we obtain from the construction of integers that
\begin{equation}\label{case13.eq1}
    4(k_1-k_2)=0 \  {\rm mod} \ N\   \ {\rm and}\ \ 3(k_1+k_2)=0 \ {\rm mod} \ N.\end{equation}
This together the range assumption on $k_1, k_2$ implies that
$4(k_1-k_2)=N$ or $-N$,  which is contradiction as 
the left hand size is a multiple of $4$ and the right hand size is
not a multiple of $4$.

Combining the above arguments for five different cases, we complete the proof of the sufficiency  in the conclusion (ii). \end{proof}

\section{Graph uncertainty principle}
\label{UP.subsection}

In this appendix,
we 
introduce a quantity $\alpha(Y, \Omega)$ for 
a nonempty set pair $(Y, \Omega)$
in the spatial-Fourier domain and show that $\alpha(Y, \Omega)<1$ is a sufficient condition for the  nonempty spatial-Fourier set pair $(Y, \Omega)$ to be  a strong annihilating pair;
see \eqref{uncertainty.thm.eq1}, Theorem \ref{uncertainty.thm} and Examples \ref{uncertainty.example2} and \ref{circulant.ex}.  As consequences of Theorem \ref{uncertainty.thm}, we obtain multiple formulations of uncertainty principles governing the (essential) supports of a nonzero graph signal and its GFT, see  Corollaries \ref{epsilonuncertainty.cor}, \ref{uncertainty.cor} and \ref{uncertainty.cor2}.

Let $Y\subset V$ and $ \Omega\subset\{1, 2, \ldots, M\}$ be nonempty sets. We say that  $(Y,\Omega)$
is a {\em strong annihilating pair} in the spatial-Fourier domain
if there exists a positive constant $C$ such that
\begin{equation} \label{sap.def}
   \|{\bf x}\|_2
   \le C \Big(\sum_{m\notin \Omega}\|\widehat {\bf x}(m)\|_2^2\Big)^{1/2} 
   + C    \Big(\sum_{v\notin Y}|x(v)|^2\Big)^{1/2} 
\end{equation}
hold for all graph signals ${\bf x}=[x(v)]_{v\in V}$ on the graph ${\mathcal G}$
\cite{demange2005, grochenig2025, jaming2007, Jaming25, ricaud2014}.
For $v\in V$, we denote the  delta signal
on the vertex $v$ by  ${\pmb \delta}_v=[\delta_{v v'}]_{v'\in V}$,
where $\delta$  is the Kronecker delta symbol.
In the following theorem, we show that any nonempty pair
$(Y, \Omega)$ in the spatial-Fourier domain satisfying
\begin{equation} \label{uncertainty.thm.eq1}
\alpha(Y,\Omega):=\Big(\sum_{m\in\Omega}\sum_{v\in Y}\big\|\widehat {{\pmb \delta}_v}(m)\big\|^2_2\Big)^{1/2}<1
\end{equation}
is a strong annihilating pair.

\begin{theorem}\label{uncertainty.thm} 
Let  $Y\subset V$ and $\Omega\subset \{1, 2, \ldots, M\}$  be nonempty sets satisfying  \eqref{uncertainty.thm.eq1}.
Then 
\begin{equation} \label{uncertainty.thm.eq2}
   \|{\bf x}\|_2
   \le (1-\alpha(Y,\Omega))^{-1} \Big(\sum_{m\notin \Omega}\|\widehat {\bf x}(m)\|_2^2\Big)^{1/2} 
   + \frac{(\dim \Omega-(\alpha(Y,\Omega))^2)^{1/2}}{1-\alpha(Y,\Omega) } 
   \Big(\sum_{v\notin Y}|x(v)|^2\Big)^{1/2}
\end{equation}
hold for all graph signals ${\bf x}=[x(v)]_{v\in V}$, where 
$\alpha(Y, \Omega)$ is the constant in \eqref{uncertainty.thm.eq1}, and
 $\dim \Omega $ is   dimension of the direct sum 
$\oplus_{m\in \Omega} W_m$ of the range spaces
$W_m, m\in \Omega$, in  \eqref{projection.eq4}.
\end{theorem}

We postpone the detailed proof to the end of this appendix.

Take  $\epsilon\in [0, 1]$ and let  $Y\subset V$ and $\Omega\subset \{1, \ldots, M\}$ be nonempty sets. We say that  a graph signal ${\bf x}=[x(v)]_{v\in V}$ is {\em $\epsilon$-concentrated on $Y$} if
$$\Big(\sum_{v\not\in Y} |x(v)|^2\Big)^{1/2}\le \epsilon \|{\bf x}\|_2,$$
and its Fourier transform $\widehat{\bf x}$ is 
{\em $\epsilon$-concentrated on $\Omega$} if 
$$\Big(\sum_{m\not\in \Omega} \|\widehat {\bf x} (m)\|_2^2\Big)^{1/2}\le \epsilon \|\widehat {\bf x}\|_F$$
\noindent \cite{bass2013, donoho1989, fuhr2019, li2022}.
By Theorem \ref{uncertainty.thm},  we have the following uncertainty principle of Donoho and Stark type, cf. \cite{donoho1989} for a similar result on the classical time-frequency  domain. 

\begin{corollary} \label{epsilonuncertainty.cor} Let
 $Y\subset V$ and $\Omega\subset \{1, \ldots, M\}$ be nonempty sets,
 and $\epsilon_T, \epsilon_F\in [0, 1]$.
If there exists a  nonzero graph signal ${\bf x}$
such that it is  $\epsilon_T$-concentrated
on $Y$  and  its Fourier transform $\widehat {\bf x}$ is $\epsilon_F$-concentrated on $\Omega$, then 
\begin{equation} \label{epsilonuncertainty.cor.eq1}
    \alpha(Y, \Omega)\ge 1- \epsilon_T- (\dim \Omega)^{1/2}\epsilon_F.
\end{equation}
\end{corollary}

Taking $\epsilon_T=\epsilon_F=0$ in  Corollary \ref{epsilonuncertainty.cor}, we conclude that
for  nonempty sets $Y\subset V$ and $\Omega\subset \{1, \ldots, M\}$ with 
$\alpha(Y, \Omega)<1$, there does not exist a nonzero graph signal  such that
it is supported on $Y$ and its Fourier transform 
is supported on $\Omega$. In other words, we have the following uncertainty principle governing  the supports of any nonzero graph signal and its GFT.

\begin{corollary} \label{uncertainty.cor} Let ${\bf x}=[x(v)]_{v\in V}$ be a nonzero graph signal. Then  
its support ${\rm supp}\ \x =\{v\in V  \mid x(v)\ne0\}$
and the support ${\rm supp} \ \widehat\x=\{m\in \{1, \ldots, M\} \mid \widehat\x(m)\ne {\bf 0}\}$
of its GFT $\widehat {\bf x}$
satisfy
\begin{equation} \label{uncertainty.cor.eq1}
    \alpha({\rm supp}\ \x,\  {\rm supp}\ \widehat\x)\ge 1.
\end{equation}
\end{corollary}

A similar result to the conclusion in Corollary 
\ref{uncertainty.cor} is established in \cite{Stankovic20}
with the GFT being defined by \eqref{GFT.distinct}.
As shown in  Examples \ref{uncertainty.example2} and \ref{circulant.ex} below, the lower bound estimate in 
Corollary \ref{uncertainty.cor} is optimal on the complete graph $K_N, N\ge 1$, and  on the circulant graph ${\mathcal C}(N, Q)$.

\begin{example} \label{uncertainty.example2}
{\rm (Continuation of Examples \ref{uncertainty.example}, \ref{uncertainty.example.part2} and \ref{uncertainty.example.part3})\  On the complete graph $K_N$, we obtain from 
\eqref{uncertainty.example.eq1} that 
\begin{equation} \label{uncertainty.example2.eq1}
\alpha(Y, \Omega)= N^{-1/2} (\# Y \times \dim \Omega)^{1/2}=\left\{ 
\begin{array}{ll} (N^{-1} \# Y)^{1/2} & {\rm if} \ \Omega=\{1\}\\
((1-1/N) \# Y)^{1/2} & {\rm if} \ \Omega=\{2\}\\
 (\# Y)^{1/2} & {\rm if} \ \Omega=\{1, 2\},
\end{array}
\right.
\end{equation}
where $\# Y$ is  the cardinality of the set
$Y$.  
Therefore 
the requirement \eqref{uncertainty.thm.eq1}
is satisfied only when either $Y\subsetneq K_N$ and $\Omega=\{1\}$, or
$\# Y=1$ and $\Omega=\{2\}$. 
Also we see that  the lower bound estimate \eqref{uncertainty.cor.eq1}
in Corollary
\ref{uncertainty.cor} could be achieved on the complete graph $K_N$, since 
$\alpha({\rm supp}\ \x,\  {\rm supp}\ \widehat\x)=\alpha (K_N, \{1\})=1$
for the constant graph signal ${\bf x}={\bf 1}$. 
}\end{example}

\begin{example} \label{circulant.ex}  {\rm 
Consider the unweighted  circulant graph
${\mathcal C}:={\mathcal C}(N, Q)$
generated by $Q=\{q_1, \ldots, q_d\}$ in Appendix \ref{circulantgraph.section}, and take 
graph shifts ${\bf S}_l, 1\le l\le d$, as in 
\eqref{circulantgraph.example.eq1}.
For the delta signal ${\bf \delta}_v$ at the vertex $v\in V_N:=\{1, \ldots, N\}$, 
    we have 
\begin{equation} \label{circulant.uncertainty.eq1}
     \|\widehat{\pmb \delta}_v(1)\|^2_2=N^{-1}= N^{-1}{\dim W_1},
\end{equation}
and for $2\le m\le M$,
\begin{eqnarray} \label{circulant.uncertainty.eq2}
     \|\widehat{\pmb \delta}_v(m)\|^2_2&=&\sum_{\pmb\lambda(2k)=\pmb\mu(m)}
     \Big(2N^{-1}\cos^2 \frac{2(v-1)k\pi}{N} \|{\bf u}_k\|_2^2+ 2N^{-1}
     \sin^2 \frac{2(v-1)k\pi}{N} \|{\bf v}_k\|_2^2\Big)\nonumber\\
    & = &  N^{-1} {\rm dim}W_m, 
\end{eqnarray}
provided that the order $N$ is odd. 
Similarly  we can show that for even order $N$, 
\begin{equation} \label{circulant.uncertainty.eq3}
 \|\widehat{\pmb \delta}_v(m)\|^2_2=  N^{-1} {{\rm dim}W_m}, \ 1\le m\le M.
\end{equation}  
Therefore  the quantity 
 $\alpha(Y, \Omega)$ in \eqref{uncertainty.thm.eq1}
admits the following explicit form in the unweighted circulant graph setting,
\begin{equation}\label{circulant.uncertainty.eq4}
\alpha(Y, \Omega)=\big(N^{-1}{{\rm dim}\Omega \times  \#Y}\big)^{1/2}.
\end{equation}
As $\alpha({\rm supp}\ \x,\  {\rm supp}\ \widehat\x)=\alpha ({v}, \{1,2,\ldots, M\})=1$
for the delta graph signal ${\pmb \delta}_v,v\in V_N$, 
we see that  the lower bound estimate  in 
\eqref{uncertainty.cor.eq1}
could be achieved on the circulant graph ${\mathcal C}_N$. 
}
\end{example}

For the family of orthogonal projections ${\bf P}=\{{\bf P}_m, \ 1\le m\le M\}$ in \eqref{projection.eq0}, define 
\begin{equation}
    \label{Puncertainty.def}
\|{\bf P}\|_*:=\sup_{Y\subset V,\ \Omega\subset \{1, \ldots, M\}}\ \big\{
(\# Y \times  \dim \Omega)^{-1/2} \ \mid\ 
\alpha(Y, \Omega)\ge 1\big\},\end{equation}
\begin{equation}
    \label{Puncertainty.def2}
\|{\bf P}\|_{**}:=\sup_{Y\subset V,\ \Omega\subset \{1, \ldots, M\}}\ \big\{
(\# Y \times  \# \Omega)^{-1/2} \ \mid\ 
\alpha(Y, \Omega)\ge 1\big\},\end{equation}
and {\em graph Fourier coherence} by
$$\|{\bf P}\|_\infty :=\sup_{v\in V,\ 1\le m\le M}
\big\|\widehat {\pmb \delta}_v(m)\big\|_2.$$ 
Then
\begin{equation} \label{Puncertainty.eq1}
N^{-1/2}\le \|{\bf P}\|_*,
\ \ M^{-1/2}\le \|{\bf P}\|_{**} \ \ {\rm and} \ \ \|{\bf P}\|_*\le    \|{\bf P}\|_{**}\le  \|{\bf P}\|_\infty,
\end{equation}
where the first two inequalities follow from the observation that
$\sum_{m=1}^M \big\|\widehat {{\pmb \delta}_v}(m)\big\|_2^2=1$ for all $v\in V$, and the third one holds as 
$$\alpha(Y, \Omega)\le \|{\bf P}\|_\infty (\#Y\#\Omega)^{1/2}\le 
 \|{\bf P}\|_\infty (\#Y\times \dim\Omega)^{1/2}$$ hold for all $Y\subset V$ and $\Omega\subset \{1, \ldots, M\}$.
 On the complete graph $K_N$,
  we obtain from \eqref {uncertainty.example2.eq1} that 
$$\|{\bf P}\|_*=N^{-1/2},\  \|{\bf P}\|_{**}=M^{-1/2} \  {\rm  and} \ \|{\bf P}\|_\infty= (1-1/N)^{1/2}$$
 for the family of projections ${\bf P}$  containing ${\bf P}_1$ and ${\bf P}_2$ in \eqref{uncertainty.example.eq1}, and hence the lower bound estimate in \eqref{Puncertainty.eq1} for $\|{\bf P}\|_*$ 
 and  $\|{\bf P}\|_{**}$  are achieved on the complete graph $K_N$.  On the circulant graph ${\mathcal C}(N, Q)$, it follows from 
\eqref{circulant.uncertainty.eq1}, \eqref{circulant.uncertainty.eq2}
and \eqref{circulant.uncertainty.eq3} that
$$\|{\bf P}\|_*=N^{-1/2} \ \ {\rm and} \ \ \|{\bf P}\|_{**}=\|{\bf P}\|_\infty =(N^{-1} \max_{1\le m\le M} \dim W_m)^{1/2}$$
provided that $N/(\max_{1\le m\le M} \dim W_m)$ is an integer,
where $W_m, 1\le m\le M$ are the range space in \eqref{rangecirculant.def}. In this  case, 
 the upper bound 
in \eqref{Puncertainty.eq1} for $\|{\bf P}\|_{**}$ could be attained.

 \smallskip
 
By Corollary \ref{uncertainty.cor}, we have
the following lower bound estimate for supports of a nonzero graph signal and its Fourier transform.

\begin{corollary} \label{uncertainty.cor2} 
Let ${\bf P}=\{{\bf P}_m, \ 1\le m\le M\}$ be the family of projections in \eqref{projection.eq0} and define  $\|{\bf P}\|_*$  as
in \eqref{Puncertainty.def}.  Then 
    \begin{equation}\label{uncertain.support}
    \#({\rm supp}\ \x)\times \dim ({\rm supp}\ \widehat\x)\ge (\|{\bf P}\|_*)^{-2}
\end{equation}
hold for all nonzero  graph signals $\x$ on the graph $\G$.
\end{corollary}

We remark that
the  lower bound  in \eqref{uncertain.support} is achieved
for delta signals on  complete graphs 
by \eqref{uncertainty.example.eq2},
and also for delta signals on circulant graphs 
by \eqref{circulant.uncertainty.eq1}, \eqref{circulant.uncertainty.eq2}, \eqref{circulant.uncertainty.eq3} and \eqref{circulant.uncertainty.eq4}. 
Similar low bound estimates 
on supports of a nonzero graph signal and its Fourier transform
to the one in Corollary \ref{uncertainty.cor2} can be found in \cite{agaskar2012, elad2002, pasdeloup2019,  Perraudin2018, Rebrova23, Stankovic20, Teke17, tsitvero2016}
with the GFT being defined by \eqref{GFT.distinct}
in the single graph shift setting, i.e., $d=1$.  

We finish this appendix with the proof of Theorem \ref{uncertainty.thm}.

\begin{proof}
[Proof of Theorem \ref{uncertainty.thm}]
Take a graph signal ${\bf x}=[x(v)]_{v\in V}$ on the graph ${\mathcal G}$. 
Then  we have
\begin{eqnarray*}  \|{\bf x}\|_2 & \hskip-0.08in = & \hskip-0.08in  \|\widehat\x\|_F \le 
\Big(\sum_{m\not\in\Omega}\|\widehat \x(m)\|_2^2\Big)^{1/2}+
\Big(\sum_{m\in\Omega}\|\widehat \x(m)\|_2^2\Big)^{1/2} 
\nonumber\\
& \hskip-0.08in \le  & \hskip-0.08in \Big(\sum_{m\not\in\Omega} \|\widehat \x(m)\|_2^2\Big)^{1/2}+
\Big(\sum_{m\in\Omega}\Big( \sum_{v\in Y} |x(v)| \big \|\widehat {{\pmb \delta}_v}(m)\big\|_2+ \sum_{v\not\in Y} |x(v)|  \big \|\widehat {{\pmb \delta}_v}(m)\big\|_2\Big)^2\Big)^{1/2}\nonumber\\
    & \hskip-0.08in\le  & \hskip-0.08in
    \Big(\sum_{m\not\in\Omega} \|\widehat \x(m)\|_2^2\Big)^{1/2}
  +
\Big\{\sum_{m\in\Omega}\Big( \|{\bf x}\|_2 \Big(\sum_{v\in Y} 
\big \|\widehat {{\pmb \delta}_v}(m)\big\|_2^2\Big)^{1/2}\nonumber\\
& & \qquad + \Big(\sum_{v\not\in Y} |x(v)|^2\Big)^{1/2} \Big(\sum_{v\not\in Y} \big \|\widehat {{\pmb \delta}_v}(m)\big\|_2^2\Big)^{1/2}\Big)^2\Big\}^{1/2}\nonumber\\
     & \hskip-0.08in\le  & \hskip-0.08in 
     \Big(\sum_{m\not\in\Omega} \|\widehat \x(m)\|_2^2\Big)^{1/2}
+     \|{\bf x}\|_2   \Big(\sum_{m\in \Omega} \sum_{v\in Y} 
\big \|\widehat {{\pmb \delta}_v}(m)\big\|_2^2\Big)^{1/2}\nonumber\\
& & + \Big(\sum_{v\not\in Y}  |x(v)|^2\Big)^{1/2}
\Big (\sum_{m\in\Omega} \sum_{v\not\in Y}  \big \|\widehat {{\pmb \delta}_v}(m)\big\|_2^2\Big)^{1/2}\\
 & \hskip-0.08in\le  & \hskip-0.08in \alpha(Y, \Omega)\|{\bf x}\|_2+ \Big(\sum_{m\not\in\Omega} \|\widehat \x(m)\|_2^2\Big)^{1/2}   
+ \big(\dim \Omega-  (\alpha(Y, \Omega))^2\big)^{1/2} \Big(\sum_{v\not\in Y}  |x(v)|^2\Big)^{1/2},
\end{eqnarray*}
where the equality is true by the Parseval identity \eqref{parsevelidentity}, the first, third and fourth inequalities are obtained by applying some triangle inequalities, the second equality holds by the representation $\widehat \x=\sum_{v\in V}x(v) \widehat {\pmb \delta}_v$, and the last inequality
follows from  \eqref{uncertainty.thm.eq1} and the observation
$$\sum_{m\in\Omega} \sum_{v\not\in Y} \big \|\widehat {{\pmb \delta}_v}(m)\big\|_2^2
  = \sum_{m\in \Omega} \sum_{v\in V}  \big \|\widehat {{\pmb \delta}_v}(m)\big\|_2^2-
\sum_{m\in \Omega} \sum_{v\in Y}  \big \|\widehat {{\pmb \delta}_v}(m)\big\|_2^2 =   \dim (\Omega)-  (\alpha(Y, \Omega))^2.$$
Subtracting $\alpha(Y, \Omega)\|{\bf x}\|_2$ and then dividing
$1-\alpha(Y, \Omega)$ at both sides of the established inequality yields the desired estimate \eqref{uncertainty.thm.eq2}.
\end{proof}

\end{appendix}

{\bf Acknowledgement}:\  
The authors acknowledge the assistance of Microsoft Copilot to provide phrasing refinements and typesetting suggestions.

\end{document}